\documentclass[12pt]{amsart}
\usepackage{pb-diagram,amssymb,epic,eepic,verbatim,graphicx}
\usepackage{graphics,epsfig,psfrag}

\newtheorem{theorem}{Theorem}[subsection]

\newtheorem{corollary}[theorem]{Corollary}
\newtheorem{question}[theorem]{Question}

\newtheorem{proposition}[theorem]{Proposition}
\newtheorem{lemma}[theorem]{Lemma}
\newtheorem{lem}[theorem]{}
\theoremstyle{definition}
\newtheorem{definition}[theorem]{Definition}
\theoremstyle{remark}
\newtheorem{remark}[theorem]{Remark}
\newtheorem{example}[theorem]{Example}
\newcommand{\blem}{\begin{lem} \rm}
\newcommand{\elem}{\end{lem}}

\newcommand\D{\mathcal{D}}

\renewcommand\S{\on{Sym}}

\renewcommand{\O}{\mathcal{O}}

\newcommand{\N}{\mathbb{N}}
\newcommand{\R}{\mathbb{R}}
\renewcommand{\H}{\mathbb{H}}
\newcommand{\RR}{\mathcal{R}}
\newcommand{\C}{\mathbb{C}}

\newcommand{\Z}{\mathbb{Z}}
\newcommand{\Q}{\mathbb{Q}}

\newcommand{\ddt}{\frac{d}{dt}}

\newcommand{\dds}{\frac{d}{ds}}

\newcommand{\ppx}{\frac{\partial}{\partial x}}
\newcommand{\ppy}{\frac{\partial}{\partial y}}
\newcommand{\ppz}{\frac{\partial}{\partial z}}

\renewcommand{\P}{\mathbb{P}}

\newcommand\lie[1]{\mathfrak{#1}}

\renewcommand{\k}{\lie{k}}
\renewcommand{\l}{\lie{l}}
\newcommand{\h}{\lie{h}}
\newcommand{\g}{\lie{g}}

\newcommand{\p}{\lie{p}}

\renewcommand{\t}{\lie{t}}

\renewcommand{\b}{\lie{b}}

\newcommand{\so}{\lie{so}}
\renewcommand{\o}{\lie{o}}
\renewcommand{\sp}{\lie{sp}}
\renewcommand{\sl}{\lie{sl}}
\newcommand{\on}{\operatorname}

\newcommand{\dist}{\on{dist}}

\newcommand{\dual}{\vee}

\newcommand{\graph}{\on{graph}}

\newcommand{\Symp}{\on{Symp}}
\newcommand{\Diff}{\on{Diff}}

\newcommand{\ann}{\on{ann}}

\newcommand{\Sp}{\on{Sp}}

\newcommand{\End}{\on{End}}
\newcommand{\s}{\on{ps}}

\newcommand{\Aut}{ \on{Aut} }

\newcommand{\Ad}{ \on{Ad} }

\newcommand{\Res}{\on{Res}}
\newcommand{\Hom}{ \on{Hom}}
\newcommand{\Ind}{ \on{Ind}}
\renewcommand{\ker}{ \on{ker}}

\newcommand{\diag}{  \on{diag}}
\newcommand{\Spec}{  \on{Spec}}

\newcommand{\codim}{\on{codim}}

\newcommand\dirac{/\kern-1.2ex\partial} 
\newcommand\qu{/\kern-.7ex/} 
\newcommand\lqu{\backslash \kern-.7ex \backslash} 
\newcommand\bs{\backslash}
\newcommand\dr{r_+ \kern-.7ex - \kern-.7ex r_-}
 
\newcommand{\labell}\label

\renewcommand{\d}{{\on{d}}}
\newcommand{\ol}{\overline}

\newcommand\Phinv{\Phi^{-1}}

\newcommand\eps{\epsilon}

\newcommand{\f}{\frac}
\newcommand{\lan}{\langle}
\newcommand{\ran}{\rangle}
\newcommand{\hh}{{\f{1}{2}}}

\newcommand\pt{\on{pt}}

\newcommand\cF{\mathcal{F}}

\renewcommand{\ss}{{\on{ss}}}
\renewcommand{\s}{{\on{s}}}
\newcommand{\ps}{{\on{ps}}}
\newcommand{\us}{{\on{us}}}
\newcommand{\hull}{{\on{hull}}}

\newcommand\Tr{\on{Tr}}
\newcommand\Der{\on{Der}}

\newcommand\curv{\on{curv}}

\newcommand\Gr{\on{Gr}}
\newcommand\Map{\on{Map}}
\newcommand\rank{\on{rank}}

\newcommand\Eul{\on{Eul}}

\newcommand\Vect{\on{Vect}}

\newcommand\mO{\mathcal{O}}

\renewcommand\H{\mathcal{H}}
\renewcommand\Im{\on{Im}}
\newcommand\Ker{\on{Ker}}
\newcommand\grad{\on{grad}}

\newcommand\bra[1]{ < \kern-.7ex {#1} \kern-.7ex >} 
\newcommand\bdefn{\begin{definition}}
\newcommand\edefn{\end{definition}}
\newcommand\bea{\begin{eqnarray*}}
\newcommand\eea{\end{eqnarray*}}
\newcommand\bcv{\left[ \begin{array}{r} }
\newcommand\ecv{\end{array} \right] }

\newcommand\bma{\left[ \begin{array}{l} }
\newcommand\ema{\end{array} \right]}
\newcommand\ben{\begin{enumerate}}
\newcommand\een{\end{enumerate}}
\newcommand\beq{\begin{equation}}
\newcommand\eeq{\end{equation}}
\newcommand\bex{\begin{example}}
\newcommand\bsj{\left\{ \begin{array}{rrr} }
\newcommand\esj{\end{array} \right\}}

\newcommand\eex{\end{example}}
\newcommand\crit{{\on{crit}}}

\newcommand\sx{*\kern-.5ex_X}

\def\mathunderaccent#1{\let\theaccent#1\mathpalette\putaccentunder}
\def\putaccentunder#1#2{\oalign{$#1#2$\crcr\hidewidth \vbox
to.2ex{\hbox{$#1\theaccent{}$}\vss}\hidewidth}}

\begin{document}

\title{Moment maps and geometric invariant theory}

\author{Chris Woodward, Rutgers University, New Brunswick}

\address{Mathematics-Hill Center,
Rutgers University, 110 Frelinghuysen Road, Piscataway, NJ 08854-8019,
U.S.A.}  \email{ctw@math.rutgers.edu}

\thanks{Partially supported by NSF
 grants  DMS060509 and DMS0904358}

\maketitle

\tableofcontents

\section{Introduction} 

These are expanded notes from a set of lectures given at the school
``Actions Hamiltoniennes: leurs invariants et classification'' at
Luminy in April 2009.  The topics center around the theorem of Kempf
and Ness \cite{ke:le}, which describes the equivalence between the
notion of quotient in geometric invariant theory introduced by Mumford
in the 1960's \cite{mu:ge}, and the notion of symplectic quotient
introduced by Meyer \cite{me:sy} and Marsden-Weinstein \cite{ma:re} in
the 1970's.  Infinite-dimensional generalizations of this
equivalence have played an increasingly important role in geometry,
starting with the theorem of Narasimhan and Seshadri \cite{ns:st}
connecting unitary structures on a bundle with holomorphic stability,
which by historical accident preceded the finite-dimensional theorem.

The proof of the Kempf-Ness theorem depends on the convexity of
certain {\em Kempf-Ness functions} whose minima are zeros of the
moment map.  The convexity also plays an important role in the
relation to geometric quantization discovered by Guillemin and
Sternberg \cite{gu:gea}: it corresponds to the fact that ``invariant
quantum states concentrate near zeros of the moment map''.  Roughly
speaking these notes were written as an exercise in ``just how far''
one can carry the convexity of the Kempf-Ness function.  For example,
using convexity I give alternative proofs of some of the results in
Kirwan's book \cite{ki:coh} as well as finite-dimensional versions of
Harder-Narasimhan and Jordan-H\"older filtrations; the former appears
in the algebraic literature under the name of Hesselink one-parameter
subgroups \cite{hess:uni} but the latter seems to have been
undeveloped.

The text is interspersed with applications to existence of invariants
in representation theory, such as the problem of determining the
existence of invariants in tensor products of irreducible
representations, and various techniques for computing moment
polytopes.  For example, the last section describes Teleman's improved
version of quantization commutes with reduction \cite{te:qu} which
also covers the behavior of the higher cohomology groups, and the
non-abelian localization formula which computes the difference between
the sheaf cohomology of the quotient and the invariant cohomology of
the action.  Some of the topics not treated are notably:
Duistermaat-Heckman theory, symplectic normal forms, localization
theorems in equivariant cohomology, and connections to classical
invariant theory, to name a few.

The author is grateful for comments and corrections by Michel Brion,
Gert Heckman, and Reyer Sjamaar, and apologizes for any omissions of
work in what has become a vast literature.

\section{Actions of Lie groups} 

To establish notation we review the basics of Lie group actions.

\subsection{Lie groups} 

A {\em Lie group} is a smooth manifold $K$ equipped with a group
structure so that group multiplication $K \times K \to K$ is a smooth
map.  The {\em Lie algebra} $\k$ is the space of left-invariant vector
fields on $K$, and may be identified with the tangent space of $K$ at
the identity $e \in K$.  The {\em exponential map} $\exp: \k \to K$ is
defined by evaluating the time-one flow at the identity.

Suppose that $K$ is compact and connected.  Let $T \subset K$ be a
maximal torus with Lie algebra $\t$.  We denote by $\Lambda :=
\exp^{-1}(e) \cap \t$ the integral lattice and by $\Lambda^\dual
\subset \t^\dual$ its dual, the weight lattice.  Any element $\mu \in
\Lambda^\dual$ defines a character $T \to U(1), t \mapsto t^\mu$ given
for $\xi \in \t$ by $\exp(\xi)^{\mu} := \exp( 2\pi i \mu(\xi))$.  The
{\em Weyl group} of $T$ is denoted $W = N(T)/T$.  The Lie algebra $\k$
splits under the action of $T$ into the direct sum of the Lie algebra
$\t$ and a finite sum of {\em root spaces} $\k_\alpha, \alpha \in
\RR(\k)$ where $\RR(\k) \subset \Lambda^\dual/ \{\pm 1 \}$ is the set
of roots and each $\k_\alpha$ is identified with a one
complex-dimensional representation on which $T$ acts by
$\exp(\xi)^{\alpha} := \exp( 2 \pi i \alpha(\xi))$.  The kernels
$\ker(\alpha)$ of the roots $\alpha \in \RR(\k)$ divide $\t$ into a
set of (open) {\em Weyl chambers}; given a generic linear function on
$\t$ there is a unique open {\em positive Weyl chamber} on which the
function is positive; we denote by $\t_+$ its closure and by
$\t_+^\dual \subset \t^\dual$ the image of $\t_+$ under an
identification $\t \to \t^\dual$ induced by an invariant metric on
$\k$.  Simple $K$-modules are classified by the set of {\em dominant
  weights} $\Lambda_+ := \Lambda \cap \t_+^\dual$.

\subsection{Smooth actions and quotients} 

Let $X$ be a smooth manifold.  A {\em (left) action} of $K$ on $X$ is
a smooth map $K \times X \to X, \ (k,x) \mapsto kx$ with the
properties that $k_0(k_1 x) = (k_0 k_1)x$ and $e x = x$ for all
$k_0,k_1 \in K$ and $x \in X$.  A {\em $K$-manifold} is a smooth
manifold equipped with a smooth $K$-action.  Let $X_0,X_1$ be
$K$-manifolds.  A smooth map $\varphi: X_0 \to X_1$ is {\em
  $K$-equivariant} if $\varphi(kx) = k\varphi(x)$ for all $k \in K, x
\in X_0$.

Both the Lie algebra and its dual are naturally $K$-manifolds: The
{\em adjoint} action of an element $k \in K$ on the Lie algebra $\k$
is denoted $\Ad(k) \in \End(\k)$.  The coadjoint action of $k$ on the
dual $\k^\dual$ is $\Ad^\dual(k) := (\Ad(k^{-1}))^\dual$.  The group
$K$ itself is a $K$-manifold in three different ways: the left action,
the (inverted) right action, and the {\em adjoint} action by
conjugation $\Ad(k_0)k_1 := k_0 k_1 k_0^{-1}$.  The exponential map
$\exp: \k \to K$ is equivariant with respect to the adjoint action on
$\k$ and $K$.  If $K$ is compact, then the dual $\t^\dual$ of the Lie
algebra $\t$ of the maximal torus $T$ admits a canonical embedding in
$\k^\dual$, whose image is the $T$-fixed point set for the coadjoint
action of $T$ on $\k^\dual$, and so $\k^\dual$ admits a canonical
projection onto $\t^\dual$.  

Let $X$ be $K$-manifold.  Let $\Diff(X)$ denote the
infinite-dimensional group of diffeomorphisms of $X$ and $\Vect(X)$
the Lie algebra of vector fields on $X$.  The $K$-action induces a
canonical group homomorphism
$$ K \to \Diff(X), \ \ k \mapsto k_X, \ k_X(x) = kx $$
and a Lie algebra homomorphism
$$ \k \to \Vect(X), \ \ \xi \mapsto \xi_X, \ \xi_X(x) = \ddt_{t=0}
\exp(- t \xi) x .$$
The sign here arises because the Lie bracket is defined using
left-invariant vector fields which are the generating vector fields
for the {\em right} action of the group on itself, whereas our actions
are by default from the {\em left}.  The {\em orbit} of a point $x \in
X$ is the set $Kx := \{ kx | k \in K \} \subset X$.  The {\em
  stabilizer} of a point $x \in X$ is $K_x := \{ k \in K | kx = x \}$;
its Lie algebra is the set $\k_x := \{ \xi \in \k \, | \, \xi_X(x) = 0
\}$.  A (co)adjoint orbit is an orbit of the (co)adjoint action of $K$
on $\k$ resp. $\k^\dual$.

Let $\psi: K_0 \to K_1$ be a homomorphism of Lie groups and let $X$ be
a $K_1$-manifold.  The action of $K_1$ and the homomorphism $\psi$
induce a $K_0$-action on $X$ by $k_0x := \psi(k_0)x$.  The orbits of
the $K_0$ action are those of the $K_1$-action, while the stabilizers
$(K_0)_x = \psi^{-1}((K_1)_x)$ are inverse images under $\psi$.

Let $X$ be a $K$-manifold. A {\em slice} at $x$ is a $K_x$-invariant
submanifold $V \subset X$ containing $x$ such that $KV$ is open in $X$
and the natural smooth $K$-equivariant map $K \times_{K_x} V \to KV$
is a diffeomorphism onto its image.  It follows from the existence of
geodesic flows etc.  that actions of {\em compact} groups have slices.
A {\em quotient} of a $K$-space is a pair $(Y,\pi)$ consisting of a
space $Y$ and a $K$-invariant morphism $\pi: X \to Y$ such that any
other $K$-invariant morphism factors through $\pi$.  The existence of
slices implies that any free action of a compact group $K$ on a
manifold $X$ has a manifold quotient $X / K$; more generally if the
action is not free then the quotient exists in the category of
Hausdorff topological spaces.  (Strictly speaking one should write the
quotient on the left, since our actions are by convention left
actions.  However, I find this rather cumbersome since in English $X /
K$ reads ``the quotient of $X$ by $K$'').

\subsection{Equivariant differential forms} 

Recall that a {\em graded derivation} of a graded algebra $A$ of 
degree $d$ is an
operator $D \in \End(A)_{d}$ such that $D(a_0 a_1) = D(a_0) a_1 +
(-1)^{d |a_0|} a_0 D(a_1)$ for homogeneous elements $a_0,a_1 \in A$.
The space of graded derivations $\Der(A)$ (direct sum over degrees)
forms a {\em graded Lie algebra} with bracket given by the {\em graded
  commutator}: given graded derivations $D_0,D_1$ of degrees $|D_0|,
|D_1|$, define $\{ D_0, D_1 \} = D_0 D_1 - (-1)^{|D_0||D_1|} D_1 D_0$.

Let $X$ be a smooth manifold of dimension $n$.  We denote by
$\Vect(X)$ the Lie algebra of smooth vector fields on $X$, and by
$\Omega(X) = \bigoplus_{j=0}^{n} \Omega^j(X)$ the graded algebra of
smooth forms on $X$.  For any $v \in \Vect(X)$ we have the derivations
defined by {\em contraction} $ \iota_v : \Omega^j(X) \to
\Omega^{j-1}(X) $ and {\em Lie derivative} $ L_v : \Omega^j(X) \to
\Omega^{j+1}(X) .$
Let $\d$ denote the de Rham operator, the graded derivation
$ \d: \Omega^j(X) \to \Omega^{j+1}(X) $
such that $\d f (v) = L_v f, \d \d f = 0 $ for $f \in \Omega^0(X), v
\in \Vect(X)$.  The operators $\iota_v, L_v, \d$ generate a finite
dimensional graded Lie algebra of $\Der(\Omega(X))$ with graded
commutation relations for $v,w \in \Vect(X)$ given by 
$$ \begin{array}{c|ccc} 
\{ \ , \ \} &  \iota_v & L_v & \d  \\
\hline 
\iota_w     &  0       & \iota_{[v,w]}   & L_w \\
L_w         &  \iota_{[w,v]} & L_{[w,v]}   &  0 \\ 
\d          &  L_v     &  0            & 0   
\end{array} .$$
It suffices to check the commutation relations by verifying them on
generators $f \in \Omega^0(X), \d g \in \Omega^1(X)$ of $\Omega(X)$.
We denote by $Z^j(X)$ the space of {\em closed forms} $ Z^j(X) = \{
\alpha \in \Omega^j(X) | \d \alpha = 0 \} $ by $B^j(X) = \{ \alpha \in
\Omega^j(X) | \exists \beta \in \Omega^{j-1}(X), \d \beta = \alpha \}
$ the space of {\em exact forms} and by $H^j(X)$ the {\em de Rham
  cohomology} $H^j(X) = Z^j(X)/B^j(X) .$

Suppose that $X$ admits a smooth action of a Lie group $K$.  Cartan
(see \cite{gu:eqdr}) introduced a space $\Omega_K(X)$ of {\em
  $K$-equivariant forms}
$$ \Omega_K^j(X) = \bigoplus_{2a + b = j} \Hom^a(\k, \Omega^b(X))^K,
\quad \Omega_K(X) = \bigoplus_{j=0}^\infty \Omega_K^j(X) $$
where $\Hom^a( \cdot)^K$ denotes equivariant polynomial maps of
homogeneous degree $a$.  The {\em equivariant de Rham operator} is
defined by
$$ \d_K: \Omega_K^j(X) \to \Omega_K^{j+1}(X), \quad
(\d_K(\alpha))(\xi) = (\d + \iota_{\xi_X}) (\alpha (\xi)) .$$
%
%
%
Let $Z_K^j(X)$ resp. $B_K^j$ denote the equivariant closed
resp. exact forms.  The {\em equivariant de Rham cohomology} is 
$$ H_K^j(X) = Z_K^j(X)/B_K^j(X), \quad H_K(X) = \bigoplus_{j=0}^\infty
H_K^j(X) .$$
If $K$ action is free, $H_K(X)$ is isomorphic to
the cohomology $H(X/K)$ of the quotient, see for example
\cite{gu:eqdr}.  

\section{Hamiltonian group actions}  

This section contains a quick review of equivariant symplectic
geometry.  More detailed treatments can be found in Cannas
\cite{cannas:intro}, Guillemin-Sternberg \cite{gu:sy}, Abraham-Marsden
\cite{ab:fo}, or Delzant's lectures in this volume.

\subsection{Symplectic manifolds} 

Let $X$ be a smooth manifold.  A {\em symplectic form} on $X$ is a
closed non-degenerate two-form $\omega \in \Omega^2(X)$.  A {\em
  symplectic manifold} is a manifold equipped with a symplectic
two-form.  A {\em symplectomorphism} of symplectic manifolds
$(X_0,\omega_0),(X_1,\omega_1)$ is a diffeomorphism $\varphi: X_0 \to
X_1$ with $\varphi^* \omega_1 = \omega_0$.  The term {\em symplectic}
is the Greek translation of the Latin word {\em complex}, and was used
by Weyl to distinguish the classical groups of linear
symplectomorphisms resp. complex linear transformations.

The simplest example of a symplectic manifold is $ \R^{2n}$ equipped
with the standard two-form $\sum_{j=1}^n \d q_j \wedge \d p_j $;
Darboux's theorem says that any symplectic manifold is locally
symplectomorphic to $\R^{2n}$ equipped with the standard form.  There
are simple cohomological restrictions on which manifolds admit
symplectic forms: Suppose that $X$ has dimension $2n$.  Non-degeneracy
of a two-form $\omega \in \Omega^2(X)$ is equivalent to the
non-vanishing of the highest wedge power $\omega^n \in
\Omega^{2n}(X)$; if $X$ is compact and $\omega$ is symplectic then the
cohomology class $[\omega^n] = [\omega]^n$ must be non-zero, since its
integral is non-vanishing, which implies that the classes $[\omega],
[\omega]^2,\ldots, [\omega]^{n-1}$ are also non-vanishing.  For
example this argument rules out the existence of symplectic structures
on spheres except for the two-sphere, where any area form gives a
symplectic structure.

Symplectic manifolds provide a natural framework for Hamiltonian
dynamics as follows.  For any symplectic manifold $(X,\omega)$ let
$\Symp(X,\omega) \subset \Diff(X)$ denote the group of
symplectomorphisms and $\Vect^s(X) \subset \Vect(X)$ the Lie
subalgebra of {\em symplectic vector fields} $v \in \Vect(X), L_v
\omega = 0$.  Any smooth function $H \in C^\infty(X)$ defines a
symplectic vector field $H^\# \in \Vect^s(X)$ by $ \iota_{H^\#} \omega
= \d H .$ In local Darboux coordinates, $H^\#$ is given by
$$ H^\# = 
\sum_{j=1}^n \frac{\partial H}{\partial p_j} \frac{\partial}{\partial q_j} - \frac{\partial H}{\partial q_j} \frac{\partial}{\partial p_j} .$$
The image of $C^\infty(X)$ in $\Vect^s(X)$ is the space $\Vect^h(X)$
of {\em Hamiltonian} vector fields.  Thus a vector field $v \in
\Vect(X)$ is symplectic resp. Hamiltonian iff the associated closed
one-form $\iota_v \omega$ is closed resp. exact.  The {\em Poisson
  bracket} is the Lie bracket on $C^\infty(X)$ defined by the formula
\begin{equation} \label{Poisson} \{ H_0,H_1 \} = \omega(H_0^\#, H_1^\#) .\end{equation}   
The map $H \mapsto -H^\#$ extends to an exact sequence of Lie algebras
$$0 \to H^0(X,\R) \to C^\infty(X) \to \Vect^s(X) \to H^1(X,\R) \to
0 $$
where the Lie bracket on the de Rham cohomology groups $H^0,
H^1(X,\R)$ is taken to be trivial.  A {\em Hamiltonian dynamical
  system} is a pair $(X,H)$ consisting of a symplectic manifold $X$
and an {\em energy function} $H \in C^\infty(X)$.  Time evolution is
given by the flow of $H^\# \in \Vect(X)$.  If $K \in C^\infty(X)$ is
another function, such as a component of angular momentum, then $\{ K,
H \} = - L_{K^\#} H = L_{H^\#} K$, so $H$ is invariant under the flow
generated by $K^\#$ iff $K$ is conserved in time.  This equivalence is
often called {\em Noether's theorem}: for every symmetry of a
Hamiltonian system there is a conserved quantity.

The cotangent bundle $T^\dual Q $ of a smooth manifold $Q$ possesses a
canonical symplectic structure: Let $\pi: T^\dual Q \to Q, (q,p) \to
q$ be the canonical projection.  The {\em canonical one-form} on
$T^{\dual} Q$ is
$$ \alpha \in \Omega^1(T^{\dual} Q), \quad \alpha_{(q,p)}(v) = p(
D\pi_{q,p}(v)) .$$
Local coordinates $q_1,\ldots, q_n$ on $Q$ induce dual coordinates
$p_1,\ldots, p_n$ in which $\alpha = \sum_{j=1}^n p_j \d q_j $.  It
follows that the {\em canonical two-form} $\omega$ on $T^{\dual}
Q$ given by $\omega = - \d \alpha$ is symplectic.  These forms are
canonical in the sense that any diffeomorphism $Q_0 \to Q_1$ induces
an isomorphism $T^{\dual} Q_0 \to T^{\dual} Q_1$ preserving the
canonical one-forms, which is therefore a symplectomorphism.
Physically $T^\dual Q$ represents the space of states of a classical
particle moving on a manifold $Q$.  However, many Hamiltonian
dynamical systems have symplectic manifolds that are not cotangent
bundles.  For example, the two-sphere is the natural symplectic
manifold for the study of the evolution of the angular momentum vector
of a rigid body.

\begin{proposition} 
The following are natural operations on symplectic manifolds:
\begin{enumerate} 
\item (Sums) Let $(X_0,\omega_0),(X_1,\omega_1)$ be symplectic
  manifolds.  Then the disjoint union $(X_0 \sqcup X_1, \omega_0
  \sqcup \omega_1)$ is a symplectic manifold.
\item (Products) Let $(X_j,\omega_j)$ be symplectic
  manifolds, $j=0,1$.  Then the product $X_0 \times X_1$ equipped with
  two-form $\pi_0^* \omega_0 + \pi_1^* \omega_1$ is a symplectic
  manifold, where $\pi_j: X_0 \times X_1 \to X_j, j = 0,1$ is the
  projection onto $X_j$.
\item (Duals) Let $(X,\omega)$ be a symplectic manifold.  Then the
  dual $(X,-\omega)$ (or more generally, $(X, \lambda \omega)$ for any
  non-zero $\lambda \in \R$) is a symplectic manifold.
\end{enumerate}
\end{proposition} 

Symplectomorphism is a very restrictive notion of morphism, since in
particular the symplectic manifolds must be the same dimension.  A
more flexible notion of morphism in the symplectic category is given
by the notion of {\em Lagrangian correspondence} \cite{we:sc}.  (The
discussion of correspondences is only used to formulate the universal
property for symplectic quotients; readers not interested in this can
skip all discussion of correspondences and the symplectic category.)
Let $(X,\omega)$ be a symplectic manifold.  A {\em Lagrangian
  submanifold} of $X$ is a submanifold $i: L \to X$ with $i^* \omega =
0$ and $\dim(L) = \dim(X)/2$. Let $(X_j,\omega_j),j = 0,1$ be
symplectic manifolds.  A {\em Lagrangian correspondence} from $X_0$ to
$X_1$ is a Lagrangian submanifold of $X_0^- \times X_1$.  Let $L_{01}
\subset X_0^- \times X_1$ and $L_{12} \subset X_1^- \times X_2$ be
Lagrangian correspondences.  Let $\pi_{02}$ denote the projection from
$X_0^- \times X_1 \times X_1^- \times X_2$.  Then
$$ L_{01} \circ L_{12} := \pi_{02}(L_{01} \times_{X_1} L_{12}) $$
is, if smooth and embedded, a Lagrangian correspondence in $X_0^-
\times X_2$ called the {\em composition} of $L_{01}$ and $L_{12}$.
%
The graph $\graph(\psi_{01})$ of any symplectomorphism $\psi_{01}$
from $X_0$ to $X_1$ is automatically a Lagrangian correspondence, and
if $\psi_{01},\psi_{12}$ are two such symplectomorphisms then
$\graph(\psi_{01} \circ \psi_{12}) = \graph(\psi_{01}) \circ
\graph(\psi_{12})$.
%
%
With this notion of composition, the pair (symplectic manifolds,
Lagrangian correspondences) becomes a partially defined category, with
identity given by the diagonal correspondence.  The partially defined
composition leads to an honest category, obtained by allowing
sequences of morphisms and identifying sequences if they are related
by geometric composition \cite{we:co}.  

Symplectic geometry can be considered a special case of Poisson
geometry: A {\em Poisson bracket} on a manifold $X$ is a Lie bracket
$\{ \ , \ \} : C^\infty(X) \times C^\infty(X) \to C^\infty(X)$ that is
a derivation with respect to multiplication of functions, that is, $\{
f, g h \} = \{ f, g\} h + g \{ f,h \}$.  A {\em Poisson manifold} is a
manifold equipped with a Poisson bracket.  A {\em morphism} of Poisson
manifolds is a smooth map $\psi: X_0 \to X_1$ such that $\{ \psi^* f,
\psi^* g \} = \psi^* \{ f,g \}$.
Given any Poisson bracket on a manifold $X$, for each $H \in
C^\infty(X)$ the derivation $\{ H, \ \}$ is equal to $L_{H^\#}$ for
some vector field $H^\#$.  The span of the vector fields $H^\#$
defines a decomposition of $X$ into {\em symplectic leaves}, each of
which is equipped with a symplectic structure so that \eqref{Poisson}
holds.  On the other hand, the notion of symplectic geometry as a
special case of Poisson geometry is not particularly compatible with
the idea that Lagrangian correspondences should serve as morphisms.

\subsection{Hamiltonian group actions} 

Let $K$ be a Lie group acting smoothly on a manifold $X$.  The action
is {\em symplectic} if it preserves the symplectic form, that is, $k_X
\in \Symp(X,\omega)$ for all $k \in K$, {\em infinitesimally
  symplectic} if $\xi_X \in \Vect^s(X)$ for all $\xi \in \k$, and {\em
  weakly Hamiltonian} if $\xi_X \in \Vect^h(X)$ for all $\xi \in \k$.
A {\em symplectic $K$-manifold} is a symplectic manifold equipped with
a symplectic action of $K$.

Let $(X,\omega)$ be a symplectic $K$-manifold.  The action is {\em
  Hamiltonian} if the map $\k \to \Vect(X), \ \xi \mapsto \xi_X$ lifts
to an equivariant map of Lie algebras $\k \to C^\infty(X)$.  Such a
map is called a {\em comoment map}.  A {\em moment map} is an
equivariant map $\Phi: X \to \k^\dual$, satisfying
\begin{equation} \label{mom}
\iota_{\xi_X} \omega = - d \lan \Phi, \xi \ran, \quad 
\forall  \xi \in \k \end{equation}
Any comoment map $\phi:\k \to C^\infty(X)$ defines a moment map by
$\lan \Phi(x), \xi \ran = (\phi(\xi))(x)$.

\begin{example} 
Let $K = V$ be a vector space acting on $X = T^\dual V$ by
translation. After identifying $\k \to V$ and so $\k^\dual \to
V^\dual$, a moment map for the action is given by the projection $X
\cong V \times V^\dual \to V^\dual, (q,p) = p$, that is, by the
ordinary momentum, hence the terminology {\em moment map}.
\end{example} 
\noindent The notion of moment map was introduced in independent work of
Kirillov, Kostant, and Souriau, in connection with geometric
quantization and representation theory.  See \cite{br:con} for a
discussion of the history of the moment map and the relationship of
the work between these authors.  Unfortunately there is no standard
sign convention for \eqref{mom}; our convention agrees with that of
Kirwan \cite{ki:coh}.  More generally, if $X$ is a smooth manifold
equipped with a closed two-form $\omega$ and an action of $K$ leaving
$\omega$ invariant, then we say that $\Phi$ is a moment map if
\eqref{mom} holds.

A {\em Hamiltonian resp. degenerate Hamiltonian $K$-manifold} is a
datum $(X,\omega,\Phi)$ consisting of a symplectic $K$-manifold
$(X,\omega)$ resp. smooth $K$-manifold $X$ equipped with an invariant
closed two-form $\omega$, and a moment map $\Phi$ for the action.  Let
$(X_0,\omega_0,\Phi_0)$ and $(X_1,\omega_1,\Phi_1)$ be Hamiltonian
$K$-manifolds.  An {\em isomorphism} of Hamiltonian $K$-manifolds is a
$K$-equivariant symplectomorphism $\varphi: (X_0,\omega_0) \to
(X_1,\omega_1)$ such that $\varphi^* \Phi_1 = \Phi_0$.

\begin{example} \label{archim}
Archimedes' computation of the area of the two-sphere is essentially a
moment map calculation.  Let $S^2 = \{ x^2 + y^2 + z^2 = 1 \}$ be the
unit sphere in $\R^3$.  Let $v = x \ppx + y \ppy + z \ppz \in
\Vect(\R^3)$.  The two-form
$\omega = \iota_v ( \d x \wedge \d y \wedge \d z ) = x \d y \wedge \d z -
y \d x \wedge \d z + z \d x \wedge \d y $
restricts to a symplectic form on $S^2$, invariant under rotation on
$\R^3$.  A moment map for the action of $S^1$ on $S^2$ by rotation
clockwise around the $z$-axis is given by $(x,y,z) \mapsto z$, under
the identification of the Lie algebra of $S^1$ and its dual with $\R$.
Indeed, the generating vector field for $\xi = 1$ is $\xi_X = - x
\frac{\partial}{\partial y} + y \frac{\partial}{\partial x}$.  A
computation shows that $\iota_{\xi_X} \omega = - dz .$

To relate this to Archimedes' area formula, note that if $r,\theta,z$
are cylindrical coordinates on $\R^3$, then $\iota_{
  \frac{\partial}{\partial \theta}} \omega = \d z $ and so $\omega =
\d z \wedge \d \theta$.  Thus the area of the unit two-sphere between
any two values $z_1,z_2 \in (-1,1)$ of $z$ is the same as the area of
the cylinder $S^1 \times [-1,1]$ between those two values, $2\pi (z_2
-z_1)$.  In particular (and this is the result reported by Cicero to
be inscribed on Archimedes' tombstone)
the area of the unit two-sphere $S^2$ is equal to the area of the
cylinder $S^1 \times [-1,1]$, namely $4 \pi$.
\begin{figure}[h]
\includegraphics[height=1in]{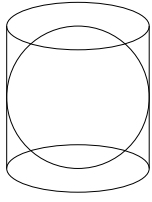}
\caption{$S^1 \times [-1,1]$ has the same area as $S^2$}
\end{figure}
\noindent 
We can deduce from the moment map for the circle action
the moment map for the full rotation group $SO(3)$ as follows. We
identify $\so(3) \to \R^3$ so that the infinitesimal rotation around
the $j$-th basis vector $e_j$ maps to $e_j$, and $\so(3)^\dual \to
\R^3$ using the standard metric on $\R^3$.
\label{S2}
The action of $SO(3)$ on $S^2$ has moment map the inclusion $S^2 \to
\R^3$.  Indeed, by symmetry, moment maps for the rotation around the
other two axes are given by $(x,y,z) \mapsto x$ resp. $y$.  Hence the
inclusion satisfies the equation \eqref{mom}.  In addition $\Phi$ is
equivariant and so defines a moment map.  This ends the example. 
\end{example} 

The following are natural operations on Hamiltonian $K$-manifolds:

\begin{proposition}   \label{hamops}
\begin{enumerate} 
\item (Sums) Let $(X_0,\omega_0,\Phi_0),(X_1,\omega_1,\Phi_1)$ be
  Hamiltonian $K$-manifolds.  Then the disjoint union $X_0 \sqcup X_1$
  is a Hamiltonian $K$-manifold, equipped with moment map $\Phi_0
  \sqcup \Phi_1$.
\item (Exterior Products) Let $(X_j,\omega_j,\Phi_j)$ be Hamiltonian
  $K_j$-manifolds, $j=0,1$.  Then the product $X_0 \times X_1$ is a
  Hamiltonian $K_0 \times K_1$-manifold, equipped with moment map
  $\pi_0^* \Phi_0 \times \pi_1^* \Phi_1$, where $\pi_j: X_0 \times X_1
  \to X_j, j = 0,1$ is the projection onto $X_j$.
\item (Duals) Let $(X,\omega,\Phi)$ be a Hamiltonian $K$-manifold.
  Then the dual $(X,-\omega,-\Phi)$ (or more generally, any rescaling
  by a non-zero constant) is a Hamiltonian $K$-manifold.
\item (Pull-backs) \label{restr} Let $\varphi: K_0 \to K_1$ be a homomorphism of Lie
  groups, and $(X,\omega,\Phi)$ a Hamiltonian $K_1$-manifold.  The Lie
  algebra homomorphism $D \varphi: \k_0 \to \k_1$ induces a dual map
  $D\varphi^\dual : \k_1^\dual \to \k_0^\dual$.  The action of $K_0$ induced by
  $\phi$ has moment map $D\varphi^\dual \circ \Phi$.
\item (Interior products) \label{intprods} Let $(X_j,\omega_j,\Phi_j)$ be Hamiltonian
  $K$-manifolds, $j=0,1$.  Then the product $X_0 \times X_1$ is a
  Hamiltonian $K$-manifold, equipped with moment map $\pi_0^* \Phi_0 +
  \pi_1^* \Phi_1$. This is a combination of the previous two items,
  using the diagonal embedding $\k \to \k \times \k$ whose adjoint is
  $\k^\dual \times \k^\dual \to \k^\dual, (\xi_0,\xi_1) \mapsto \xi_0 + \xi_1$.
\end{enumerate}
\end{proposition} 

More generally one can speak of Hamiltonian actions on Poisson
manifolds.  The dual $\k^\dual$ of the Lie algebra $\k$ has a
canonical {\em Lie-Poisson} bracket, $ C^\infty(\k^\dual) \times
C^\infty(\k^\dual) \to C^\infty(\k^\dual)$ with the property that $\{
\xi, \eta \} = [\xi,\eta]$ for $\xi, \eta \in \k$.  A {\em Poisson
  moment map} for a $K$-action on a Poisson manifold $X$ is a Poisson
map $\Phi: X \to \k^\dual$.  A {\em Hamiltonian-Poisson $K$-manifold} is a
Poisson $K$-manifold equipped with a Poisson moment map.   

\begin{proposition}  Any Hamiltonian $K$-manifold $(X,\omega,\Phi)$ 
is a Hamiltonian-Poisson $K$-manifold.
\end{proposition} 

\begin{proof}  For $\lambda,\xi \in \k$ we have 
$ \Phi^* \{ \lambda, \xi \} = \Phi^* [\lambda, \xi] = L_{\lambda_X}
  \Phi^* \xi = \{ \Phi^* \lambda, \Phi^* \xi \}$.  The case of
  non-linear functions is similar.
\end{proof}  

Conversely, any Poisson moment map induces an ordinary moment map on
its symplectic leaves.  In particular the coadjoint action is
Poisson-Hamiltonian with moment map the identity, and the symplectic
leaves are the coadjoint orbits.  Thus as observed by Kirillov,
Kostant, and Souriau,

\begin{proposition}  Any coadjoint orbit $K\lambda, \lambda \in \k^\dual$ of $K$
has the canonical structure of a Hamiltonian $K$-manifolds with moment
map given by the inclusion $K\lambda \to \k^\dual$.  
\end{proposition} 

\begin{example} Identify $\R^3 \cong \so(3) \cong \so(3)^\dual$. 
The Proposition gives Hamiltonian $SO(3)$-structures on the orbits of
$SO(3)$ on $\R^3$, which are either spheres (for non-zero radii
$\lambda$) or a point (if $\lambda =0$.)  This reproduces the 
moment map in Example \ref{S2}.
\end{example} 

\noindent For any {\em transitive} Hamiltonian action, the moment map
is a local diffeomorphism and so gives a covering of the coadjoint
orbit that is its image, see Kostant \cite{ko:qu}.

The Darboux theorem has various equivariant generalizations that we
will not discuss here; we only mention that as a consequence:

\begin{proposition} \label{morse} (see \cite{ki:coh}) Let
$X$ be a Hamiltonian $K$-manifold, $K$ compact.  For any $\xi \in \k$,
  the function $\lan \Phi, \xi \ran $ is a Morse function with even
  index.
\end{proposition} 

In the remainder of the section we explain two other ways in which
moment maps can be naturally interpreted.  The first is closely
related to the notion of {\em equivariantly closed differential form}
introduced in Section 2.3, see Atiyah and Bott \cite{at:mom}:

\begin{proposition}  Let $(X,\omega)$ be a symplectic $K$-manifold.  
There exists a one-to-one correspondence between moment maps for the
action of $K$, and equivariantly closed extensions of $\omega \in
\Omega^2(X)$ to $\Omega^2_K(X)$.
\end{proposition} 

\begin{proof}  Since $\Omega^2_K(X) \cong \Omega^2(X)^K \oplus 
\Hom(\k,\Omega^0(X))^K$ any extension in $\Omega^2_K(X)$ is equal to
$\omega + \Phi$ for some $\Phi \in \Map_K(X,\k^\dual) \cong
\Hom(\k,\Omega^0(X))^K$.  The extension if equivariantly closed iff $
0 = \d_K (\omega + \Phi) = (\d \omega, \iota_{\xi_X} \omega + \d \lan
\Phi, \xi \ran) .$ Since $\omega$ is by assumption closed, $\d_K
(\omega + \Phi) = 0$ iff $\Phi$ is a moment map.
\end{proof}  

The second interpretation of a moment map depends on the notion of
{\em linearization} of an action, as we now explain.  Suppose that $L
\to X$ is a Hermitian line bundle with unit circle bundle $L_1$ with
generating vector fields $\xi_L \in \Vect(L_1), \xi \in \R$.  The
circle group $U(1)$ acts on $L_1$ by scalar multiplication.  Let
$\alpha \in \Omega^1(L_1)^{U(1)}, \alpha(\xi_L) = \xi$ be a connection
one-form with curvature $(2\pi/i) \omega \in \Omega^2(X)$.  (That is, to fix
conventions, $\d \alpha = \pi^* \omega$ where $\pi: L_1 \to X$ is the
projection.)  The group $\Aut(L_1,\alpha)$ of unitary automorphisms of
$L$ preserving $\alpha$ naturally maps to the symplectomorphism group
$\Symp(X,\omega)$ of $X$, defining an exact sequence $ 1 \to U(1) \to
\Aut(L_1,\alpha) \to \Symp(X,\omega) .$ A {\em linearization} of the
action of $K$ on $X$ is a lift $K \to \Aut(L_1,\alpha)$.  An {\em
  infinitesimal linearization} is a lift $\k \to \Vect(L_1)^{U(1)}.$

\begin{proposition}  
Let $X$ be a $K$-manifold, $\omega \in \Omega^2(X)^K$ a closed
invariant two-form, and $\pi: L \to X$ a Hermitian line-bundle with
connection one-form $\alpha \in \Omega^1(L_1)^{K \times U(1)}$ whose
curvature is equal to $(2 \pi/i) \omega$.  The set of moment maps $\Phi$ for
the $K$-action is in one-to-one correspondence with the set of
infinitesimal linearizations of the action of $K$.  \end{proposition}

\begin{proof}  
 Let $\pi_1: L_1 \to X$ denote the projection. Given a lift $\k \to
 \Vect(L_1)^{U(1)}$, define a moment map $\Phi: X \to \k^\dual$ by $
 \lan \Phi(x),\xi \ran = (\alpha(\xi_L))(l)$, for any $l \in
 \pi^{-1}(x)$, independent of the choice of $l$.  Then
\begin{eqnarray*} \pi_1^* \d \lan \Phi,\xi \ran  &=& \d (\alpha (\xi_L))  =
\d \iota_{\xi_L} \alpha (l) 
= (L_{\xi_L} - \iota_{\xi_L} \d ) \alpha \\
&=& L_{\xi_L} \alpha  - \iota_{\xi_L} \pi_1^* \omega 
= - \pi_1^* \iota_{\xi_X} \omega .\end{eqnarray*}
Since $\alpha$ is invariant, $\Phi$ is equivariant, and so defines a
moment map.  Conversely, given a moment map define $\xi_L \in
\Vect(L_1)^{S^1}$ by $ \lan \Phi(x),\xi \ran =
(\alpha(\xi_L))(l)$. Then the same computation shows that $L_{\xi_L}
\alpha = 0$.  To see that $\xi \mapsto \xi_L$ defines a lift of $\k
\to \Vect^s(X,\omega)$ to $\Vect(L_1)^{U(1)}$, note that given
$\xi,\eta \in \k$, the vectors $[\xi,\eta]_L$ and $[\xi_L,\eta_L]$
agree up to a vertical vector field.  To see that they are equal, note
%
$ \alpha([\xi_L,\eta_L]) = [L_{\xi_L}, \iota_{\eta_L}] \alpha  =
 \pi^* L_{\xi_L} \lan \Phi, \eta \ran  = \pi^* \lan \Phi, [\xi,\eta]
 \ran  = \alpha([\xi,\eta]_L).
$
%
\end{proof}
\noindent The following is immediate from the definitions:

\begin{proposition} \label{weight} Suppose that  $\Phi$ is the moment map induced by a lift of the action to a
Hermitian line bundle with connection $L$.  Then $\exp(\xi), \xi \in
\k_x$ acts on the fiber $L_x$ via $l \mapsto \exp( i \lan \Phi(x), \xi
\ran) l$.
\end{proposition}  
\noindent In other words, the value of the moment map at a fixed point
determines the action of the identity component of the group on the
fiber over that point.

The notion of Lagrangian correspondence generalizes to Hamiltonian
actions as follows.  (again, readers not interested in universal
properties of quotients may skip this discussion):

\begin{definition}
Let $X$ be a Hamiltonian $K$-manifold with moment map $\Phi: X \to
\k^\dual$.  A {\em $K$-Lagrangian submanifold} is a $K$-invariant
Lagrangian submanifold on which $\Phi$ vanishes.  Let
$(X_j,\omega_j,\Phi_j)$ be Hamiltonian $K$-manifolds for $j = 0,1$. A
{\em $K$-Lagrangian correspondence} is a $K$-Lagrangian submanifold of
$X_0^- \times X_1$.
\end{definition}

Allowing sequences of $K$-Lagrangian correspondences
and identifying sequences related by a geometric composition gives an
honest category as in non-equivariant case.

\subsection{Symplectic quotients} 

Naturally one would like a notion of quotient of a Hamiltonian
$K$-manifold, which should be an object in the symplectic category and
satisfy a universal property for morphisms in the equivariant
symplectic category.  It is easy to see that the most naive
definition, of the actual quotient, is unsatisfactory for several
reasons. For example, even if the action is free, then the quotient
will not necessarily have even dimension, and so may not admit a
symplectic structure.  Also the action will not in general be free,
and so the quotient will not even have the structure of a manifold.

The construction of Meyer \cite{me:sy} and Marsden-Weinstein
\cite{ma:re} is free of these problems, at least under suitable
hypotheses: Let $(X,\omega,\Phi)$ be a Hamiltonian $K$-manifold with
moment map $\Phi: X \to \k^\dual$.  
Define the {\em symplectic quotient} 
$$ X \qu K := \Phi^{-1}(0)/ K .$$
\begin{theorem} [Meyer \cite{me:sy}, Marsden-Weinstein \cite{ma:re}]  
Let $X$ be a Hamiltonian $K$-manifold.  If $K$ acts freely and
properly on $\Phi^{-1}(0)$, then $X \qu K$ has the structure of a
smooth manifold of dimension $\dim(X) - 2 \dim(K)$ with a unique
symplectic form $\omega_0$ satisfying $i^* \omega = p^* \omega_0$,
where $i: \Phi^{-1}(0) \to X$ and $p: \Phi^{-1}(0) \to X \qu K$ are
the inclusion and projection respectively.  \end{theorem}

\noindent The double slash in the notation $X \qu K$ is meant to
reflect that the dimension drops by $2 \dim(K)$, in contrast to the
ordinary quotient $X/K$ for which dimension drops by $\dim(K)$, if the
action is free.  The proof depends on the following.  Let $\ann(\k_x)
\subset \k^\dual$ be the annihilator of $\k_x$.

\begin{lemma} Let $X$ be a Hamiltonian $K$-manifold.  For any $x \in X$,  \label{redlem}
\begin{enumerate}
\item $\Im D_x \Phi = \ann(\k_x)$.
\item $\Ker D_x \Phi = \{ \xi_X(x), \xi \in \k \}^{\omega_x}$.
\end{enumerate}
\end{lemma}

\begin{proof}  (a) 
We have $ \lan D_x \Phi(v), \eta \ran = \omega_x(v,\eta_X(x)) $ for $v
\in T_xX$ which vanishes for all $v \in T_xX$ iff $\eta_X(x) = 0$. (b)
The same identity shows $\omega_x(\xi_X(x),v) = 0 $ for $v \in \Ker
D_x \Phi$, so the left-hand-side of (b) is contained in the right.
Equality now follows by a dimension count, using (a).
\end{proof} 

\begin{proof}[Proof of Theorem]  By part (a) of the Lemma, 
the pull-back $i^* w$ vanishes on the orbits of $K$ and is
$K$-invariant, and so descends to a form $\omega_0$ on $X \qu K$. Part
(b) shows that $\omega_0$ is non-degenerate.  Since $ p^* \d \omega_0
= \d i^* \omega = i^* \d \omega = 0$, $\omega_0$ is closed, hence
symplectic.
\end{proof}  

\noindent The following is a fundamental example:

\begin{example} (Products of spheres) 
Let $\lambda_1,\ldots,\lambda_n$ be positive real numbers and $X =
S^2_{\lambda_1} \times \ldots \times S^2_{\lambda_n}$, where
$S^2_\lambda$ denotes the unit two-sphere with invariant area form
re-scaled by $\lambda$.  The group $K = SO(3)$ acts diagonally on $X=
(S^2)^n$ with moment map
$$ \Phi : X \to \k^\dual \cong \R^3, \quad (x_1,\ldots,x_n) \mapsto x_1 +
\ldots + x_n $$
by \ref{S2} and \ref{hamops} \eqref{intprods}.  The symplectic
quotient is the {\em moduli space of closed $n$-gons with lengths
  $\lambda_1,\ldots, \lambda_n$}
$$ X \qu SO(3) = \{ (x_1,\ldots,x_n) \in (\R^3)^n \ |\ \Vert x_j \Vert = \lambda_j,
\ x_1 + \ldots + x_n = 0 \} / SO(3) .$$
Its topology depends on the choice of $\lambda_1,\ldots,\lambda_n$,
see for example Hausmann-Knutson \cite{hk:co}.  In general there are a
finite number of ``chambers'' in which the topology of $X \qu SO(3)$
is constant.   The chambers in which $X \qu SO(3)$ is non-empty
are described by the following: 

\begin{proposition} \label{poly} $X \qu SO(3) \neq \emptyset$  
iff $\lambda_j \leq \sum_{i \neq j} \lambda_i$ for all $j = 1,\ldots,
n$.
\end{proposition} 

\begin{proof} For $n = 3$, these are the triangle inequalities.  
For $n > 3$, we assume without loss of generality that $\lambda_1 \ge
\ldots \ge \lambda_n$.  Then the inequalities above are equivalent to
the single inequality $\lambda_1 \leq \lambda_2 + \ldots + \lambda_n$.
One checks that there exists $j$ so that $|\lambda_2 + \ldots +
\lambda_j - \lambda_{j+1} - \ldots - \lambda_n| < \lambda_1 $.  The general case follows from that for $n = 3$, which implies that
there exists a triangle with side lengths $\lambda_1, \lambda_2 +
\ldots + \lambda_j, \lambda_{j+1} + \ldots + \lambda_n$.
\end{proof}  
\noindent This ends the example. \end{example} 

We end this section with two remarks on the definition of symplectic
quotient. First, the symplectic quotient of a Hamiltonian action can
be viewed as a symplectic leaf of the quotient of the corresponding
Hamiltonian-Poisson action in the following sense.  Suppose that $X$
is a Hamiltonian-Poisson $K$-manifold such that $K$ acts freely.  The
restriction of the Poisson bracket to $C^\infty(X)^K$ defines a
canonical Poisson structure on $X/K$.  Then $ X \qu K$ is a symplectic
leaf on the smooth locus in $X/K$ \cite{ar:un}; the other leaves are
symplectic quotients at other coadjoint orbits, discussed in Section
\ref{shiftquot}.

Second, the symplectic quotient satisfies the following universal
property for quotients.  Suppose that $(X,\omega,\Phi)$ is a
Hamiltonian $K$-manifold and $K$ acts freely on $\Phinv(0)$.  We
denote by $L_\Phi \subset X^- \times (X \qu K)$ the image of
$\Phinv(0)$ under $i \times p$.  Then $L_\Phi$ is a $K$-Lagrangian
correspondence.

\begin{theorem}  Suppose that $X$ is a Hamiltonian $K$-manifold. 
If $Y$ is a symplectic manifold with trivial $K$-action, then any
$K$-Lagrangian correspondence from $X$ to $Y$ factors through
$L_\Phi$.
\end{theorem}

\begin{proof} 
Suppose for simplicity that the morphism consists of a single
correspondence $L \subset X^- \times Y$.  By definition of
$K$-Lagrangian correspondence, $L \subset \Phinv(0) \times Y$.  Since
$K$ acts freely on $\Phinv(0)$, $L / K $ is a submanifold of $X^- \qu K
\times Y$ and is easily checked to be Lagrangian.  Then $L = L/K \circ
L_\Phi$.
\end{proof} 
\noindent Unfortunately the generalization of this universal property
to arbitrary morphisms in the symplectic category requires rather
complicated freeness assumptions.

\subsection{Fubini-Study actions} 

K\"ahler manifolds are complex manifolds with symplectic structures
that are compatible, in a certain sense, with the complex structure.
An {\em almost complex structure} on a manifold $X$ is an endomorphism
$J \in \End(TX)$ with $J^2 = -I$, where $I \in \End(TX)$ is the
identity.  An almost complex structure $J$ is {\em compatible} with a
symplectic structure $\omega$ if $\omega( \cdot, J
\cdot)$ is a Riemannian metric.  Any symplectic manifold admits a
compatible almost complex structure; a K\"ahler manifold is a
symplectic manifold equipped with an integrable compatible almost
complex structure.

Affine and projective space have natural {\em Fubini-Study} K\"ahler
structures as follows. Any Hermitian structure $ ( \, \ ) : V \times V
\to \C$ defines a symplectic structure on $V$ via its imaginary part,
$$ \omega_{V,v}(v_1,v_2) = \Im (v_1,v_2) .$$
while its real part gives a Riemannian metric on $V$.  Let $K$ be a
Lie group acting on $V$.  If $K$ preserves the Hermitian structure
then the action is symplectic and a canonical moment map is given by
$$ \lan \Phi_{V}(v), \xi \ran = \Im (v, \xi v)/2 .$$

\begin{example} Let $K = \Sp(V,\omega)$ be the group of linear
symplectomorphisms of $V$ then the map $\xi \mapsto \lan \Phi_V, \xi
\ran$ defines an isomorphism of the Lie algebra $\sp(V,\omega) $ with
$\on{Sym}^2(V^\dual)$, analogous to the isomorphism of the {\em
  orthogonal} Lie algebras $\o(V,g)$ with $\Lambda^2(V)$.  The Lie
algebra structure induced on $\on{Sym}^2(V^\dual)$ is that induced
from the Poisson bracket by the inclusion $\on{Sym}^2(V^\dual) \subset
C^\infty(V)$. \end{example}

\begin{example} Let $K = S^1$ act on $V = \C^n$ with
weights $a_1,\ldots, a_n$.  If the Hermitian structure on $V$ is the
standard one then the moment map on $V$ is Hamiltonian with moment map
$$ \Phi(z_1,\ldots, z_n) = \sum_{j=1}^n - a_j | z_j|^2 /2 $$
In particular, if $K$ acts by scalar multiplication then 
the moment map is 
$$ \Phi(z_1,\ldots, z_n) = - \sum_{j=1}^n | z_j|^2 /2 .$$
The canonical symplectic quotient $V \qu S^1$ is a point.  If we shift
the moment map by a scalar, $\Phi_c = \Phi + c$, then the symplectic
quotient is
$$ V \qu S^1 = \left\{ \sum_{j=1}^n | z_j|^2 /2 = c \right\} / S^1 
$$
which identifies with the projective space $\P(V)$ of complex lines in
$V$ via $V \qu S^1 \to \P(V), [v] \mapsto \on{span}(v)$.
\end{example} 

It follows that projective space $\P(V)$ naturally has a symplectic
structure, called the {\em Fubini-Study} symplectic form
$\omega_{\P(V)}$. Explicitly this is given as follows: The tangent
space to $\P(V)$ at $[v], v \in V - \{ 0 \}$ naturally identifies with
the Hermitian orthogonal to $[v]$.  Then
$$ \omega_{\P(V),[v]}(v_1,v_2) = \frac{ \Im(v_1,v_2) }{ (v,v) } .$$
If $z_1,\ldots, z_n$ are coordinates corresponding to a unitary basis
then
$$ \omega_{\P(V),[z]} = \frac{ i \sum_{j=1}^n \d z_j \wedge \d \ol{z}_j
} {2 \sum_{j=1}^n z_j \ol{z}_j } .$$
If $K$ acts on $V$ preserving the Hermitian structure, then it
commutes with the action of $S^1$.  The induced action on $\P(V)$ is
also symplectic, and has canonical moment map
$$ \lan \Phi_{\P(V)}([v]), \xi \ran = \Im (v, \xi v)/(v,v) .$$
Suppose that $K = S^1$, and acts on $V$ with weights $a_1,\ldots, a_n
\in \Z$.  The action of $K$ on $\P(V)$ is Hamiltonian with moment map
\begin{equation} \label{projsp}
 \Phi_{\P(V)}([z_1,\ldots, z_n]) = \frac{\sum_{j=1}^n - a_j | z_j|^2
  /2}{\sum_{j=1}^n | z_j|^2 /2}.
 \end{equation}
%

\begin{proposition}  Let $K$ act on $V$ preserving the Hermitian
structure.  Any smooth invariant subvariety $X \subset \P(V)$ inherits
the structure of a Hamiltonian $K$-manifold from the Fubini-Study
Hamiltonian $K$-manifold structure on $\P(V)$.
\end{proposition} 

\begin{proof}  It suffices to check that the restriction 
of $\omega_{\P(V)}$ to $X$ is non-degenerate, which holds since
$\omega_{\P(V)}(v, Jv) > 0$ for $v \in T_xX, Jv \in T_x X$ since
$T_xX$ is $J$-invariant. 
\end{proof} 

\subsection{Geometric quantization} 
\label{gq}

The philosophy of geometric quantization played an important role in
the development of equivariant symplectic geometry.  Unfortunately
good quantization schemes exist only for certain classes of
Hamiltonian actions.

Suppose that $Q$ is a manifold and $T^{\dual} Q$ its cotangent bundle.
One thinks of $T^{\dual} Q $ as the space of classical states for a
particle moving on $Q$, with a vector in $T_q^\dual Q$ representing
the momentum.  In quantum mechanics the state of the system is given
by a {\em quantum wave-function} $\psi \in L^2(Q)$, whose norm-square
$| \psi(q) |^2$ represents the probability of finding the particle at
position $q$, if its position is measured.  The construction of
$L^2(Q)$ from $T^{\dual} Q$ can be done in two steps: first cut down
the number of directions by half, then pass to functions.

One can try to extend this procedure to arbitrary symplectic manifolds
$(X,\omega)$ by axiomatizing this two-step process.  A {\em Lagrangian
  distribution} resp. {\em complex Lagrangian distribution} is a
subbundle $P \subset TX$ resp $TX \otimes_\R \C$ such that each fiber
$P_x$ is a Lagrangian subspace of $T_xX$ resp. complex Lagrangian
subspace of $T_x X \otimes_\R \C$.  A {\em polarization} is a
Hermitian line bundle $L$ with connection $\nabla$ such that the
curvature of $\nabla$ is $\curv(\nabla) = (2\pi/i)\omega$.  A {\em
  quantization datum} resp. {\em complex quantization datum} consists
of a Lagrangian distribution resp. complex Lagrangian distribution
together with a polarization.  The original literature on geometric
quantization uses polarization to refer to the Lagrangian
distribution. This conflicts with the use of polarization in the
geometric invariant theory literature, which we have adopted.  The
{\em geometric quantization} of $(X,\omega)$ (depending on the choice
of $(P,L,\nabla)$) is the vector space of smooth sections of $L$ which
are covariant constant with respect to $\nabla$ along $P$:
$$ \H(X,\omega) := \{ \sigma \in \Gamma(L), \nabla_v \sigma = 0
       \ \forall v \in P \} .$$
We ignore the problem of defining a Hilbert space structure on
$\H(X,\omega)$, see \cite{gu:gea} for more details.

A case for which a good quantization procedure exists is the case that
$X$ is a compact K\"ahler Hamiltonian $K$-manifold equipped with
polarization $\mO_X(1) \to X$.  A Lagrangian distribution is provided
by the antiholomorphic directions on $X$, that is, $P = T^{0,1}X
\subset TX \otimes_\R \C$.  Then $ \H(X,\omega) = H^0(X,\mO_X(1)) .$
In other words, in the language of geometric quantization {\em
  holomorphic sections of the polarizing line bundle are quantum
  states.}

One can now compare the various operations on symplectic manifolds
with those on vector spaces:
\begin{proposition} 
\begin{enumerate}
\item (Duals) If $J$ is the complex structure for $X$ then $-J$ is a
  compatible complex structure for $X^-$.  If $P = T^{0,1} X$ then
  $\ol{P} = T^{0,1} X^-$.  Furthermore, $\ol{L}$ with connection
  $-\alpha$ is naturally a polarization for $X^-$.  Thus $\H(X^-)$ is
  the space of complex-conjugates of sections of $L$, which is
  naturally identified with the dual $\H(X)^\dual$ of $\H(X)$.
\item (Sums) If $X_0,X_1$ are K\"ahler Hamiltonian $K$-manifolds with
  polarizations, then $\H(X_0 \cup X_1) = \H(X_0) \oplus \H(X_1)$.
\item (Products) With the same assumptions as in (b), $\H(X_0 \times
  X_1) = \H(X_0) \otimes \H(X_1)$.
\end{enumerate} 
\end{proposition}

\begin{example}  Let $X = S^2 \cong \P^1$ and $\omega$ the standard
symplectic form.  The moment map for the action of $S^1$ on $(X, d
\omega)$ is has image $[-d,d]$.  The $d$-th tensor product $\mO_X(d)$
of the hyperplane bundle $\mO_X(1)$ is a polarization of $(X,d
\omega)$, so that $\H(X,\omega) = H^0(X,\mO_X(d))$ is the space of
homogeneous polynomials in two variables of degree $d$.  Note that the
weights of $\H(X,\omega)$ are $\{ d, d-2, d-4,\ldots, -d \}$, which
are the intersections points of the image $\Phi(X)$ with the lattice
$d + 2\Z \subset \Z$.  The $SU(2)$-action on $X$ induces on $\H(X)$
the structure of an $SU(2)$-module with highest weight $d$.  The
product of spheres $S^2_{\lambda_1} \times \ldots \times
S^2_{\lambda_n}$ has quantization the tensor product of simple
$SU(2)$-modules $V_{\lambda_1} \otimes \ldots \otimes V_{\lambda_n}$.
\end{example} 

\noindent Unfortunately (i) quantizing arbitrary morphisms (i.e.
Lagrangian correspondences) is quite difficult, even in this case (ii)
there is no good geometric quantization scheme for arbitrary
symplectic manifolds.  The problem of finding good schemes for say,
coadjoint orbits of real Lie groups or moduli spaces of flat
connections have vast literatures attached to them.

The reader may notice that we have not said anything yet about the
behavior of the quantum state spaces under the symplectic quotient
construction. We take this up in Section \ref{knthm}.

\section{Geometric invariant theory}  

In this section we review Mumford's geometric invariant theory
\cite{mu:ge}, see also Brion's review in this volume or the reviews by
Newstead \cite{news:git} or Schmitt \cite{schmitt:git}.  For
connections to moduli problems see  Newstead
\cite{ne:mo}.

\subsection{Algebraic group actions and quotients} 

Let $G$ be a complex linear algebraic group.  $G$ is called {\em
  reductive} iff every $G$-module splits into simple $G$-modules, or
equivalently, if $G$ is the complexification of a compact Lie group
$K$.  A {\em Borel subgroup} of a reductive group $G$ is a maximal
closed connected solvable subgroup $B \subset G$.  The set of Borel
subgroups is in bijection with set of right cosets $G/B$, called the
{\em generalized flag variety} for $G$, via the map $g B \mapsto
gBg^{-1}$.  A subgroup $P \subset G$ is {\em parabolic} iff $G/P$ is
complete iff $P$ contains a Borel subgroup.  The quotient $G/P$ is
called a {\em generalized partial flag variety}.  Let $T$ be a maximal
torus of $G$ (for example, the complexification of a maximal torus of
the maximal compact subgroup, which was also somewhat confusingly
called $T$.)  We denote by $W = N(T)/T$ the Weyl group of $T$.  The
action of $T$ on the Lie algebra $\g$ induces a {\em root space
  decomposition}
$$\g = \t \oplus \bigoplus_{\alpha  \in \RR(\g)} \g_\alpha $$
where $T$ acts trivially on $\t$ and on $\g_\alpha$ by $t \xi =
t^{\alpha} \xi$, and $\RR(\g) \subset \Lambda^\dual$ is the set of
roots of $\g$.  Given a choice of positive Weyl chamber let $B^\pm$ be
the Borel subgroups whose Lie algebras contain the positive
resp. negative root spaces of $\g$.  Each $\lambda \in \t^\dual$
determines {\em standard parabolic subgroups} $P_\lambda^\pm$ with Lie
algebra $\p_\lambda^\pm = \b^\pm + \bigoplus_{ \lan h_\alpha, \lambda
  \ran = 0 } \g_\alpha $, where $h_\alpha \in \t$ is the coroot
corresponding to $\alpha \in \t^\dual$.  Any parabolic subgroup (in
particular, any Borel) is conjugate to a standard parabolic subgroup.

An {\em action} of $G$ on a variety $X$ is a morphism $G \times X \to
X$ such that $g_1(g_2x) = (g_1g_2)x$ and $ex =x$, for all $g_1,g_2 \in
G, x \in X$.  A variety $X$ equipped with a $G$-action is called a
{\em $G$-variety}.  An {\em (\'etale) slice} for the action of $G$ at
$x \in X$ is an affine subvariety $V \subset X$ and a $G$-morphism $G
\times_{G_x} V \to X$ that is an isomorphism (\'etale morphism) onto a
neighborhood of $X$.  In contrast with the case of compact group
actions, reductive group actions do not in general have slices.
Luna's slice theorem \cite{lu:sl} asserts that any {\em closed} orbit
of an action of a reductive group on an affine variety has an \'etale
slice.  A {\em categorical quotient} of $X$ by $G$ is a pair $(Y,\pi)$
where $Y$ is a variety and $\pi:X \to Y$ is a $G$-invariant morphism
that satisfies the universal property for quotients: if $f: X \to Z$
is a $G$-invariant morphism then $f$ factors uniquely through $Y$. A
{\em good quotient} of $X$ is a pair $(Y,\pi)$ where
\begin{enumerate} 
\item $\pi: X \to Y$ is $G$-invariant, affine, surjective, 
\item if $U \subset Y$ is open then $\mO_Y(U) \to \mO_X(\pi^{-1}(U))^G$ 
is an isomorphism 
\item If $W_1,W_2$ are disjoint closed $G$-invariant subsets
of $X$ then $\pi(W_1),\pi(W_2)$ are disjoint closed subsets
of $x$. 
\end{enumerate} 
A good quotient is automatically a categorical quotient.  A {\em
  geometric quotient} is a good quotient that separates orbits.

\begin{example}  The generalized flag variety $X = G/B^-$ 
has is a $G$-variety for the left action.  If $G$ is connected
reductive then $X$ has a canonical decomposition into {\em Bruhat
  cells}
\begin{equation} \label{bruhat} 
X = \bigcup_{w \in W} X_w, \quad X_w := B wB^-/B^- \end{equation} 
and {\em opposite Bruhat cells} 
\begin{equation} \label{oppbruhat} 
X = \bigcup_{w \in W} Y_w, \quad Y_w := B^- wB^-/B^- .\end{equation}
The codimension resp. dimensions are given by
$$ \codim(X_w) = l(w),  \quad \dim(Y_w) = l(w) $$ 
where $l(w)$ is the minimal number of simple reflections in a
decomposition of $w$.  We denote by $x_w = wB^-/B^- = X_w \cap Y_w$
the unique $T$-fixed point in $X_w$ resp. $Y_w$.  There is a similar
decomposition of any generalized flag variety $X = G/P_\lambda^-$ into
cells $X_{[w]}$ indexed by $[w] \in W/W_\lambda$.  In the special case
$G = GL(r)$, the Weyl group $W$ is naturally identified with the
symmetric group and $B^\pm$ are the groups of invertible upper
resp. lower triangular matrices.  We identify $\k \to \k^\dual$; if
$\lambda = \diag i(1,\ldots,1,0,\ldots,0)$ has rank $s$ then
$P_\lambda$ is the group of matrices preserving the subspace $\C^s
\oplus 0 \subset \C^r$.  The quotient $X = G/P_\lambda$ is isomorphic
to the Grassmannian $G(s,r)$ of $s$-dimensional subspaces of $\C^r$.
The quotient $W/W_\lambda$ is natural identified with the set of
subsets $I \subset \{ 1,\ldots, r \}$ of size $s$ via the map $w
\mapsto w \{ 1,\ldots, s \}$.  Let $F_1 \subset F_2 \subset \ldots
\subset F_r = \C^r$ be the standard flag in $\C^r$.  Then the opposite
Bruhat cell $Y_I$ has closure the {\em Schubert variety}
\begin{equation} \label{schubert}
 \ol{Y}_I = \{ E \in G(s,r), \dim(E \cap F_{i_j}) \ge j , j = 1,\ldots, s
\} .\end{equation} 
This ends the example.
\end{example}

\subsection{Stability conditions} 

Let $G$ be a complex reductive group and $X$ a $G$-variety.  A {\em
  polarization} of $X$ is an ample $G$-line bundle $\mO_X(1) \to X$.
Its $d$-th tensor power is denoted $\mO_X(d)$.  Let
$$ R(X) = \bigoplus_{d \ge 0} H^0(X,\mO_X(d)) .$$
The action of $X$ induces an action on $R(X)$ by pull-back.  We denote
by $R(X)^G \subset R(X)$ the subring of invariants, and by 
$R(X)_{> 0}^G$ the part of $R(X)^G$ of positive degree.

\begin{definition}   A point $x \in X$ is 
\begin{enumerate}
\item {\em semistable} if $s(x) \neq 0$ for some $s \in R(X)^G_{> 0}$; 
\item {\em polystable} if $x$ is semistable and $Gx \subset X^{\ss}$
is closed;
\item {\em stable} if $x$ is polystable and has finite stabilizer;
\item {\em unstable} if $x$ is not semistable. 
\end{enumerate}
\end{definition}  

\begin{example}  Suppose that $G = \C^*$ acts on $\P^2$ by 
$g[z_0,z_1,z_2] = [g^{-1} z_0, z_1, gz_2]$.  Then $R(X)_d$ is spanned
  by $z_0^{d_0} z_1^{d_1} z_2^{d_2}$ with $d_0 + d_1 + d_2 = d$, which
  has weight $d_0 - d_2$ under $\C^*$.  Thus the invariant sections
  have $d_0 = d_2$.  One sees easily that $x$ is
\begin{enumerate} 
\item semistable iff $x \neq [1,0,0],[0,0,1]$ 
\item polystable iff $x \in \{ [0,1,0] \} \cup \{ [z_0,z_1,z_2] | z_0
  z_2 \neq 0 \}$
\item stable iff $x \in \{ [z_0,z_1,z_2] | z_0 z_2 \neq 0 \}$
\end{enumerate} 
\end{example}

Let $X^{\ss}$ resp. $X^{\ps}$ resp $X^\s$ resp. $X^{\us}$ denote the
semistable resp. polystable resp. stable resp. unstable locus.  We
will need the following alternative characterizations of poly
resp. semistability, see Mumford \cite{mu:ge} or Brion's lectures in
this volume:

\begin{lemma}\label{altern}  Let $X \subset \P(V)$ be a $G$-variety.
A point $x \in X$
%
%
%
%
is polystable (resp. semistable) iff the orbit of any lift $v$ in $V$
is closed (resp. $0$ does not lie in the closure of $Gv$).
\end{lemma}  

\noindent Define an equivalence relation on orbits as follows:

\begin{definition} 
 {\em Orbit-equivalence} is the equivalence relation on $X^{\ss}$
 defined by $x_0 \sim x_1$ iff $\ol{Gx_0} \cap \ol{Gx_1} \cap X^{\ss}
 \neq \emptyset$.
\end{definition} 

\noindent Transitivity of this relation follows from:

\begin{proposition} (see \cite{mu:ge}) 
The closure $\ol{Gx}$ of any semistable $x$ contains a unique
polystable orbit.  Hence two orbits $Gx_0,Gx_1$ are orbit-equivalent
iff their closures contain the same polystable orbit.
\end{proposition}  

\noindent See Theorem \ref{jh} for an analytic proof.  The following
can be considered the main result of geometric invariant theory
\cite{mu:ge}:

\begin{theorem}[Mumford] \label{git}
 Let $X$ be a projective $G$-variety equipped with polarization
 $\mO_X(1)$.
\begin{enumerate} 
\item There exists a categorical quotient $\pi: X^{\ss} \to X \qu G$.  
\item $\pi(X^{\s}) \subset X \qu G$ is open and $\pi | X^{\s}: X^{\s}
  \to \pi(X^{\s})$ is a geometric quotient.
\item The topological space underlying $X \qu G$ is the
  space of orbits modulo the orbit-closure relation $ X^{\ss} / \sim
  .$ 
\item $X \qu G$ is isomorphic to the projective variety with
  coordinate ring $R(X)^G .$
\end{enumerate} 
\end{theorem} 
\noindent Some authors prefer to write $X^{\ss} \qu G$ for the
geometric invariant theory quotient, while we drop the superscript
from the notation.

\subsection{The Hilbert-Mumford criterion} 

Mumford \cite{mu:ge}, based on previous work of Hilbert for the case
of the special linear group acting on projective space, gave a method
for explicitly identifying the semistable loci:

\begin{theorem} \label{HN} (Hilbert-Mumford criterion)
Let $X$ be a polarized projective $G$-variety.  $x \in X$ is
semistable iff $x$ is semistable for all one-parameter subgroups $\C^*
\to G$.
\end{theorem} 

\noindent
One direction of the Hilbert-Mumford criterion is trivial: Let $X$ be
a polarized $G$-variety.  Suppose that $x$ is $G$-semistable, so that
there exists $s \in R(X)_{> 0}^G$ with $s(x) \neq 0$. Then $s$ is also
invariant for any one-parameter subgroup, hence $x$ is semistable for
any one-parameter subgroup.  The other direction is somewhat harder;
the proof given in Mumford \cite{mu:ge} uses an algebraic theorem of
Iwahori.  We will give an alternative analytic proof using the
Kempf-Ness function in Section \ref{hesselink}.

\vskip .in

\noindent The following is a fundamental example:

\begin{example}  Let $X = (\P^1)^n$ and $\O_X(1) = \O_{\P^1}(1)^{\boxtimes n}$
the $n$-fold exterior tensor product.  The group $G = SL(2,\C)$ acts
diagonally on $X$.  We wish to show
\begin{enumerate}
\item $X^{\ss} = \{ (x_1,\ldots, x_n) \in (\P^1)^n, \ \text{ at most $n/2$
 points equal} \} $.
\item 
$X^{\s} = \{ (x_1,\ldots, x_n) \in (\P^1)^n, \ \text{ less than $n/2$
 points equal} \} $.
\item 
$X^{\ps} - X^{\s} = \{ (x_1,\ldots, x_n) \in X^{\ss}, \# \{
  x_1,\ldots, x_n \} = 2 \} $.  In other words, $n/2$ are equal and the
    other $n/2$ are also equal.
\end{enumerate} 
Indeed, if $z_j,w_j$ are the coordinates on the $j$-factor then
$H^0(\mO_X(d))$ is spanned by $z_1^{d_1} w_1^{d- d_1} \ldots z_n^{d_n}
w_n^{d- d_n}$ where $d_j \in [0,d], j = 1,\ldots, n$. If $\C^* \subset
G$ is the standard maximal torus given by $g \mapsto \diag(g,g^{-1})$
then $H^0(\mO_X(d))^{\C^*}$ is spanned by the polynomials $z_1^{d_1}
w_1^{d- d_1} \ldots z_n^{d_n} w_n^{d- d_n}$ with $\sum_{j=1}^n d_j =
\sum_{j=1}^n d - d_j $, that is, $\sum (d_j/d) = n/2$.  Since $d_j/d
\in [0,1]$, this means that at least $n/2$ of the $d_j$'s are
non-zero. Thus $([z_1,w_1], \ldots, [z_n,w_n])$ is $\C^*$-semistable
iff at most $n/2$ $z_j$'s and at most $n/2$ $w_j$'s equal zero.
Repeating the same for an arbitrary one-parameter subgroup (or
equivalently, basis for $\C^2$) proves the claim.
\end{example} 

\begin{example} More generally, suppose that 
$X = (\P^1)^n$ is equipped with the polarization $\mO_X(1) :=
  \boxtimes_{i=1}^n \mO_{\P^1}(\lambda_i) $ for some positive integers
  $\lambda_1,\ldots, \lambda_n$.  Then $x = (x_1,\ldots, x_n)$ is
  semistable iff for all $x \in \P^1$,
$$ \sum_{x_j = x} \lambda_j \leq \sum_{x_j \neq x} \lambda_j .$$
\end{example} 

For future use we mention the following equivalent form of the
Hilbert-Mumford criterion and Lemma \ref{altern}:

\begin{corollary} \label{linear} Let $G$ be a reductive group acting linearly 
on a finite dimensional vector space $V$.  For any $v \in V$,
$\ol{Gv}$ contains $0$, if and only if $ \ol{\C^* v}$ contains $0$ for
some one-parameter subgroup $\C^* \subset G$.
\end{corollary} 

\begin{remark} The statement of the corollary does not hold for
arbitrary (that is, not linear) actions resp. arbitrary points.  An
example I learned from Brion: Let $X = \P( S^3(\C^2) \oplus \C)$ with
the action induced from the action of $SL(2,\C)$ on $\C^2$ and the
trivial action on $\C$.  Identifying $S^3(\C^2)$ with homogeneous
polynomials in two variables $u,v$, one sees that the orbit of
$[u^2v,1]$ contains the orbit of $[u^3,1]$ in its closure.  The
stabilizer of $[u^3,1]$ is a maximal unipotent subgroup of $SL(2,\C)$
and so does not contain a copy of $\C^*$.  Thus $[u^3,1]$ cannot be
contained in the closure of an orbit of a one-parameter subgroup.  On
the other hand, the lemma is true for arbitrary actions of abelian
groups, as follows from, for example, Atiyah Theorem' \ref{atiyah}
below.
\end{remark}

\section{The Kempf-Ness theorem} 
\label{knthm} 

The material in this section is contained in the original paper of
Kempf-Ness \cite{ke:le}, the book of Mumford-Fogarty-Kirwan
\cite{mu:ge}, and the paper of Guillemin-Sternberg \cite{gu:ge}.  The
notes of Thomas \cite{thomas:rev} and the thesis of Sz{\'e}kelyhidi
\cite{sz:em} also describe the Kempf-Ness theorem with many examples
and generalizations.

\subsection{Complexification of Lie groups and their actions} 

We begin with some basic remarks on the relation between complex and
compact group actions.  Any compact Lie group $K$ admits a {\em
  complexification} $G$, a complex reductive Lie group $G$ containing
$K$ as a maximal compact real subgroup, and whose Lie algebra $\g$ is
equal to $ \k \oplus i \k$.  The complexification $G$ satisfies the
universal property that any Lie group homomorphism from $K$ to a
complex Lie group $H$ extends to a complex Lie group homomorphism from
$G$.  The complexification $G$ admits a {\em Cartan decomposition}: a
diffeomorphism (see Helgason \cite[VI.1.1]{he:sy})
\begin{equation} \label{cartan}
 K \times \k \to G, \quad (k,\xi) \mapsto k \exp(i\xi) .\end{equation}
We denote by $K \bs G$ the quotient by the left action, which is a
symmetric space of non-compact type with non-positive curvature
\cite{he:sy}, \cite{bgs:man}.  For any point $[g] \in K \bs G$ the
geodesics through $[g]$ are of the form $[\exp( i \xi)g]$ for $\xi \in
\k$ (see \cite[Exercise 1]{he:sy}) and $\xi \mapsto [\exp(i\xi)]$
defines a diffeomorphism of $\k$ onto $K \bs G$.

If $X$ is a compact complex manifold then the group $\Aut(X)$ of
automorphisms is a complex Lie group, with Lie algebra given by the
space $H^0(X,TX)$ of holomorphic vector fields on $X$, see for example
Akhiezer \cite{akh:lie}.  Any action of a compact group $K$ therefore
extends to the complexification $G$.

By a {\em K\"ahler Hamiltonian $K$-manifold} we mean a compact
Hamiltonian $K$-manifold equipped with an integrable $K$-invariant
complex structure. If $X$ is compact then the $K$-action automatically
extends to a $G$-action preserving the complex structure but not the
symplectic structure.  By the Kodaira embedding theorem, if the
symplectic form is rational then a compact K\"ahler Hamiltonian
$K$-manifold is isomorphic as a complex $G$-manifold to a smooth
complex algebraic $G$-variety.  However, the symplectic form may not
be the pull-back of the Fubini-Study form under any holomorphic
embedding of $X$, see for example Tian \cite{tian:pol}.  The
generating vector fields for $\xi \in \k$ are the Hamiltonian flows
corresponding to the moment map components $\lan \Phi, \xi \ran$,
while the generating vectors fields for $i \xi, \xi \in \k$ are the
{\em gradient flows} corresponding to $\lan \Phi, \xi \ran$.  In
particular, for any $x \in X, \xi \in \k$, the trajectory $\exp( i t
\xi) x$ converges to a point $x_\infty \in X$ with $\xi_X(x_\infty) =
0$.  Furthermore, since $\lan \Phi, \xi \ran$ is a Morse function by
Lemma \ref{morse}, this convergence is {\em exponentially fast} in
$t$; the exponential nature of convergence will be used later.

\begin{example} 
The example of flag varieties will be particularly important later and
we briefly describe these actions from the algebraic and symplectic
points of view.  Let $V$ be a finite dimensional vector space.  A {\em
  partial flag} in $V$ is a filtration $F =( F_1 \subset F_2 \subset
\ldots \subset F_m \subset V)$.  The {\em type} of $F$ is the sequence
of dimensions $\dim(F_1) < \dim(F_2) < \ldots < \dim(F_m)$.  Given a
sequence $t = (0 < t_1 < \ldots < t_m < \dim(V)) \in \Z^m$ we let
$\on{Fl}(t,V)$ denote the set of partial flags of type $t$.  The
general linear group $GL(V)$ acts transitively on $\on{Fl}(t,V)$ with
stabilizer the parabolic subgroup of transformations preserving the
filtration.  A $GL(V)$-equivariant canonical projective embedding of
$\on{Fl}(t,V)$ is given by choosing a basis $v_1,\ldots, v_n$ so that
$v_1,\ldots, v_{t_j}$ is a basis for $F_j$ for each $j = 1,\ldots, m$,
and mapping
$$ \on{Fl}(t,V) \to \prod_{j=1}^m \P(\Lambda^{t_j} V), 
\quad F \mapsto
\prod_{j=1}^m \Lambda_{k=1}^{t_j} v_k .$$
Given a Hermitian metric on $V$, any partial flag induces a
Hermitian splitting
$$ V = F_{1} \oplus (F_{2} \cap F_{1}^\perp) \oplus (F_{3} \cap F_{2}^\perp)
\ldots \cap (F_{m} \cap F_{{m-1}}^\perp) $$
and such splittings are in one-to-one correspondence with flags.
Given real numbers $\lambda_1 > \ldots > \lambda_m$ the flag defines a
skew-Hermitian operator acting by $i\lambda_j$ on $F_{j} \cap
F_{j-1}^\perp$.  Conversely, any such Hermitian operator determines a
splitting via its eigenspace decomposition.  The unitary group $K =
U(V)$ acts transitively on the space of such matrices, which form an
orbit of the action of $K$ on the Lie algebra $\k$.  Now $\k$ may be
identified with its dual via any invariant inner product, so one sees
that $\on{Fl}(t,V)$ is naturally identified with the coadjoint orbit
$K\lambda$ of $\lambda$, identified with an element of $\k^\dual$ via
the inclusion $\t \to \k$ and an identification $\k \to \k^\dual$.
Given a generic $\xi \in \t_+$, the stable resp. unstable manifolds of
the Morse function $\lan \Phi, \xi \ran$ are the Bruhat resp. opposite
Bruhat cells of \eqref{bruhat} resp. \eqref{oppbruhat}.
\end{example} 


\subsection{Statement and proof}

The Kempf-Ness theorem states the equivalence of the symplectic and
geometric invariant theory quotients; the affine case is treated in
\cite{ke:le} and the projective case is similar (Theorem 8.3 in
\cite{mu:ge}). 

\begin{theorem} 
\label{KN} 
 Let $K$ be a compact group and $G$ its complexification.  Let $V$ be
 a $G$-module equipped with a $K$-invariant Hermitian structure.  Let
 $X \subset \P(V)$ be a smooth projective $G$-variety, and $\Phi: X
 \to \k^\dual$ the Fubini-Study moment map. Then $\Phinv(0) \subseteq
 X^{\ps}$ and the inclusion induces a homeomorphism $X \qu K \to X \qu
 G$.
\end{theorem}  

\noindent The proof uses the properties of a {\em Kempf-Ness function}
for each $v \in V - \{ 0 \}$:
$$ \psi_v: K \bs G \to \R, \quad [g] \mapsto \log \Vert g v \Vert^2/2 .$$
We denote by $\partial_\lambda \psi_v([g])$ the derivative of $\psi_v$
along the geodesic $[\exp(i \lambda)g]$ determined by $\lambda$; note
that this depends on a choice of representative $g$ of $[g]$.  The
Kempf-Ness function can be viewed as the integral of the moment map in
the following sense:

\begin{lemma} \label{deriv} 
For all $v \in V$ and $\lambda \in \k$ we have \label{integral} $
\partial_\lambda \psi_v( [g] ) = - \lan \Phi([gv]),\lambda \ran .$
\end{lemma}

\begin{proof}  The proof uses the explicit formula for the Fubini-Study
moment map
\begin{eqnarray*} 
\partial_\lambda \psi_v([g]) &=& \ddt |_{t = 0} \log \Vert \exp( i t
\lambda) g v) \Vert^2/2 \\
&=& \frac{ ( i \lambda g v, g v)} { (g v, g v)} =  -  \lan \Phi([gv]) ,
\lambda \ran .\end{eqnarray*}
\end{proof} 

\begin{corollary}  \label{gradpsi} 
The gradient of $\psi$ is equal to $\Phi$, that is,
$$\grad \psi ([g]) = \ddt |_{t=0} [\exp(- i t \Phi (g x) g] .$$
\end{corollary} 

\begin{proof}  By Lemma \ref{deriv},
$$ (\grad \psi ([g]), \lambda) = \ddt |_{t = 0} \psi( [ \exp( i t
    \lambda) g ]) = - \lan \Phi( gx ), \lambda \ran .$$
\end{proof} 

The basic property of the Kempf-Ness function is its {\em convexity}:
its restriction to any geodesic in $K \bs G$ is a convex function, or
equivalently, its second derivatives along geodesics are non-negative
\cite[Section 6]{do:fo}, \cite[Section 2]{klm:co}.

\begin{corollary} \label{convexp} 
\begin{enumerate}
\item For any $v \in V$, $\psi_v$ is a convex function with critical
  points given by points $[g] \in K \bs G$ such that $\Phi([gv]) = 0
  $.  
\item The second derivative $\partial_\lambda^2 \psi_v ([e])$ along
  the geodesic determined by $\lambda \in \k$ is positive iff
  $\lambda$ lies in $\k - \k_x$.
\item For $\xi \in \k_x$ we have $\psi_v([\exp(i\xi)]) = \psi_v([e]) +
  2 \lan \Phi(x), \xi \ran$.
\end{enumerate} 
\end{corollary} 

\begin{proof}  The claim on the critical points follows from 
Lemma \ref{deriv}: 
\begin{eqnarray*}
  D_{[g]} \psi_v = 0 &\iff& \partial_\lambda \psi_v ([g]) = 0,
  \ \forall \lambda \in \k\\ &\iff& \lan \Phi([gv]) , \lambda \ran =
  0, \ \forall \lambda \in \k \\ &\iff& \Phi([gv]) = 0 .\end{eqnarray*}
To prove convexity we compute the second derivatives
\begin{eqnarray*} 
(\partial_\lambda)^2 \psi_v([g]) &=& - \ddt|_{t = 0} \lan \Phi ( [ \exp(
    i \lambda t) g v]) , \lambda \ran \\
&=& - 2 L_{J \lambda_X} \lan \Phi, \lambda \ran ([g v]) \\
&=& 2 \omega( \lambda_X, J\lambda_X )([g v]) \ge 0  
\end{eqnarray*}
since $\omega( \cdot, J \cdot)$ is a Riemannian metric.  The claim on
strict convexity and the formula for $\psi_v([ \exp(i \xi)]), \xi \in
\k_x$ are immediate from the previous lemma.
\end{proof} 

Note that if $\psi_v$ is strictly convex (that is, has trivial
infinitesimal stabilizer) and has a critical point, then the critical
point is the unique global minimum.  The following lemma characterizes
for which $v$ minima of $\psi_v$ exist:

\begin{lemma} \label{bounded} Let $v \in V - \{0 \}$ and $x = [v] \in \P(V)$. 
\begin{enumerate} 
\item $\psi_v$ attains a minimum iff $x$ is polystable.
\item $\psi_v$ is bounded from below iff $x$ is semistable. 
\end{enumerate} 
\end{lemma} 

\begin{proof}   (a) 
Recall from \ref{altern} that $x$ is polystable iff $Gv$ is closed.
Suppose $Gv$ is closed.  Let $ Kg_j \in K \bs G$ be a minimizing
sequence for $\psi_v$.  Then after passing to a subsequence $g_j v$
converges to $gv$ for some $g \in G$, since $Gv$ is closed, and $K g$
must be a global minimum of $\psi_v$, since $\psi_v$ is convex.

Conversely, suppose that $\psi_v$ attains a minimum at $Kg \in K \bs
G$.  After replacing $\psi_v$ with $\psi_{gv}$, we may assume that $g
= e$.  Clearly $\psi_v$ is invariant under the stabilizer $G_v$ of
$v$.  We claim that the induced map
\begin{equation} \label{qpsi} 
\psi_v/G_v : K \bs G / G_v \to \R \end{equation}
is proper.  Suppose that $K g_j G_v$ is a sequence of points in $K \bs
G$ such that $\psi(K g_j)$ is bounded.  Let $O$ be the image of $G_v$
under $G \to K \bs G$; since this is proper and $O$ is an orbit of
$G_v$, $O$ is a closed submanifold.  It follows that there exists a
minimal length path connecting $Kg_j$ to $O$.  Since any such path is
a geodesic, we have
$$ \gamma_j(t) = K \exp( i t \xi_j ) h_j , t \in [0,1] $$
for some $\xi_j \in \k$ and $h_j \in G_v$.  The direction of
$\gamma(t)$ at $t = 0$,
$$\ddt |_{t = 0} \gamma_j(t) = \ddt |_{t = 0} K \exp( i t \xi_j) h_j $$
is perpendicular to the tangent space $T_{Kh_j} O$ to $O$ at $Kh_j$,
$$ T_{Kh_j} O = \left\{ \ddt |_{t = 0} K \exp( t \mu) h_j , \mu \in
\g_v \right\} $$
since otherwise one could find a shorter path.  Thus
$$ \ddt |_{t = 0} K \exp( i t \xi_j) h_j \perp \ddt |_{t = 0} K \exp(
 i \mu ) h_j $$
for all $\mu \in \g_v$.  Since the metric is invariant under the right
action of $h_j$, $i \xi_j$ is perpendicular to the projection of
$\g_v$ on $i \k$.  Note that since $\psi_v$ is bounded, $\Phi([v])$
vanishes on $\k_x$ by the last sentence in Corollary \ref{convexp}.
Hence $i \k_x$ is contained in $\g_v$ by Proposition \ref{weight},
thus $i \xi_j$ is perpendicular to $i \k_x$.  Strict convexity of
$\psi_v$ along any geodesic of the form $K\exp(i t \xi)$ with $\xi \in
\k_x^\perp$ implies that
$$ \psi_v( K \exp( it \xi)) < C_0 + C_1 \Vert \xi \Vert , \quad
\forall \xi \in \k_x^\perp .$$
where
$$C_0 = \sup_{ \Vert \xi \Vert \leq 1 } \psi_v (K \exp(i \xi)), \quad
C_1 = \sup_{ \Vert \xi \Vert = 1 } \ddt |_{t = 0} \psi_v (K \exp(i t
\xi)) .$$
Since $\psi( K \exp( i t \xi_j) h_j) = \psi(K \exp( it \xi_j))$ is
bounded, so is $\xi_j$.  Hence \eqref{qpsi} is proper, which completes
the proof of the claim.

To show that $Gv$ is closed, suppose that $g_j v$ is a sequence
converging to some $v_\infty$.  Then $ \psi_v (K g_j)$ is bounded, so
by the claim on properness of \eqref{qpsi}, the sequence $g_j$
converges, after passing to a subsequence, to some $g_\infty$.
Continuity of the action implies that $g_j v $ converges to $ g_\infty
v = v_\infty$, so $v_\infty$ lies in $Gv$.

(b) If $\psi_v$ is bounded from below, then any minimizing sequence
$\xi_j$ has $\exp(i \xi_j)x$ converging to a critical point of
$\psi_v$, which is necessarily a zero of $\Phi$.  Hence $Gx$ contains
a polystable orbit in its closure and is therefore semistable.  If
$\psi_v$ is not bounded from below, then $\ol{Gv}$ contains $0$ and so
$x$ is unstable, see Lemma \ref{altern}.
\end{proof} 

\begin{corollary}\label{ps}  $X^{\ps} = G \Phinv(0)$.  
\end{corollary} 

\begin{proof}  By  Lemmas \ref{bounded}, \ref{convexp}, \ref{altern}.   
\end{proof} 

\begin{proof}[Proof of the Kempf-Ness theorem \ref{KN}]  Consider the
inclusion 
$$i/K: \Phinv(0)/K \to X^{\ps}/G \cong X \qu G .$$  
First note that $i/K$ is injective: Suppose that $x_0, x_1 \in
\Phinv(0)$ are such that $x_0 = gx_1$ for some $g \in G$.  Choose a
lift $v$ of $x_0$.  Then both $[e],[g]$ are global minimum points of
$\psi_v$, and since $\psi_v$ is convex this implies that the geodesic
$ [\exp( i t \xi)], t \in [0,1]$ connecting $[e], [g]$ also consists
of global minima.  But then $\xi \in \k_{x_0}$ and so $Kx_0 = Kx_1$.
Next note that $i/K$ is surjective by Corollary \ref{ps}.  Finally
$i/K$ is a homeomorphism: Any bijection from a Hausdorff space to a
compact space is a homeomorphism.  (Alternative, the gradient flow of
the norm-square of the moment map discussed in Section \ref{hkn}
defines a continuous inverse to $i/K$.)
\end{proof} 

\begin{remark}  \label{KNkahler} Let $X$ be a compact K\"ahler Hamiltonian $K$-manifold.  
An analog of the Kempf-Ness function may be obtained by integrating
the one-form given by the moment map: Define $\alpha \in
\Omega^1(K \bs G)$ by 
$$ \alpha_{[g]}([ \ddt |_{t = 0} \exp( i t \lambda)g) ]) = \lan
\Phi(gx), \lambda \ran .$$
Then anti-symmetry of $\omega$ implies that $\alpha$ is closed, hence
exact by the Poincar\'e lemma, hence $\alpha_x = \d \psi_x$ for some
$\psi_x: K \bs G \to \R$.  Say that a point $x \in X$ is {\em
  polystable} iff $\psi_x$ attains a minimum, {\em semistable} iff
$\psi_x$ is bounded from below.  With these definitions the following
K\"ahler analog of the Kempf-Ness theorem holds, c.f. Mundet
\cite{mun:poly}, Heinzner-Loose \cite{he:re}, Heinzner-Huckleberry
\cite{heinz:kahl}, Bruasse-Teleman \cite{brausse:hn}, Teleman
\cite{teleman:stab}: Let $X \qu G$ be the quotient of the semistable
locus by the orbit closure equivalence relation. Then the same
arguments show that $\Phinv(0)$ is contained in the semistable locus
and the inclusion induces a homeomorphism $X \qu K \to X \qu G$.
\end{remark}

\label{cgsec}

\noindent We discuss the geometry of the Kempf-Ness function further
in Theorem \ref{conelem}.

\begin{example} 
We illustrate the theorem with the {\em Clebsch-Gordan} theory of
existence of invariants in tensor products of representations of $G =
SL(2,\C)$.  The weight lattice $\Lambda^\dual$ for $G$ is naturally
identified with the set $\Z/2$ of non-negative half-integers and for
any $\lambda \in \Lambda^\dual, \lambda \ge 0$ we denote by
$V_\lambda$ the corresponding simple $G$-module.  Given
$\lambda_1,\ldots, \lambda_n$ we ask whether $V_{\lambda_1} \otimes
\ldots \otimes V_{\lambda_n}$ contains an invariant vector.  Now
$H^0(\P^1, \mO_{\P^1}(d)) \cong V_{d/2}$ and so
$ R(\P^1) = \oplus_\lambda V_\lambda .$
If we equip $X = (\P^1)^n$ with the ample line bundle 
$\mO_X(1) := \boxtimes_{j=1}^n \mO_{\P^1}(\lambda_j) $
then 
$$ R(X) = \bigoplus_{d \ge 0} \bigotimes_{j=1}^n H^0(\mO_{\P^1}(d
\lambda_j)) = \bigoplus_{d \ge 0} \bigotimes_{j=1}^n V_{d \lambda_j} .$$
So 
$$ R(X \qu G) = R(X)^G = 
(\bigoplus_{d \ge 0} \bigotimes_{j=1}^n V_{d \lambda_j})^G .$$
This is non-zero if and only if $X \qu G$ is empty.  The Kempf-Ness
Theorem \ref{KN} gives
$ X \qu G \cong X \qu K \cong (S^2_{\lambda_1} \times \ldots \times S^2_{\lambda_n}) 
\qu SU(2) $
where $S^2_\lambda$ denotes the two-sphere equipped with re-scaled
symplectic form $\lambda$ and $SU(2)$ acts via the double cover $SU(2)
\to SO(3)$.  By Proposition \ref{poly},

\begin{corollary}  
$(\otimes_{j=1}^n V_{d \lambda_j})^G $ is non-trivial for some $d$ iff
\begin{equation} \label{cg}
 \lambda_j \leq \sum_{i \neq j} \lambda_i, j = 1,\ldots, n
 .\end{equation}
\end{corollary} 
\noindent This gives a geometric proof of the well-known
Clebsch-Gordan rules.  A basis for the space of invariants is induced
from a choice of parenthesization of the tensor product above, see for
example \cite{ca:6j}.  The relation between the different invariants
is also connected to symplectic geometry \cite{ro:as}.
\end{example}

\subsection{Quantization commutes with reduction}  

The proof of the Kempf-Ness Theorem \ref{KN}, which seems otherwise
somewhat miraculous, has a conceptual interpretation given by
Guillemin-Sternberg \cite{gu:ge} in terms of geometric quantization
(Section \ref{gq}) as follows.  Namely, rather than choosing a lift of
$x \in X$ to $V - \{ 0 \}$, which is the total space of $\mO_X(-1)$,
it is more natural from the viewpoint of geometric quantization to
choose a lift $l$ in the positive line bundle $\mO_X(1) \to X$.
Define the {\em Guillemin-Sternberg stability function}
$$ \psi^\dual_l : K \bs G \to \R , \quad g \mapsto \log \Vert g l \Vert^2
/2 .$$
The same computation as in the Kempf-Ness case, except for a change of
sign, implies that the gradient of $\psi_l^\dual$ is minus the moment
map, and $\psi_l$ is concave.  In particular, suppose that $s \in
H^0(X,\mO_X(1))^G$ is an invariant section.  Then
$$ \psi^\dual_{s(x)}([g]) = \log \Vert g s(x) \Vert^2/2 = \log \Vert
s(gx) \Vert^2 /2 .$$
Now concavity of $\psi^\dual_{s(x)}$ implies that any critical point
of $\Vert s \Vert^2$ occurs at $\Phinv(0)$ and is a local maximum, and
$s$ is approximately Gaussian.  This type of behavior is quite
standard for ``typical quantum states'', which introductory physics
lectures often show as concentrating near some submanifold of the
corresponding classical state space in Gaussian fashion.

Suppose that $K$ acts freely on the zero level set $\Phinv(0)$.  The
complex structure $J$ on $X$ induces an almost complex structure $J
\qu K$ on $X \qu K$ by identifying $\pi^* T(X \qu K)$ with the
subbundle of $TX | \Phinv(0)$ perpendicular to the generating vector
fields $\xi_X, \xi \in \k$.  This complex structure is integrable
since the Nijenhuis tensor vanishes. Similarly the polarization
$\mO_X(1) \to X$ naturally descends to a polarization $\mO_{X \qu
  K}(1) \to X \qu K$, defined by restricting to $\Phinv(0)$ and
quotienting by the action of $K$.

\begin{theorem} [Quantization commutes with reduction]  
Let $X$ be a compact Hamiltonian $K$-manifold equipped with moment map
$\Phi: X \to \k^\dual$, polarization $\mO_X(1) \to X$ and a compatible
$K$-invariant K\"ahler structure $J$, such that $K$ acts freely on the
zero level set $\Phinv(0)$, and let $R(X)_d$ denote the space of
sections of $\mO_X(d)$ as above.  For each $d \ge 0$ there is a
canonical isomorphism
$ \rho: R(X)_d^K \to R(X \qu K)_d$
\end{theorem}

\begin{proof}  For smooth projective varieties $X \subset \P(V)$ 
this is a combination of Mumford's Theorem \ref{git} and the
Kempf-Ness Theorem \ref{KN}.  More generally let $X$ be a compact
polarized K\"ahler Hamiltonian $K$-manifold.  Any section $s \in
H^0(X,\mO_{X}(1))^K$ naturally defines a section $ \rho(s) \in H^0(X
\qu K, \mO_{X \qu K}(1))$ by restriction to $\Phinv(0)$ and descent to
the quotient.  Then $\rho$ is an injection, since any invariant
section has maximum norm on $\Phinv(0)$.  Proving surjectivity
required a somewhat complicated argument in the approach of
Guillemin-Sternberg, and the following alternative algebraic argument
is substantially easier: By Kodaira embedding $X$ is biholomorphic to
smooth subvariety of $\P(V)$, and the polarization $\mO_X(1)$ is
isomorphic as a holomorphic line bundle to the pull-back of the
hyperplane bundle on $\P(V)$, although the symplectic structure and
moment map may not be pull-backs.  By the extension of Kempf-Ness to
K\"ahler varieties discussed in \ref{KNkahler}, the semistable locus
corresponding to the polarization $\mO_X(1)$ has quotient by $G$
diffeomorphic to $X \qu K$. Given a section $s \in H^0(X \qu K, \mO_{X
  \qu K}(1))$, $s$ naturally lifts to an invariant section on the
semistable locus $X^{\ss}$ with maximum on $\Phinv(0)$.  Since the
norm of this section is bounded, it extends over all of $X$.
\end{proof} 

Guillemin-Sternberg also proved ``quantization commutes with
reduction'' for another of class of Hamiltonian actions for which
there exists a good quantization scheme, namely cotangent bundles
\cite{gu:ho}.  Quantization commutes with reduction was generalized to
arbitrary compact Hamiltonian manifolds using ``Spin-c'' quantization
by Meinrenken \cite{me:sym}, and further generalized to ``non-abelian
localization'' by Teleman and Paradan, see the last section of these
notes.

\subsection{Convex functions on $K\bs G$}

In this section we further investigate the geometry of the Kempf-Ness
functional, mostly following Kapovich-Leeb-Millson \cite[Section
  3.1]{klm:co}.  The discussion uses some geometry of $K \bs G$ for
which the reader may refer to \cite{bgs:man}.  We already mentioned
that the quotient $K \bs G$ is a {\em Hadamard manifold}, that is, a
space of non-positive sectional curvature.  Such a manifold has a
natural compactification by adding a {\em boundary at infinity}
$\partial_\infty (K \bs G)$, given by equivalence classes of
unit-speed geodesics where two unit-speed geodesics are equivalent if
they have bounded distance.  This boundary is naturally equipped with
a {\em Tits metric} defined as the supremum of angles of formed by a
pair of geodesics approaching the given points at infinity.  The
topology defined by the Tits metric is discrete iff $K \bs G$ is
higher rank, and defines the standard topology on $(T \cap K) \bs T
\subset K \bs G$ for any complex maximal torus $T$.

Let $\psi : K \bs G \to \R$ be a Lipshitz continuous convex function.
The {\em slope at infinity} of $\psi: K \bs G \to \R$ is the function
$$ \mu: \partial_\infty(K \bs G) \to \R,
\xi \mapsto  \lim_{t \to \infty} \frac{\psi(\rho(t))}{t} $$
where $\rho$ is any geodesic ray asymptotic to $\xi$.  By \cite[Lemma
  3.2]{klm:co}, $\mu$ is Lipshitz continuous on $\partial_\infty (K
\bs G)$ with respect to the Tits metric with the same Lipshitz
constant.  The boundary $\partial_\infty(K \bs G)$ has curvature
bounded by $1$; one says that a function on $\partial_\infty(K \bs G)$
is convex if it is convex along any geodesic of length at most $\pi$.

\begin{definition}  Let $C_{< 0}
(x)$ resp. $C_{\leq 0}(x)$ resp. $C_0(x)$ denote the subset of
  $\partial_\infty(K \bs G)$ with negative resp. non-positive
  resp. zero slope.
\end{definition} 

\begin{theorem}  \label{conelem}  
 \begin{enumerate} 
\item $C_{\leq 0}(x)$ is convex, and the function $\mu$ is convex on
  $C_{\leq 0}(x)$ and strictly convex on $C_{< 0}(x)$.
\item $\psi$ is proper and bounded below iff $\mu > 0 $ everywhere
on $\partial_\infty (K \bs G)$. 
\item If $C_{< 0}(x) \neq \emptyset$, then 
\begin{enumerate} 
\item $\mu$ has a unique  minimum $\mu(\xi_{\min})$, 
\item $C_{\leq 0}(x)$ is the closure of $C_{< 0}(x)$, and
\item any gradient trajectory of $\psi$ has asymptotic direction 
 $\xi_{\min}$ and asymptotic slope
  $\mu(\xi_{\min})$.
\end{enumerate} 
\item If $C_{0}(x)$ is open, then $\mu \ge 0 $ everywhere.
\end{enumerate} 
\end{theorem}

\begin{proof} 
Except for the assertion about gradient trajectories, this is Eberlein
\cite[4.1.1']{eb:geom} and Kapovich-Leeb-Millson \cite[Section
  3.1]{klm:co}, and is essentially a consequence of convexity of
$\psi$.  Indeed convexity of $\psi$ implies that if $\xi_1 \in
\partial_\infty(K \bs G)$ is the midpoint of a geodesic segment
connecting $\xi_0,\xi_2 \in \partial_\infty(K \bs G)$ with angles less
than $\pi$ then
$$ \mu(\xi_1) \leq \frac{\mu(\xi_0) + 
\mu(\xi_2)}{2 \cos(d(\xi_0,\xi_2)/2)} $$
and most of the claims follow from this inequality.  Angles of $\pi$
must be dealt with separately; in particular, for example, in the case
$G = SL(2,\C)$ the Tits metric on $G/K$ assigns distance $\pi$ to
every pair of distinct points and so the above argument is not
particularly helpful.

Suppose that $\psi$ has a direction of negative slope.  By Lemma
\ref{gradbound} below, $\Vert \grad(\psi) \Vert$ is bounded below by a
positive constant.  By Caprace-Lytchak \cite[Proposition
  4.2]{caprace:infty} all gradient trajectories converge to the same
point at infinity and at the same rate of escape.  (For the special
case of a Kempf-Ness function the existence of a limiting direction
follows from Duistermaat's result Lemma \ref{conv} and Corollary
\ref{gradpsi}.)  Hence there exists $\xi_\infty \in \k$ such that any
gradient trajectory $[g_t]$ of $- \psi$ has limiting direction
$\xi_\infty$, that is,
$$( \ddt [g_t]) g_t^{-1} \to \ddt |_{t = 0} [ \exp( - i t \xi_\infty)] $$
for some $\xi_\infty \in \k$, then 
\begin{equation} \label{lim} 
\lim_{t \to \infty} [g_t] = \lim_{t \to \infty} [\exp( i \xi_\infty
  t)] \in \partial_\infty(K \bs G) \end{equation}
see Kaimanovich \cite[Theorem 2.1]{kaim:sym}.  (In fact Chen-Sun
\cite{chensun:calabi} show that, in the Kempf-Ness setting, any
gradient trajectory is asymptotic to a geodesic ray.)  It follows from
Lipshitz continuity of $\psi$ that the rate of decay of $\psi$ along
$[g_t]$ is the same as that along $ [ \exp(t i \xi_\infty)]$, so that
$\mu(\xi_{\min}) \leq - \Vert \xi_\infty \Vert$.  If $\xi_\infty \neq
\xi_{\min} \Vert \xi_\infty \Vert$ then one obtains a contradiction by
connecting $[g_t]$ to $[ \exp( - i t \xi_{\min} \Vert \xi_\infty
  \Vert)]$ by a geodesic $[ \exp( i s \zeta_t) g_t ], s \in [0,1]$ and
using convexity of $\psi$: Since $\psi( [ \exp( i s \zeta_t) g_t]$
goes to $-\infty$ at least as fast for $s = 1$ as for $s = 0$ as $t
\to \infty$,
$$ \dds |_{s = 0} \psi( [ \exp( i s \zeta_t) g_t ]) < \eps, \forall
\eps > 0 , t \gg 0 .$$
Now $\grad(\psi) \to \xi_\infty$ implies that $\dds |_{s = 0} \psi( [
  \exp( i s \zeta_t) g_t ])$ is approximately $(\xi_\infty,\zeta_t)$
for $t \gg 0$.  On the other hand, by angle comparison the angle
formed by $ \ddt |_{t = 0} [g_t]$ and $\dds |_{s = 0} \psi( [ \exp( i
  s \zeta_t) g_t ])$ is bounded from below by $\pi/2 + d( \xi_{\min},
\xi_\infty / \Vert \xi_\infty \Vert)/2$ for $t \gg 0$.  This implies
that $(\xi_\infty,\zeta_t) > 0$, which is a contradiction.
\end{proof} 

It remains to show

\begin{lemma}  \label{gradbound}  Let $\psi, \mu, \xi_{\min}$ be 
as in Theorem \ref{conelem}.  Then $\mu(\xi_{\min}) \geq \inf_{[g ]
  \in K \bs G} \Vert \grad(\psi([g])) \Vert$.
\end{lemma} 

\begin{proof}  Convexity of $\psi$ along $[ \exp( i t \xi) g]$
implies that for $t \ge 0$,
\begin{eqnarray*}
 \ddt \psi( [ \exp(  i t \xi) g]) &\ge& \ddt |_{t = 0} \psi( [
  \exp( i t \xi) g]) \\
&=&  ( \xi, \grad(\psi)([g])) \ge - \Vert
\grad(\psi)([g]) \Vert.\end{eqnarray*}
Taking the infimum over $[g] \in K \bs G$ and $\xi \in \k$ of unit
norm gives the result.
\end{proof}  

\begin{remark}   The direction of maximal descent in Theorem 
\ref{conelem} (c) is not necessarily rational.  However, if $\mu$ is
negative somewhere then it negative on some rational vector, since the
Tits metric is the standard one on $(T \cap K) \bs T$ for any maximal
torus $T$ and rational directions are dense in $\t$.  That is,
non-negativity of $\mu$ is equivalent to non-negativity of $\mu$ on
the rational vectors, i.e., those generating one-parameter subgroups.
\end{remark}

\subsection{Polystable points} 
\label{polystable} 

By Lemma \ref{ps}, the polystable orbits are the orbits of points $x
\in \Phinv(0)$.  In this section we investigate these and the
orbit-closure equivalence relation in more detail.  The following was
observed by Kempf-Ness \cite{ke:le} in the linear case and by Slodowy
\cite{sl:op} in general, see also Sjamaar \cite{sj:ho}.

\begin{proposition}   Let $X$ be a K\"ahler Hamiltonian
$K$-manifold, and $x \in \Phinv(0)$. Then $G_x$ is the
  complexification of $K_x$; in particular, $G_x$ is reductive.
\end{proposition} 

\begin{proof}  Suppose that $x \in \Phinv(0)$ and $gx = x$.
Write $g = k \exp(i \xi) $ for some $\xi \in \k, k \in K$.  Let
$\psi_x$ be a Kempf-Ness function for $x$.  Since $x,gx \in
\Phinv(0)$, we have
$$\grad \psi_x ([k \exp(i \xi)]) = \grad \psi_x([ \exp(i \xi)]) =
\grad \psi_x([e]) = 0 .$$
By convexity, $\psi_x$ is constant along the geodesic $[\exp( i t
  \xi)]$, so $\xi \in i \k_x$ by Corollary \ref{convexp}. Hence $x =
kx$ so $k \in K_x$, which implies $g \in (K_x)_\C$.  The reverse
inclusion $(K_x)_\C \subset G_x$ is obvious.
\end{proof} 

\begin{remark} Stabilizer groups are not in general reductive.
For example 
let $X = SL(2,\C) \times_B \P^1$.  Then every stabilizer is either
solvable or unipotent, and so no projective embedding of $X$ has
semistable points.
\end{remark} 

Second we show that polystable points are ``seen by one-parameter
subgroups.'' For this we need to review some results on existence of
holomorphic slices.  Let $X$ be a complex manifold with a holomorphic
action of a group $G$.  Let $x \in X$. Recall that a {\em slice} at
$x$ is a $G_x$-invariant submanifold $S$ of $X$ containing $x$ such
that $GS$ is open in $X$ and the natural $G$-equivariant map from $G
\times_{G_x} S \to X$ is an isomorphism onto $GS$.  Sjamaar
\cite{sj:ho} has proved the following analog of slice theorems of Luna
\cite{lu:sl} and Snow:

\begin{theorem}[Sjamaar] \label{slice} Let $G$ be a connected complex reductive
group with maximal compact $K$.  Let $X$ be a K\"ahler Hamiltonian
$K$-manifold such that the action of $K$ extends to a holomorphic
action of $G$.  Suppose that $x \in \Phinv(0)$.  Then there exists a
slice at $x$.
\end{theorem}

\begin{corollary} \label{1ps} An orbit $Gx$ contains a polystable point 
$y$ in its closure, iff there exists a one-parameter subgroup $\C^*
  \subset G$ and a point $z \in Gx$ such that $\C^* z$ contains a
  polystable point in its closure.
\end{corollary} 

\begin{proof} Let $y$ be a polystable point. We may assume that $\Phi(y) = 0$. 
By Theorem \ref{slice}, there exists a slice $S$ at $y$.  Now $S$ is
biholomorphic to its tangent space $T_y S$, equivariantly for the
action of $K_x$, in a neighborhood $U$ of $y$.  Furthermore, since
this map is holomorphic, the map is equivariant for the {\em
  infinitesimal} $G$-action.  By Lemma \ref{linear}, there exists a
one-parameter subgroup $\C^* \to G$ and a point $v \in T_y S$ such
that the closure of $\C^* v $ contains $0 \in T_y S$.  By choosing $v$
sufficiently small, we ensure that $\{ zv, |z | \leq 1 \}$ is in the
image of $U$.  Let $s \in S$ be the pre-image of $v$.  Then $\{ zs,
|z| \leq 1 \}$ contains $y$ in its closure, as required.
\end{proof}  

\noindent Using this corollary we prove a finite-dimensional analog of
the {\em Jordan-H\"older} theory for semistable vector bundles, see
for example Seshadri \cite{se:fi}.

\begin{definition}  
For any $\lambda \in \k$, let
$x_\lambda = \lim_{t \to  \infty} \exp(-t i \lambda) x$
the {\em associated graded point} of $x$ with respect to $\lambda$.
\end{definition}

\begin{remark} \label{exp}  The fact that $\exp( -t i \lambda)x$
is the gradient flow of a Morse function (see \ref{morse}) implies
that the gradient trajectory converges {\em exponentially fast} to
$x_\lambda$, that is,
$\dist( \exp(
-ti\lambda)x, x_\lambda) \leq C_0 e^{- C_1 t}$
for some constants $C_0,C_1$.
\end{remark}

\begin{definition}  $\lambda \in \k$ is  {\em Jordan-H\"older}
for $x \in X^{\ss}$ iff $x_\lambda$ is polystable.
\end{definition}

\begin{example}   Let $X = \C^2$ and $G = (\C^*)^2$ acting 
by $(g_1,g_2)(z_1,z_2) = (g_1z_1,g_2z_2)$.  Then any
$(\lambda_1,\lambda_2)$ with $\lambda_1,\lambda_2 > 0$ is
Jordan-H\"older.
\end{example} 

\begin{theorem} \label{jh}   Let $X$ be a compact K\"ahler Hamiltonian $K$-manifold and $x \in X$ a semistable point.    
\begin{enumerate} 
\item If $x$ is semistable but not polystable then the set of
  Jordan-H\"older vectors for $x$ is a non-empty $K_x$-invariant cone
  in $\k$.
\item The orbit $Gx_\lambda$ of the associated graded $x_\lambda$ of a
  Jordan-H\"older $\lambda$ is the unique polystable orbit in
  $\ol{Gx}$.
\end{enumerate}
\end{theorem}

\begin{proof} 
(a) Suppose $x$ is semistable but not polystable.  Since $x$ is
  semistable, $Gx$ contains a polystable $y$ in its closure.  By
  Corollary \ref{1ps}, any polystable $y$ is in the closure $\C^* z$
  for some one-parameter subgroup $\C^* \subset G$ and $z \in Gx$.
  Suppose that $z = g^{-1}x$ for some $g \in G$.  Then $(\Ad(g) \C^*)x
  = g \C^* z$ contains $gy$ in its closure, and $gy$ is polystable as
  well.  Convexity of the set of Jordan-H\"older vectors follows from
  Theorem \ref{conelem} applied to a Kempf-Ness function $\psi: K \bs
  G \to \R$.  Indeed, by compactness of $X$ the norm $\Vert \Phi
  \Vert$ is bounded, so by Corollary \ref{gradpsi} $\psi$ is Lipshitz
  continuous. Furthermore, by Lemma \ref{convexp}, $\psi$ is convex.
  Hence Theorem \ref{conelem} applies.
(b) Suppose that $y_0,y_1$ are polystable points in the closure of
  $Gx$.  By Corollary \ref{1ps}, there exist vectors $\lambda_j, j
  =0,1$ and points $x_0,x_1 \in Gx$ such that $y_j =
  (x_j)_{\lambda_j}$.  The distance between $\exp( - i t \lambda_j)
  x_j $ can be estimated as follows: Suppose that $x_j = g_j x$ and
  let $ h_{s,t} = [\exp( i \delta_{s,t} s) \exp( i \lambda_0 t) g_0$
    so that $[h_{s,t}]$ is the geodesic path in $K \bs G$ between
    $[\exp( - i t \lambda_j)g_j], j= 0,1$.  Let $x_{s,t} = h_{s,t} x$.
    The square of the distance from $x_{0,t}$ to $ x_{1,t}$ is given
    by
\begin{eqnarray*}
\left( \int_0^1 \Vert \partial_s x_{s,t} \Vert \d s \right)^2 &\leq&
\int_0^1 \Vert \partial_s x_{s,t} \Vert^2 \d s  = \int_0^1 g \left( \partial_s 
x_{s,t}, \partial_s x_{s,t} \right) \d s 
 \\ &=& \int_0^1 \partial_s^2 \psi ( [h_{s,t}] ) \d s =
\partial_s \psi ([h_{s,t}]) |_{s=0}^{s=1} .
\end{eqnarray*}
Now $\grad \psi$ converges exponentially to zero along $[\exp( -i t
  \lambda_j)]$ as $t \to \infty$ for $j = 0,1$, since $\exp( -i t
\lambda_j)x_j$ converges exponentially fast to $(x_j)_{\lambda_j}$,
see Remark \ref{exp}.  On the other hand, since $h_{s,t}, \exp(-it
\lambda_0)g_0, \exp(-i t \lambda_1)g_1$ are the sides of a geodesic
quadrangle with one side of fixed length, there exist constants
$C_0,C_1$ such that
$$\Vert \delta_t \Vert < C_0 + t C_1 \Vert \lambda_0 + \lambda_1 \Vert $$ 
for all $t$.  Hence
$$\dist(x_{\lambda_0},x_{\lambda_1}) = \lim_{t \to \infty}
\dist(x_{0,t},x_{1,t}) = 0 $$
and the claim follows.
\end{proof}  

\begin{remark} We have included (b) to emphasize a somewhat confusing
point: distant points in $\k$ may map to near points in $X$ if the
gradient of $\psi$ on the path between them is sufficiently small.
\end{remark}

\begin{remark} 
In fact, the full strength of Sjamaar's (or Luna's) slice theorem is
not needed here; it suffices to find a slice for the {\em
  infinitesimal} action of $G$ which is substantially easier.  Some
terminology: If a Lie group with Lie algebra $\g$ acts on a manifold
we say that a submanifold $U$ is {\em $\g$-invariant} if the
generating vector fields are tangent to $U$.  A {\em slice} for the
infinitesimal action of $\g$ at $x$ is a $\g_x$-invariant holomorphic
submanifold $S$ containing $x$, such that the natural map $\g
\times_{\g_x} TS \to TX |S$ is an isomorphism.  Using the implicit
function theorem, one sees that any sequence of points converging to
$x$ may be translated by the action of $G$ (which is now only defined
in a neighborhood of the identity) into a sequence of points in $S$.
Thus if an orbit $Gy$ in $X$ contains $x \in S$ in its closure, then $Gy
\cap S$ also contains $x$ in its closure, and by Lemma \ref{linear}
$\C^* y \cap S$ contains $x$ in its closure for some one-parameter
subgroup $\C^* \subset G$.
\end{remark} 

\section{Schur-Horn convexity and its generalizations}

In this section we discuss the generalization of Clebsch-Gordan theory
to arbitrary groups, in particular, the theory of existence of
invariants in tensor products of representations of $GL(r)$, the
connections (via the Kempf-Ness theorem) with eigenvalue problems, and
a combinatorial answer by Knutson, Tao, and the author \cite{kt:ho2}.

\subsection{The Borel-Weil theorem} 

Let $G$ be a connected complex reductive group.  Let $\lambda$ be any
dominant weight for $G$ and $V_\lambda$ a simple $G$-module with
highest weight $\lambda$.  Let $P_\lambda^-$ be the opposite standard
parabolic corresponding to $\lambda$, and $G/P_\lambda^-$ the
generalized flag variety corresponding to $\lambda$.  We denote by
$\C_\lambda^\dual$ the one-dimensional representation of $P_\lambda^-$
corresponding to $-\lambda$, and by $\mO_X(\lambda) = G
\times_{P^-_\lambda} \C_\lambda^\dual$.

\begin{theorem} [Borel-Weil 
\cite{serre:bour100} ] \label{BW} Let $X = G/P_\lambda^-$ with
  $\lambda$ a  weight.  Then $H^0(X,\mO_X(\lambda)) \cong
  V_\lambda$ if $\lambda$ is dominant and vanishes otherwise.
\end{theorem} 

\begin{proof}   First consider the
case $G = SL(2,\C)$.  We identify $\Lambda^\dual$ with $\Z/2$.  Then
$H^0(\mO_X(\lambda))$ is the set of homogeneous polynomials in two
variables of degree $2\lambda$, if $\lambda$ is non-negative, and zero
otherwise. In the first case one checks easily that
$H^0(\mO_X(\lambda))$ is simple with highest weight $\lambda$.

Next let $G$ be an arbitrary connected complex reductive group.  Let
$X = G/B^-$ and $X_1 = BB^-/B^- \cong B/T \cong U$ the open Bruhat
cell, (here $U$ is a maximal unipotent) so that
$ H^0( X_1, \mO_X(\lambda) | X_1)^U = H^0(U, \C)^U \cong  \C .$
Thus $H^0( X_1, \mO_X(\lambda) | X_1)$ contains a unique highest
weight vector, which we denote by $s_\lambda$.  
We wish to determine  whether $s_\lambda$ extends over the complement of $X_1$ in
$X$.  It suffices to check the order of vanishing of $s_\lambda$ on
the divisors $X_{s_\alpha}$, as $\alpha$ ranges over simple roots.
For each root $\alpha$, we let $h_\alpha \in \t$ denote the
corresponding coroot, so that $\sl(2,\C)_\alpha := \C h_\alpha \oplus
\g_\alpha $ is the three-parameter Lie algebra corresponding to
$\alpha$.  Let $SL(2,\C)_\alpha \to G$ denote the homomorphism induced
by the inclusion $\sl(2,\C)_\alpha \to \g$.  The orbit $C_\alpha =
SL(2,\C)_\alpha B^-/B^-$ of $SL(2,\C)_\alpha$ on $X$ is isomorphic to
$SL(2,\C)_\alpha/SL(2,\C)_\alpha \cap B^- \cong \P^1$.  The curve
$C_\alpha$ intersects the Bruhat cell $X_{s_\alpha}$ in the unique
point $x_{s_\alpha} = s_\alpha B^-/B^-$.  The order of vanishing of
$s_\alpha$ along $X_{s_\alpha}$ is necessarily the order of vanishing
of $s_\alpha | C_\alpha $ at $x_{s_\alpha}$.  Now $\mO_X(\lambda)$
restricts to the line bundle $\mO_{\P^1}( \lan \lambda, h_\alpha
\ran)$ on $C_\alpha$, and the section $s_\lambda$ restricts to the
highest weight section on $C_\alpha - x_{s_\alpha}$.  It extends over
$x_\alpha$ iff $\lan \lambda, h_\alpha \ran \ge 0$, by the discussion
for the $SL(2,\C)$ case.

Now $G/B^-$ fibers over $G/P_\lambda^-$ with projective fibers and so
$$H^0(G/B^-,\mO_{G/B^-}(\lambda)) = H^0(G/P_\lambda,\mO_{G/P_\lambda^-}(\lambda)) .$$ 
Since the result is proved for $G/B^-$, this completes the proof. 
\end{proof} 

From the point of view of symplectic geometry, the Borel-Weil theorem
says that the geometric quantization of a coadjoint orbit equipped
with an integral symplectic form (that is, one that is the curvature
of some line bundle) is a simple $K$-module.  Indeed, let $\Phi$
denote the moment map induced by the action of $K$ on
$\mO_X(\lambda)$.  Since the weight of $T$ on the fiber of
$\mO_X(\lambda)$ over $B^-/B^-$ is $-\lambda$, $\Phi$ maps $X$ onto
the coadjoint orbit $K\lambda$ through $\lambda$, see Proposition
\ref{weight}.  Thus in the notation introduced in Section \ref{gq},
$\H(K\lambda) = V_\lambda$.

%

\subsection{The Schur-Horn-Kostant problem} 

The Schur-Horn theorem \cite{sch:det}, \cite{ho:do} reads:

\begin{theorem}  The set of possible diagonal entries of a Hermitian
operator with eigenvalues $\lambda = (\lambda_1,\ldots,\lambda_n)$ is 
the hull of the set of permutations of $\lambda$. 
\end{theorem}  

\begin{example} 
If $K = SO(3)$ then by Proposition \ref{S2} the coadjoint orbit
through $\diag(\lambda,-\lambda)$ may be identified with the sphere of
radius $\lambda$ via the isomorphism $\k^\dual = \so(3)^\dual \to \R^3$, and
the moment map for the maximal torus action is projection onto the
$z$-axis, and so has moment image $[- \lambda,\lambda]$.  The action
of the Weyl group $W = \Z_2$ on $\t$ is identified with the sign
representation, and so $ [-\lambda,\lambda] = \on{hull}
 \{-\lambda,\lambda \} = \on{hull}(W\lambda)$ as claimed.
\end{example} 

Kostant \cite{ko:iw} generalized this result to arbitrary compact
connected groups:

\begin{theorem} \label{kostant} Let $K$ be a compact connected group.  
The projection of a coadjoint orbit $K \lambda$ of an element $\lambda
\in \t^\dual$ is the convex hull of the orbit $W\lambda$ of $\lambda$
under the Weyl group $W$.
\end{theorem} 

Using the Kempf-Ness and Borel-Weil theorems \ref{KN}, \ref{BW}, the
Schur-Horn-Kostant theorem is equivalent to the following well-known
fact in representation theory:

\begin{theorem} With $K$ as above, 
let $\lambda$ be a dominant weight.  The set of $\mu/d$ such that the
weight space $V_{d\lambda,(\mu)} \subset V_{d\lambda}$ is non-trivial
for some $d \in \Z_+$ is the rational convex hull of $W\lambda$.
\end{theorem}

\begin{proof}  We identify $X = K\lambda = G/P_\lambda^-$ 
and $\C_\mu$ the trivial bundle over $X$ with $T$-weight $\mu$ so that
$V_{d \lambda,(\mu)} = H^0(X,\C_\mu^* \otimes \mO_X(d \lambda))^{T}$
by Borel-Weil \ref{BW}, which is the space of sections over the
quotient $(X \qu T)_\C$ by Mumford's Theorem \ref{git}.  We may use the
Hilbert-Mumford criterion to determine whether there are any
semistable points: Given a one-parameter subgroup generated by some
$\xi \in \t_+$, a point $x \in X$ flows under $\exp(t\xi)$ to $y_w$ as
$t \to -\infty$ where $x \in Y_w : = B^- wB^-/B^-$ is the opposite
Bruhat cell, see \eqref{bruhat}.  The weight of $T$ on the fiber over
$y_w$ is $\mu - w\lambda$.  Thus $x \in Y_w$ is semistable for $\xi$
iff
$ \lan w \lambda - \mu, \xi \ran \leq 0 $
iff $\mu \in w \lambda - \t_+^\dual$.  In particular $Y_{1}$ is
contained in the semistable locus for the one-parameter subgroup
generated by $-\xi$ with $\xi$ dominant iff $\mu \in \lambda -
(\t_+)^\dual$.  The semistable locus for the torus action is non-empty
iff a generic point is semistable for all one-parameter subgroups iff
\begin{equation} \label{cones}
 \mu \in \bigcap_{w \in W}  w ( \lambda - (\t_+)^\dual) .\end{equation} 
The dual cone to $\hull(w \lambda, w \in W)$
at $w \lambda$ is generated by $ (s_\alpha - 1) w \lambda $ where
$\alpha$ ranges over simple roots, which is equal to $w (\t_+)^\dual$.
It follows that \eqref{cones} is equivalent to $\mu \in \hull( w
\lambda, w \in W)$ as claimed.  
\end{proof}  

\begin{proof}[Proof of Theorem \ref{kostant}]
Let $X = K\lambda$ be as above.  The moment map corresponding to the
projective embedding $K\lambda \to \P(V_\lambda^\dual)$ is the
projection $\pi$ of $X$ onto $\t^\dual$ by Proposition \ref{hamops}
\eqref{restr}.  Hence the moment map for the projective embedding
$K\lambda \to \P(V_\lambda^\dual \otimes \C_\mu)$ is $\pi - \mu$.  By
Kempf-Ness $X \qu T_\C \cong X \qu T$, where $T_\C$ is the
complexification of $T$.  Finally $X \qu T$ non-trivial iff $0$ is in
the image of $\pi - \mu$ iff $\mu$ is contained in the image of $\pi$.
\end{proof} 

\subsection{The Horn-Klyachko problem}   

In the previous section we investigated the existence of semistable
points for an action of a torus.  Horn \cite{ho:ei} deals with the
following question, which we will rephrase in terms of existence of
semistable points for the action of a non-abelian group:

\begin{question}  Given the eigenvalues of Hermitian matrices
$H_1,\ldots, H_{n-1}$, what are the possible eigenvalues of $H_1 + \ldots +
  H_{n-1}?$.
\end{question} 

\noindent Since the eigenvalues are real, we may order them in non-increasing
order 
$$ \lambda_1(H_j) \ge \lambda_2(H_j) \ldots \ge \lambda_r(H_j) .$$
The eigenvalues must satisfy the trace equality
$$ \sum_{i,j} \lambda_i(H_j) = \sum_i \lambda_i(H_1 + \ldots + H_{n-1}) .$$ 
After that there are a finite set of linear inequalities, for example
the well-known
$$ \lambda_1(H_1 + H_2) \leq \lambda_1(H_1) + \lambda_1(H_2) .$$
We will describe the complete list.  Before we give the answer, we
note that this question has a symplectic reformulation as follows.
Taking $H_{n} = - H_1 - \ldots - H_{n-1}$, obtain a tuple
$(H_1,\ldots, H_n)$ with $H_1 + \ldots + H_n = 0 $.  Thus the problem
is a special case of the {\em generalized Horn problem}:

\begin{question}  \label{projection}
Let $K$ be a compact Lie group.  For which $\lambda_1,\ldots,
\lambda_n \in \t_+^\dual$ is the symplectic quotient $(K \lambda_1
\times \ldots \times K \lambda_n) \qu K$ non-empty?
\end{question}  

\noindent By the Kempf-Ness and Borel-Weil theorems, this problem is
equivalent to the following

\begin{question}  \label{tensor} 
Let $K$ be a compact Lie group.  For which dominant weights
$\lambda_1,\ldots, \lambda_n \in \t_+^\dual$ is space of invariants
$(V_{d\lambda_1} \otimes \ldots \otimes V_{d\lambda_n})^K$ non-trivial
for some $d \ge 0$?
\end{question}  

\noindent In the case $K = SU(2)$ this question was answered in
Section \ref{cgsec}.  The connection between Questions
\ref{projection} and \ref{tensor} was investigated in the more general
setting of projections of coadjoint orbits by Heckman \cite{he:pr}:

\begin{theorem}  
 \label{heckman}  
Suppose that $L$ is a compact connected group
containing a compact connected subgroup $K$.  The projection of any
$L$-coadjoint orbit $L \mu \subset \l^\dual$ onto $\k^\dual$
intersects $\t_+^\dual$ in a convex polytope.
\end{theorem} 

\noindent In the case $L = K^{n-1}$ containing $K$ via the diagonal
embedding, the projection $(\k^\dual)^{n-1} \to \k^\dual$ is the sum
map and the Theorem \ref{heckman} implies that the for any given
$\lambda_1,\ldots,\lambda_{n-1}$, the set of possible $\lambda_n$ in
\ref{projection} is a convex polytope.

Next we give a partial answer for which inequalities occur in
\ref{projection} in the case $K = SU(n)$ using max-min description of
eigenvalues; this implies inequalities on the invariant theory
problem.  Then we give a necessary and sufficient answer using the
Hilbert-Mumford criterion, following an argument of Klyachko
\cite{kl:lin}.  Finally we give a brief description of works of
Belkale \cite{bl:ip}, Knutson-Tao \cite{kt:sa}, and
Knutson-Tao-Woodward \cite{kt:ho2} giving a {\em minimal set} of
inequalities.  Generalizations to groups of arbitrary type and other
actions are described in Berenstein-Sjamaar \cite{be:coa},
Kapovich-Leeb-Millson \cite{klm:co} and Ressayre \cite{ress:git}.

We begin with the elementary max-min approach for $K = U(n)$.  If $H$
is a Hermitian matrix with eigenvalues $\lambda_1 \ge \lambda_2 \ge
\ldots \ge \lambda_r$ then
$$ \lambda_j = \max_{\stackrel{V \subset \C^r}{\dim(V) = j}} \min_{v \in V - \{ 0
  \}} \frac{ (v,Hv)}{(v,v)}, 
\quad j \in \{ 1,\ldots, r \}
 .$$
This has a generalization to partial sums of eigenvalues as follows:
For every subspace $E \subset \C^r$ and Hermitian operator $H$ we
denote by $H_E$ the operator on $E$ given by composing $H$ with
restriction and projection.  Then for any $J = \{j_1 < \ldots < j_s\}
  \subset \{1,\ldots, r \}$ we have 
$$ \sum_{j \in J} \lambda_j = \max_{\stackrel{ F_1 \subset \ldots
      \subset F_s }{ \dim(F_l) = j_l }} \min_{ \stackrel{E \in
      G(s,n)}{ \dim(E \cap F_l) \ge l } } \Tr( H_E) .$$
Suppose that $J_1,\ldots, J_n$ are such that for every set of flags
$F_1,\ldots, F_n$, there exists a space $E \in G(s,r)$ such that
$\dim(E \cap F_{i,l}) \ge j_{i,l}$ for $i = 1,\ldots, n$ and $l = 1,\ldots,
s$.  Then
\begin{eqnarray*} \sum_{i=1}^n  \sum_{j \in J_i} \lambda_{i,j} 
&=& \sum_{i=1}^n \max_{\stackrel{ F_{i,1} \subset \ldots \subset F_{i,s}
    }{\dim(F_{i,l}) = j_{i,l} }} \min_{ \stackrel{E_i \in G(s,r)}
{\dim(E_i \cap F_{i,l}) \ge l}} \Tr(H_{i,E_i})
\\ &\leq& \sum_{i=1}^n
      \Tr( H_{i,E}) 
= \Tr( \sum_{i=1}^n H_i | E) = 0 .\end{eqnarray*}

\begin{example} Suppose that $J_1 = \{ 1 \}, J_2 = \{ r \}, J_3 = \{ r
\}$.  Since every subspace of dimension $1$ intersects $\C^r$ in a
subspace of dimension $1$, namely itself, we obtain the inequality $
\lambda_{1,1} + \lambda_{2,r} + \lambda_{3,r} \leq 0 .$ In terms of
sums of matrices, this translates to the fact that $\lambda_r(H_1) +
\lambda_r(H_2) \leq \lambda_r(H_1 + H_2)$ for any Hermitian matrices
$H_1,H_2$.
 \end{example} 

The existence of such an $E$ for generic flags is implied by the
non-vanishing of the {\em Schubert coefficient} $ \# [\ol{Y}_{J_1}] \cap
\ldots \cap [\ol{Y}_{J_n}] $ in the homology $H(Gr(s,r))$ of the
Grassmannian $Gr(s,r)$, where $\ol{Y}_{J_i}$ are the {\em Schubert
  varieties} of \eqref{schubert}. (The singular homology has no torsion
and with real coefficients is isomorphic to the de Rham cohomology, so
there is no conflict with notation.)  Thus
\begin{theorem}\label{nec}  If the Horn problem for $\lambda_1,\ldots, \lambda_n$
has a solution, then 
$  \sum_{l=1}^n  \sum_{j \in J_i} \lambda_{i,j}   \leq 0 $
for all $s< r$ and $J_1,\ldots, J_n$ of size $s$ such that $ \#
[\ol{Y}_{J_1}] \cap \ldots \cap [\ol{Y}_{J_n}] > 0$ in $H(\Gr(s,r))$.
\end{theorem} 

Unfortunately, from this point of view it is very difficult to see
whether the list of all such inequalities is sufficient.  Klyachko
\cite{kl:lin} noticed that this follows from the Hilbert-Mumford
criterion. (See Fulton \cite{fu:bu} for a more detailed discussion.)
Let $O_{\lambda_j} = K \lambda_j \cong G/P^-_{\lambda_j}$ for some
{\em dominant} $\lambda_1,\ldots,\lambda_n$; for simplicity we assume
that $\lambda_j$ are generic.  The quotient $(O_{\lambda_1} \times
\ldots \times O_{\lambda_n} )\qu K $ is non-empty iff the semistable
locus in $O_{\lambda_1} \times \ldots O_{\lambda_n}$ is non-empty, iff
a generic point $F = (F_1,\ldots,F_n)$ in $O_{\lambda_1} \times\ldots
\times O_{\lambda_n}$ is semistable for all one-parameter subgroups.
Let $\xi \in \k$ generate a one-parameter subgroup.  Under the action
of $\exp(z \xi), z \to 0$, the point $F_j \in O_{\lambda_j}$ flows to
a $T$-fixed point $x_{w_j}$ where $Y_{w_j}$ contains $F_j$.  Thus $F$
is $\xi$-semistable iff
\begin{equation} \label{xiss}
 \sum_{j=1}^n \lan \lambda_j , w_j^{-1} \xi \ran \leq 0
 .\end{equation}
So $F$ is $\Ad(g) \xi$-semistable iff the same inequalities hold for
$w_j$ such that $F_j \in gY_{w_j}$. Let $g_j \in G$ be such that $F_j
= g_j B/B$.  Then $F_j$ lies in $gY_{w_j}$ iff $g^{-1}B/B \in g_j^{-1}
Y_{w_j^{-1}}$.  Hence the semistable locus for the diagonal action of
$G$ is non-empty iff the inequalities \eqref{xiss} hold for dominant
$\xi$ whenever $(w_1,\ldots,w_n)$ are such that the intersection of
the varieties $g_j^{-1} Y_{w_j^{-1}}$ is non-empty for generic
$(g_1,\ldots,g_n)$. This gives a necessary and sufficient set of
inequalities.  From now on we drop the inverses on the Weyl group
elements $w_j$, since they appear in both the inequalities and the
intersection condition.

The next step is to reduce to inequalities for which the intersection
number $\# [\ol{Y}_{w_1}] \cap \ldots \cap [\ol{Y}_{w_n}]$ is non-zero.  If the
intersection is positive dimensional for generic $(g_1,\ldots,g_n)$
then it represents a non-zero homology class of positive degree, and
by Poincar\'e duality there exists an element $w_{n+1} \in W$ such
that $\# [\ol{Y}_{w_1}] \cap \ldots \cap [\ol{Y}_{w_{n+1}}] \neq 0$.  Then
expanding the product of the last two $[\ol{Y}_{w_n}] \cap [\ol{Y}_{w_{n+1}}]$
and choosing $w_n'$ so that $[\ol{Y}_{w_n'}]$ has positive coefficient in
$[\ol{Y}_{w_n}] \cap [\ol{Y}_{w_{n+1}}]$ one obtains $w_n'$ such that $\#
[\ol{Y}_{w_1}] \cap \ldots \cap [\ol{Y}_{w_{n'}}] \neq 0$.  Then $w_{n} \lambda
- w_n' \lambda \in  \t_+$ and so the inequality for
$(w_1,\ldots,w_n')$ implies that for $(w_1,\ldots,w_n)$.  The
conclusion is that a generic point is semistable iff
$$ \# [\ol{Y}_{w_1}] \cap \ldots \cap [\ol{Y}_{w_n}] > 0 \implies \sum_{l=1}^n
\lan \lambda_l , w_l \xi \ran \leq 0 \quad \forall \xi \in \t_+.$$
It suffices to check the inequalities for $\xi$ in a set of generators
for $\t_+$. In particular, for $K$ semisimple it suffices to check
them for $\xi$ equal to a fundamental coweight $\omega_j^\dual$, that
is, for a generator of $\t_+$.  An argument similar to the one above
shows that these inequalities correspond to non-zero intersection
numbers in the corresponding generalized partial flag varieties:

\begin{theorem} Let $K$ be a compact connected semisimple 
group with complexification $G$.  A necessary and sufficient set of
inequalities for the Horn-Klyachko problem is given by
$$ \# [\ol{Y}_{w_1}] \cap \ldots \cap [\ol{Y}_{w_n}] > 0 \implies \sum_{l=1}^n
\lan \lambda_l , w_l \omega_l^\dual \ran \leq 0 \quad \forall \xi \in
\t_+.$$
as $\omega_l^\dual$ ranges over fundamental coweights, $[w_1],\ldots,
[w_n]$ range over elements of $W/W_{\omega_j}$, $Y_{w_1},\ldots,
Y_{w_n} \subset G/P_{\omega_j}$ are the corresponding opposite Bruhat
cells in the partial flag variety $G/P_{\omega_j}$, with the condition
that $\# [\ol{Y}_{w_1}] \cap \ldots \cap [ \ol{Y}_{w_n}] \neq 0$ in
$H(G/P_{\omega_j})$.
\end{theorem}  

For example, suppose that $K = U(r)$ (and Klyachko's argument was
restricted to this case) so that $\t$ is naturally identified with
$\R^n$ and the $j$-th fundamental weight is identified with $\omega_j
= e_1 + \ldots + e_j$, where $e_j$ is the $j$-th standard basis
vector.  In this case one obtains that $(O_{\lambda_1} \times \ldots
\times O_{\lambda_n}) \qu K$ is non-empty iff for each $j \in \{1,\ldots,
r\}$ and subsets $J_1,\ldots,J_n \subset \{ 1,\ldots, r \}$ of size
$k$,
$$ \# ([\ol{Y}_{J_1}] \cap \ldots \cap [\ol{Y}_{J_n}]) > 0 \implies \sum_{l=1}^n
\sum_{j \in J_l} \lambda_{l,j} \leq 0 $$
c.f. Theorem \ref{nec}.  So the Hilbert-Mumford approach implies the
sufficiency as well as the necessity of these inequalities.

The cohomology of the Grassmannian $G(s,r)$ has a number of
combinatorial models, for example, the famous Littlewood-Richardson
rule.  A recent ``puzzles'' model introduced by Knutson and Tao, see
\cite{kt:ho2}, is simple enough that we give a brief description.   The
     {\em puzzle board} is the diagram shown in Figure \ref{board}.
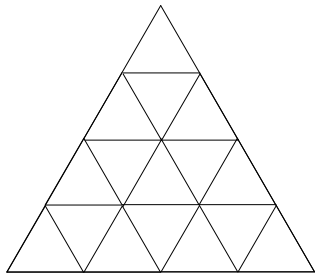
\begin{figure}[h]
\setlength{\unitlength}{0.00033333in}
\begingroup\makeatletter\ifx\SetFigFont\undefined%
\gdef\SetFigFont#1#2#3#4#5{%
  \reset@font\fontsize{#1}{#2pt}%
  \fontfamily{#3}\fontseries{#4}\fontshape{#5}%
  \selectfont}%
\fi\endgroup%
{\renewcommand{\dashlinestretch}{30}
\begin{picture}(4824,4202)(0,-10)
\path(12,18)(4812,18)
\path(2412,4175)(12,18)
\path(4812,18)(2412,4175)
\path(3617,18)(4812,18)
\path(4215,1053)(3617,18)
\path(4812,18)(4215,1053)
\path(613,1062)(1808,1062)
\path(13,12)(1208,12)
\path(611,1047)(13,12)
\path(1208,12)(611,1047)
\path(1213,12)(2408,12)
\path(1811,1047)(1213,12)
\path(2408,12)(1811,1047)
\path(1807,1068)(1210,2103)
\path(1812,1068)(3007,1068)
\path(3007,1068)(2410,2103)
\path(3017,1068)(4212,1068)
\path(2417,18)(3612,18)
\path(3015,1053)(2417,18)
\path(3612,18)(3015,1053)
\path(3607,2100)(3009,1065)
\path(4208,1058)(3611,2093)
\path(1827,3123)(3022,3123)
\path(1227,2073)(2422,2073)
\path(2422,2073)(1825,3108)
\path(2427,2073)(3622,2073)
\path(3025,3108)(2427,2073)
\path(3622,2073)(3025,3108)
\path(1808,3129)(1210,2094)
\path(2418,2070)(1820,1035)
\path(1207,2096)(609,1061)
\end{picture}
}
\caption{Puzzle board}
\label{board}
\end{figure}
There are $r$ little triangles along each big edge in the board.  The
{\em puzzle pieces} are shown in Figure \ref{pieces}. 
\begin{figure}[h]
\setlength{\unitlength}{0.00033333in}
\begingroup\makeatletter\ifx\SetFigFont\undefined%
\gdef\SetFigFont#1#2#3#4#5{%
  \reset@font\fontsize{#1}{#2pt}%
  \fontfamily{#3}\fontseries{#4}\fontshape{#5}%
  \selectfont}%
\fi\endgroup%
{\renewcommand{\dashlinestretch}{30}
\begin{picture}(4818,2097)(0,-10)
\path(12,435)(1207,435)
\path(610,1470)(12,435)
\path(1207,435)(610,1470)
\put(360,840){\makebox(0,0)[lb]{{\SetFigFont{12}{14.4}{\rmdefault}{\mddefault}{\updefault}0}}}
\put(566,478){\makebox(0,0)[lb]{{\SetFigFont{12}{14.4}{\rmdefault}{\mddefault}{\updefault}0}}}
\put(769,839){\makebox(0,0)[lb]{{\SetFigFont{12}{14.4}{\rmdefault}{\mddefault}{\updefault}0}}}
\path(1796,420)(2991,420)
\path(2394,1455)(1796,420)
\path(2991,420)(2394,1455)
\put(2363,451){\makebox(0,0)[lb]{{\SetFigFont{12}{14.4}{\rmdefault}{\mddefault}{\updefault}1}}}
\put(2137,831){\makebox(0,0)[lb]{{\SetFigFont{12}{14.4}{\rmdefault}{\mddefault}{\updefault}1}}}
\put(2564,831){\makebox(0,0)[lb]{{\SetFigFont{12}{14.4}{\rmdefault}{\mddefault}{\updefault}1}}}
\path(4209,2070)(3611,1035)
\path(4806,1035)(4209,2070)
\path(4209,12)(3611,1047)
\path(4806,1047)(4209,12)
\put(3911,1347){\makebox(0,0)[lb]{{\SetFigFont{12}{14.4}{\rmdefault}{\mddefault}{\updefault}1}}}
\put(4415,1352){\makebox(0,0)[lb]{{\SetFigFont{12}{14.4}{\rmdefault}{\mddefault}{\updefault}0}}}
\put(3893,681){\makebox(0,0)[lb]{{\SetFigFont{12}{14.4}{\rmdefault}{\mddefault}{\updefault}0}}}
\put(4462,681){\makebox(0,0)[lb]{{\SetFigFont{12}{14.4}{\rmdefault}{\mddefault}{\updefault}1}}}
\end{picture}
}
\caption{Puzzle pieces}
\label{pieces}
\end{figure}
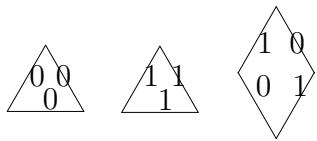
together with their rotations.  A {\em puzzle} is a way of filling in
the puzzle board with puzzle pieces so that all of the edges match.

\bex  \label{puzzle} An example of a puzzle is shown in Figure \ref{example}. \eex

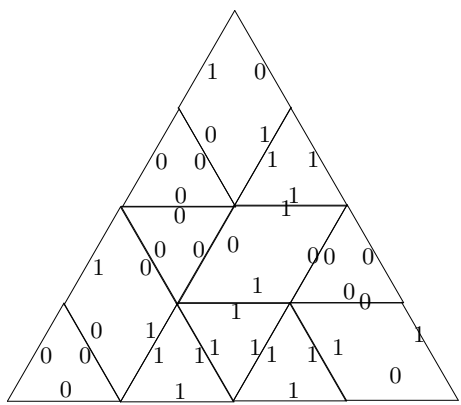
\begin{figure}[h]
\setlength{\unitlength}{0.00063333in}
\begingroup\makeatletter\ifx\SetFigFont\undefined%
\gdef\SetFigFont#1#2#3#4#5{%
  \reset@font\fontsize{#1}{#2pt}%
  \fontfamily{#3}\fontseries{#4}\fontshape{#5}%
  \selectfont}%
\fi\endgroup%
{\renewcommand{\dashlinestretch}{30}
\begin{picture}(3707,3239)(0,-10)
\path(1870,3212)(1407,2411)
\path(2332,2411)(1870,3212)
\path(1870,1619)(1407,2420)
\path(2332,2420)(1870,1619)
\put(1639,2652){\makebox(0,0)[lb]{{\SetFigFont{9}{10.8}{\rmdefault}{\mddefault}{\updefault}1}}}
\put(2029,2656){\makebox(0,0)[lb]{{\SetFigFont{9}{10.8}{\rmdefault}{\mddefault}{\updefault}0}}}
\put(1625,2137){\makebox(0,0)[lb]{{\SetFigFont{9}{10.8}{\rmdefault}{\mddefault}{\updefault}0}}}
\put(2066,2137){\makebox(0,0)[lb]{{\SetFigFont{9}{10.8}{\rmdefault}{\mddefault}{\updefault}1}}}
\path(12,17)(937,17)
\path(475,819)(12,17)
\path(937,17)(475,819)
\put(281,331){\makebox(0,0)[lb]{{\SetFigFont{9}{10.8}{\rmdefault}{\mddefault}{\updefault}0}}}
\put(441,51){\makebox(0,0)[lb]{{\SetFigFont{9}{10.8}{\rmdefault}{\mddefault}{\updefault}0}}}
\put(598,330){\makebox(0,0)[lb]{{\SetFigFont{9}{10.8}{\rmdefault}{\mddefault}{\updefault}0}}}
\path(2324,823)(3249,823)
\path(2787,1624)(2324,823)
\path(3249,823)(2787,1624)
\put(2593,1136){\makebox(0,0)[lb]{{\SetFigFont{9}{10.8}{\rmdefault}{\mddefault}{\updefault}0}}}
\put(2753,856){\makebox(0,0)[lb]{{\SetFigFont{9}{10.8}{\rmdefault}{\mddefault}{\updefault}0}}}
\put(2910,1135){\makebox(0,0)[lb]{{\SetFigFont{9}{10.8}{\rmdefault}{\mddefault}{\updefault}0}}}
\path(1864,1613)(2790,1613)
\path(2327,2415)(1864,1613)
\path(2790,1613)(2327,2415)
\put(2303,1637){\makebox(0,0)[lb]{{\SetFigFont{9}{10.8}{\rmdefault}{\mddefault}{\updefault}1}}}
\put(2128,1931){\makebox(0,0)[lb]{{\SetFigFont{9}{10.8}{\rmdefault}{\mddefault}{\updefault}1}}}
\put(2459,1931){\makebox(0,0)[lb]{{\SetFigFont{9}{10.8}{\rmdefault}{\mddefault}{\updefault}1}}}
\path(952,1605)(1878,1605)
\path(1415,2407)(952,1605)
\path(1878,1605)(1415,2407)
\put(1222,1919){\makebox(0,0)[lb]{{\SetFigFont{9}{10.8}{\rmdefault}{\mddefault}{\updefault}0}}}
\put(1381,1639){\makebox(0,0)[lb]{{\SetFigFont{9}{10.8}{\rmdefault}{\mddefault}{\updefault}0}}}
\put(1538,1918){\makebox(0,0)[lb]{{\SetFigFont{9}{10.8}{\rmdefault}{\mddefault}{\updefault}0}}}
\path(941,1610)(1867,1610)
\path(1404,809)(941,1610)
\path(1867,1610)(1404,809)
\put(1211,1192){\makebox(0,0)[lb]{{\SetFigFont{9}{10.8}{\rmdefault}{\mddefault}{\updefault}0}}}
\put(1370,1472){\makebox(0,0)[lb]{{\SetFigFont{9}{10.8}{\rmdefault}{\mddefault}{\updefault}0}}}
\put(1527,1193){\makebox(0,0)[lb]{{\SetFigFont{9}{10.8}{\rmdefault}{\mddefault}{\updefault}0}}}
\path(937,1605)(474,804)
\path(1399,804)(937,1605)
\path(937,12)(474,813)
\path(1399,813)(937,12)
\put(706,1046){\makebox(0,0)[lb]{{\SetFigFont{9}{10.8}{\rmdefault}{\mddefault}{\updefault}1}}}
\put(1096,1050){\makebox(0,0)[lb]{{\SetFigFont{9}{10.8}{\rmdefault}{\mddefault}{\updefault}0}}}
\put(692,530){\makebox(0,0)[lb]{{\SetFigFont{9}{10.8}{\rmdefault}{\mddefault}{\updefault}0}}}
\put(1133,530){\makebox(0,0)[lb]{{\SetFigFont{9}{10.8}{\rmdefault}{\mddefault}{\updefault}1}}}
\path(3695,25)(3232,827)
\path(2770,25)(3695,25)
\path(2315,823)(3240,822)
\path(2777,21)(2315,823)
\put(3326,505){\makebox(0,0)[lb]{{\SetFigFont{9}{10.8}{\rmdefault}{\mddefault}{\updefault}1}}}
\put(3134,165){\makebox(0,0)[lb]{{\SetFigFont{9}{10.8}{\rmdefault}{\mddefault}{\updefault}0}}}
\put(2886,775){\makebox(0,0)[lb]{{\SetFigFont{9}{10.8}{\rmdefault}{\mddefault}{\updefault}0}}}
\put(2665,394){\makebox(0,0)[lb]{{\SetFigFont{9}{10.8}{\rmdefault}{\mddefault}{\updefault}1}}}
\path(1405,823)(2331,823)
\path(1868,1623)(1405,823)
\path(2785,1619)(2323,818)
\path(1860,1619)(2785,1619)
\put(2005,902){\makebox(0,0)[lb]{{\SetFigFont{9}{10.8}{\rmdefault}{\mddefault}{\updefault}1}}}
\put(1807,1238){\makebox(0,0)[lb]{{\SetFigFont{9}{10.8}{\rmdefault}{\mddefault}{\updefault}0}}}
\put(2459,1149){\makebox(0,0)[lb]{{\SetFigFont{9}{10.8}{\rmdefault}{\mddefault}{\updefault}0}}}
\put(2238,1530){\makebox(0,0)[lb]{{\SetFigFont{9}{10.8}{\rmdefault}{\mddefault}{\updefault}1}}}
\path(1859,21)(2784,21)
\path(2322,823)(1859,21)
\path(2784,21)(2322,823)
\put(2298,45){\makebox(0,0)[lb]{{\SetFigFont{9}{10.8}{\rmdefault}{\mddefault}{\updefault}1}}}
\put(2123,340){\makebox(0,0)[lb]{{\SetFigFont{9}{10.8}{\rmdefault}{\mddefault}{\updefault}1}}}
\put(2454,340){\makebox(0,0)[lb]{{\SetFigFont{9}{10.8}{\rmdefault}{\mddefault}{\updefault}1}}}
\path(1392,818)(2317,818)
\path(1855,17)(1392,818)
\path(2317,818)(1855,17)
\put(1831,689){\makebox(0,0)[lb]{{\SetFigFont{9}{10.8}{\rmdefault}{\mddefault}{\updefault}1}}}
\put(1656,395){\makebox(0,0)[lb]{{\SetFigFont{9}{10.8}{\rmdefault}{\mddefault}{\updefault}1}}}
\put(1987,395){\makebox(0,0)[lb]{{\SetFigFont{9}{10.8}{\rmdefault}{\mddefault}{\updefault}1}}}
\path(937,12)(1863,12)
\path(1400,813)(937,12)
\path(1863,12)(1400,813)
\put(1376,36){\makebox(0,0)[lb]{{\SetFigFont{9}{10.8}{\rmdefault}{\mddefault}{\updefault}1}}}
\put(1201,330){\makebox(0,0)[lb]{{\SetFigFont{9}{10.8}{\rmdefault}{\mddefault}{\updefault}1}}}
\put(1532,330){\makebox(0,0)[lb]{{\SetFigFont{9}{10.8}{\rmdefault}{\mddefault}{\updefault}1}}}
\end{picture}
}
\caption{An example of a puzzle}
\label{example}
\end{figure}

For each puzzle, let $I$ denote the positions of the $1$'s on the
northwest boundary, $J$ the positions of the $1$'s on the northeast
boundary, and $K$ the positions of the edge along the southern
boundary, reading left to right.

\bex For Figure \ref{example}, 
$ I = \{ 2,4 \}, \  J = \{ 2,4 \}, \  K = \{ 2,3 \} .$
\eex

\begin{theorem} \cite{kt:ho2}  The coefficient of $[\ol{Y}_K]$ in $[\ol{Y}_I] \cap [\ol{Y}_J] 
\in H(G(s,r))$ is the number of puzzles $n_{IJ}^K$ with boundary data
$I,J,K$.
\end{theorem} 

\noindent There are several possible proofs: one given by Knutson and
Tao checks the equivalence with the Littlewood-Richardson rule.  A
second proof \cite{ktw:lr}, joint with the author, proves
associativity of the product defined by the puzzle numbers by a simple
combinatorial trick, and then checks equality with the Schubert
coefficients on generators.  The formula generalizes to intersection
numbers of arbitrary numbers of Schubert varieties, by considering
puzzle boards with arbitrary numbers of ``large boundaries''.  For
example, for $n = 4$ one can take a diamond-shaped puzzle board.

Combining this combinatorial description with Klyachko's argument
gives the following:

\begin{corollary}  
If there is a puzzle whose $1$'s on the boundary are in positions
$I,J,K$ then the inequality
$$ \sum_{i \in I} \lambda_i(H_1) + \sum_{j \in J} \lambda_j(H_2)
\leq \sum_{k \in K} \lambda_k(H_1 + H_2) $$
holds for any Hermitian matrices $A,B$, and these inequalities together
with the trace equality 
$$ \sum_{i=1}^n \lambda_i(H_1) + \sum_{j =1}^n \lambda_j(H_2)
= \sum_{k=1}^n \lambda_k(H_1 + H_2) $$
give sufficient conditions for a triple $(\lambda(H_1), \lambda(H_2),
\lambda(H_1 + H_2))$ to occur.
\end{corollary} 

\bex The puzzle in Example \ref{puzzle} gives the inequality
$ \lambda_2(H_1) + \lambda_4(H_1) + \lambda_2(H_2) + \lambda_4(H_2)
\leq \lambda_2(H_1 + H_2) + \lambda_3 (H_1 + H_2) .$
\eex

The following theorem of Knutson, Tao, and the author \cite{kt:ho2}
(see also the review \cite{kn:rev}), extending previous work of
Belkale \cite{bl:ip}, describes a {\em minimal} set of inequalities:

\begin{theorem} 
 The inequalities corresponding to $I,J,K$ with $n_{IJ}^K = 1$
 together with the trace equality
form a complete and irredundant set of necessary and sufficient
conditions for the Horn problem for the sum of two Hermitian matrices.
\end{theorem}

Many other problems of this type can be solved in the same way; for
example see Agnihotri-Woodward \cite{ag:ei} for a discussion of the
possible eigenvalues of a product of unitary matrices, and relations
with the invariant theory of {\em quantum} groups.  In this case the
existence of a good combinatorial model computing the eigenvalue
inequalities is still open.  

\section{The stratifications of Hesselink, Kirwan, and Ness} 
\label{hkn}

According to work of Kirwan \cite{ki:coh} and Ness \cite{ne:st}, the
semistable locus of a $G$-variety $X \subset \P(V)$ can be considered
the open stratum in a Morse-type stratification of $X$.  A theorem of
Ness describes the equivalence of this stratification with one
introduced by Hesselink \cite{hess:uni}, which measures the {\em
  slope of instability} of a point by its maximal Hilbert-Mumford
weight.

\subsection{The Kirwan-Ness stratification} 
\label{kn}

Let $X$ be a Hamiltonian $K$-manifold with proper moment map $\Phi: X
\to \k^\dual$.  Let $( \ , \ ): \k \to \k \to \R$ be an invariant metric
on $\k$ inducing an identification $\k \to \k^\dual$.  Let
$$ \phi = \hh  ( \Phi, \Phi ) : X \to \R $$
denote the norm-square of the moment map.  The notation $\Phi(x)_X \in
\Vect(X)$ denotes the vector field determined by $\Phi(x)$, and
$\Phi(x)_X(x) \in T_x X$ its evaluation at $x$.

\begin{lemma} \label{crit}  
$\crit(\phi) = \{ x \in X, \Phi(x)_X(x) = 0 \}$.
\end{lemma}

\begin{proof}  We have 
$ \d \phi (x) = (\Phi(x), \d \Phi(x)) = - \iota_{\Phi(x)_X(x)}
  \omega_x .$
Since $\omega$ is non-degenerate, $\d \phi(x)$ vanishes iff
$\Phi(x)_X(x) \in T_x X$ does.
\end{proof}  

\begin{example} Let $X = \P^2$ and $K = U(1)^2$ acting by 
$(g_1,g_2)[z_0,z_1,z_2] = [z_0,g_1^{-1} z_1, g_2^{-1} z_2]$.  Consider the moment
  map 
$$\Phi([z_0,z_1,z_2]) \mapsto ( |z_1|^2/2 , |z_2|^2/2) -
  (1/4,1/4) ,$$ 
which has image the convex hull
$$\Delta(X) = \hull \{ (-1/4,-1/4), (-1/4,3/4),
  (3/4,-1/4) \} .$$ 
The critical sets are the level sets of $\Phi$ at $(0,0),(-1/4,0),
(0,-1/4), (1/4,1/4)$, $(-1/4,-1/4), (-1/4,3/4), (3/4,-1/4)$, see Figure
\ref{p2crit}.
\end{example}  

\begin{figure}[h]
\includegraphics[height=1.5in]{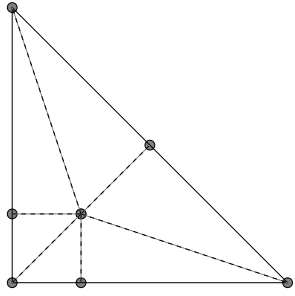}
\caption{Critical values for $X = \P^2$}
\label{p2crit}
\end{figure}

\begin{lemma}  $\Phi(\crit(\phi))$ is a discrete union of $K$-orbits
in $\k^\dual$, called the set of {\em types} for $X$.
\end{lemma}  
 
\begin{proof}  Suppose first that $K$ is abelian.   Consider
the {\em orbit-type decomposition}
$$ X = \bigcup_{H \subset K} X_H , \quad X_H = \{ x \in X | K_x = H \} .$$
where the union is over subgroups $H \subset K$.  It follows from
standard slice theorems that each $X_H$ is a smooth manifold.  Let
$\h$ denote the Lie algebra of $H$. By Lemma \ref{redlem}, $ \Phi(X_H)
$ is an open subset of an affine subspace parallel to $\ann(\h)$. Thus
$\Phi(X_H \cap \crit(\phi)) = \{ \lambda \in \Phi(X_H) | \lambda \in
\h \}$ which is the set containing the unique point in $\Phi(X_H)$
closest to $0$, if it exists, and empty, otherwise.  Since $\Phi$ is
proper, the pre-image of any compact set under $\Phi$ contains only
finitely many orbit-types, which proves the theorem in the abelian
case.

Suppose that $K$ is possibly non-abelian with maximal torus $T$.  The
action of $T$ on $X$ is also Hamiltonian with moment map $\Phi_T$
obtained by composing $\Phi$ with the projection of $\k^\dual$ onto
$\t^\dual$.  Let $\phi_T = (\Phi_T,\Phi_T)/2$.  Since $\phi$ is
$K$-invariant, any critical point is conjugate to a point $x \in
\crit(\phi)$ with $\Phi(x) \in \t^\dual$.  Then $x \in \crit(\phi)$
iff $x \in \crit(\phi_T)$ iff $\Phi(x)$ is a type for the action of
$T$.  Hence the types for $K$ are locally finite.
\end{proof} 
 
Choose a compatible $K$-invariant metric on $X$, and let
$\grad(\phi) \in \Vect(X)$ denote the gradient of $\phi$.

\begin{lemma}  \label{gradphi}
The gradient of $\phi$ is $\grad(\phi)(x) =  J(x) \Phi(x)_X(x)$.  
\end{lemma} 

\begin{proof}  Using the proof of Lemma \ref{crit}, for $v \in T_x X$
$$ g_x(\grad(\phi)(x),v) = D_x \phi(v) = - \omega_x(\Phi(x)_X(x), v) = 
  g_x(J(x) \Phi(x)_X(x),v) .$$
The claim follows. 
\end{proof}  

Let $\varphi_t: X \to X$ be the flow of $-\grad(\phi)$; since $\Phi$
is proper, so is $\phi$ and so $\varphi_t$ exists for all times $t \in
[0,\infty)$.  

\begin{proposition}[Duistermaat, see \cite{le:gra}, \cite{wo:norm}]
\label{conv}
Any trajectory of $\varphi_t$ has a limit.
\end{proposition} 

For the construction of the Kirwan stratification the actual
convergence of $\varphi_t$ is not needed.  
For each type $\lambda$,
let $C_\lambda = \Phinv(K\lambda) \cap \crit(\phi)$ denote the
corresponding component of the critical set of $\phi$.  Since the set
of types is discrete, any two limit points are contained in some
$C_\lambda \subset \crit(\phi)$, and in fact in the same connected
component of $\crit(\phi)$.  Let $X_\lambda$ denote the set of points
$x \in X$ flowing to $C_\lambda$,
$$ X_\lambda := \{ \ol{ \{ \varphi_t(x) , t \in [0,\infty) \}} \cap
  C_\lambda \neq \emptyset .\} .$$
Note that since $\phi$ is not Morse-Bott in general, there is no
guarantee that $X_\lambda$ is smooth.  The {\em Kirwan-Ness
  stratification} is the decomposition \cite{ki:coh}, \cite{ne:st}:
$$ X = \bigcup_{\lambda} X_\lambda .$$
\begin{theorem} [Kirwan]
There exists an invariant metric on $X$ so that each stratum
$X_\lambda$ is smooth.  
The spectral sequence for the equivariant
stratification $X = \cup_\lambda X_\lambda$ collapses at the second
page, so that
$$ H_K(X) \cong \bigoplus_\lambda H_K(X_\lambda) .$$
\end{theorem} 
\noindent In particular the canonical map $H_K(X) \to H_K(\Phinv(0))$ (which is
isomorphic to $H(X \qu K)$ if $K$ acts freely on $\Phinv(0)$) is a
surjection and the equivariant Poincar\'e polynomial of $X$
$$ p_X^K(t) = \sum t^j \rank H^j_K(X) $$
is given by 
$$ p_X^k(t) = \sum_\lambda (-1)^{\codim(X_\lambda)} p_{X_\lambda}^K(t) .$$
If $X$ acts freely on $\Phinv(0)$ this means that the difference $
p_X^K(t) - p_{X \qu K}(t) $ is a finite sum of contributions from
fixed point sets of one-parameter subgroups.  We will see a version of
this formula for sheaf cohomology in the last chapter.

In the case that $X$ is a K\"ahler Hamiltonian $K$-manifold with
proper moment map, the Kirwan-Ness stratification has a more explicit
description.  For each type $\lambda$ let $\varphi_{\lambda,t}$ denote
the time $t$ flow of $- \grad \lan \Phi, \lambda \ran$, $Z_\lambda$
the component of the fixed point set $X^\lambda$ of the action of
$\lambda$ containing $C_\lambda$, $Y_\lambda$ the subset of $X$
flowing to $Z_\lambda$ under $\varphi_{\lambda,t}$, $K_\lambda$ the
centralizer of $\lambda$, and $U(1)_\lambda$ the one-parameter
subgroup generated by $\lambda$.  Then $K_\lambda/U(1)_\lambda$ acts
naturally on $Z_\lambda$ in Hamiltonian fashion with moment map
denoted $\Phi_\lambda$, obtained by restricting $\Phi$ to $Z_\lambda$
and projecting out the direction generated by $\lambda$.  We denote by
$Z_\lambda^{\ss}$ the set of points flowing to $\Phinv_\lambda(0)$
under the flow of minus the gradient of the norm-square of
$\Phi_\lambda$.  Let $Y_\lambda^{\ss}$ denote the inverse image of
$Z_\lambda^{\ss}$ in $Y_\lambda$.

\begin{theorem} [Kirwan \cite{ki:coh}]  \label{kirwan}
Let $X$ be a compact K\"ahler Hamiltonian $K$-manifold with proper
moment map $\Phi: X \to \k^\dual$.  For the K\"ahler metric each
$X_\lambda$ is a $G$-invariant complex submanifold, each $Y_\lambda$
is a $P_\lambda$-invariant complex submanifold, and $G
\times_{P_\lambda} Y_\lambda^{\ss} \to X_\lambda, \, [g,y] \mapsto gy$
is an isomorphism of complex $G$-manifolds.
\end{theorem} 

We give a proof, and explain the relation with a theorem of Ness
\cite{ne:st}, in the following section.  In the point of view we will
present, a key fact is that the gradient flow of the norm-square of
the moment map is essentially equivalent to the gradient flow of the
Kempf-Ness function, as was pointed out in Donaldson-Kronheimer
\cite[Section 6]{do:fo}. Let $X$ be a K\"ahler Hamiltonian
$K$-manifold with proper moment map.  For any $x \in X$, let $x_t$
denote the trajectory of the gradient flow of $-\phi$ starting at $x$.
On the other hand, let $\psi: K \bs G \to \R$ be the Kempf-Ness function
for $x$.  We may also consider the gradient flow of $\psi$, with
respect to the given metric on $\k$.

\begin{proposition} \label{compare} Let $X,x,\psi$ be as above.  
The map 
$$K \bs G \to X/K, \, [g] \mapsto [gx]$$ 
maps the gradient trajectories of $ \psi$ onto the image of the
gradient trajectories of $\phi$ mod $K$.
\end{proposition} 

\begin{proof}
Using Corollary \ref{gradpsi},
\begin{eqnarray*} 
 J_X (( \grad (\psi) )([g]))_X(gx) &=& J_X (\Phi(gx))_X(gx) \\ &=&
 (\grad(\phi))(gx) .\end{eqnarray*}
The vector field on $G$ given by $g \mapsto i \grad(\psi)([g])$ has
trajectories that map to the gradient trajectories of $\grad(\psi)$
under $G \to K \bs G$, and to the trajectories of $\grad(\phi)$ under
$G \to X$, which gives the result.
\end{proof} 

\noindent In particular, since the trajectories of $\psi$ exist for
all time by the bound on $\Phi$, any trajectory of $-\grad(\phi)$ is
contained in a single $G$-orbit: $x_t \in Gx, \forall x \in X, t \in
\R$.

\begin{corollary} \label{bb} $\psi$ is bounded from below iff 
the gradient flow for $ - \phi$ converges to $\Phinv(0)$.
\end{corollary} 

\begin{proof} In the algebraic case, this is nothing but a reformulation 
of \ref{linear}.  For the K\"ahler case, note that if $\psi$ is
bounded from below then $\grad(\psi)$ converges to zero along any
gradient trajectory, and by equivalence of gradient flows
\ref{compare} it follows that $\Phi$ must converge to zero.  The
converse follows as in the proof of Theorem \ref{jh}, using that
$\grad(\psi)$ converges to zero exponentially fast along any
one-parameter subgroup whose limit corresponds to a polystable
point.  \end{proof}

One obtains an analytic proof of the Hilbert-Mumford
criterion Theorem \ref{HN} by combining Corollary \ref{bb}
and Theorem \ref{conelem}.  

\begin{remark}   A rather confusing point is that if $\psi: K \bs G \to \R$
is a Kempf-Ness function, then the slope function $ (K \bs G)_\infty
\to \R$ of Theorem \ref{conelem} is not continuous in the topology on
$(K \bs G)_\infty $ induced by the identification with the unit sphere
in the Lie algebra $\k$, but rather only in the topology induced by
the Tits metric.  This happens already for the action of $SL(2,\C)$ on
$\P^1$: the asymptotic slope for the Kempf-Ness function for $[0,1]$
is $1$ for every direction except that generated by $\diag(i,-i)$,
where it is $-1$; the topology induced by the Tits metric in this case
is discrete.
\end{remark} 

\subsection{The Hesselink stratification}
\label{hesselink}

Let $X \subset \P(V)$ be a projective $G$-variety, or more generally a
compact K\"ahler Hamiltonian $K$-manifold.  The Hesselink
stratification uses the weights appearing in the Hilbert-Mumford
criterion to construct a stratification on $X$: Define for any
non-zero $\lambda \in \k$ the {\em Hilbert-Mumford slope}
$$\mu_\lambda(x) = \lan \Phi(x_\lambda), \lambda \ran/ \Vert \lambda
\Vert .$$
By Corollary \ref{gradpsi}, the Hilbert-Mumford slope is equal to the
asymptotic slope of the Kempf-Ness function studied in Theorem
\ref{conelem}.

\begin{definition}  
A point $x \in X$ is
\begin{enumerate}
\item {\em slope semistable} iff $\mu_\lambda(x) \leq 0$ for all $\lambda$,
\item {\em slope stable} iff $\mu_\lambda(x) < 0$ for all $\lambda$, 
\item {\em slope unstable} iff $x$ is not semistable, and 
\item {\em slope polystable} iff it is slope semistable and
  $\mu_{\lambda}(x) = 0$ implies $\mu_{- \lambda}(x) =0$ for all
  $\lambda$.
\end{enumerate} 
\end{definition}  

Slope semistability might also be called Hilbert-Mumford
semistability, but this seems a little unwieldy.  We have already seen
in the proof of the Kempf-Ness theorem that slope semistability is
equivalent to semistability.  The equivalence of slope polystability
with polystability is proved in Mundet \cite{mun:poly}.  It follows
from Section \ref{polystable} that a point $x \in X$ is polystable but
not stable iff its Jordan-H\"older cone contains a line.

The set of destabilizing one-parameter subgroups is studied by
Hesselink in the algebraic case \cite{hess:uni}, \cite{hess:strat},
see also Ramanan-Ramanathan \cite{ram:stab}.  For any $\lambda$ we denote by $G_\lambda$
the centralizer of $\lambda$ and by $\C^*_\lambda$ the one-parameter
subgroup generated by $\lambda$.  Obviously $\C^*_\lambda \subset
G_\lambda$.  Let $x \in X$ and $Z_\lambda$ denote the component of
$X^\lambda$ containing $x_\lambda$. Then the action of $G_\lambda$ on
$Z_\lambda$ descends to an action of $G_\lambda/\C^*_\lambda$.
Furthermore, the inner product on $\k$ determines a splitting
$\g_\lambda = \C \lambda \oplus \g_\lambda/\C\lambda$ which defines a
lift of $G_\lambda/\C^*_\lambda$ to the polarizing line bundle, at
least up to finite cover.  So we may consider $Z_\lambda$ as a
polarized $G_\lambda/\C^*_\lambda$-variety, with the caveat that the
polarization depends on the choice of inner product on $\k$.

\begin{theorem} \label{Gvar} 
\label{hesselinkthm}
Any unstable $x$ has a unique (up to scalar multiple) {\em maximally
  destabilizing one-parameter subgroup} generated by $\lambda \in \k$
with the property that $x_\lambda$ is a semistable point for the
action of $G_\lambda/ \C^*_\lambda$ on $Z_\lambda$ and $\lambda$ is
{\em maximally destabilizing}: $ \mu_\nu(x) \leq \mu_\lambda(x)$ for
all $\nu \in \k - \{ 0 \}$ and equality holds iff $\R_+ \nu = \R_+
\lambda$.
\end{theorem} 

\begin{proof}
We already proved in Theorem \ref{conelem} the existence of a maximally
destabilizing one-parameter subgroup generated by some $\lambda \in
\k$. It remains to show that for any $x \in X$, $x_\lambda$ is
semistable for $G_\lambda/\C^*_\lambda$, or what is equivalent,
$\Phi((x_\lambda)_t)$ converges to $K \lambda$.  Since
$(\Phi(x_\lambda),\lambda) = (\lambda,\lambda)$, the function $\psi$
goes to $-\infty$ along $ \exp( - i \lambda t)x$ as fast as $x_t$.  If
$\lim \Vert \Phi((x_\lambda)_t) \Vert > \lambda$, then $\psi( (\exp( -
i \lambda t_1) x )_{t_2})$ as $t_1 \gg t_2 \to \infty$ goes to
$-\infty$ faster than $\psi(x_t)$, which contradicts convexity of
$\psi$ as in the proof of Theorem \ref{conelem}.  Hence $\Vert
\Phi((x_\lambda)_t) \Vert \to \lambda$.  Now the gradient trajectories
for $\exp( -i \lambda t)x$ converge to a broken gradient trajectory
for $x_\lambda$; since each piece in the broken gradient trajectory
must decrease $\phi$, the limit has only one piece.  That is,
$\Phi((x_\lambda)_t) \to K\lambda$.  \end{proof}

Let $\Lambda$ denote the set of equivalence classes of one-parameter
subgroups appearing in Hesselink's theorem (with equivalence given by
the adjoint action) we call the decomposition $X = \cup_{\lambda}
X_\lambda$ the {\em Hesselink stratification} of $X$.

\begin{remark} The Hesselink stratification is the finite-dimensional
analog of the Shatz stratification \cite{shatz:de} of the moduli stack
of vector bundles on a curve by the type of the Harder-Narasimhan
filtration.
\end{remark} 

\noindent The following is proved in the algebraic case by Ness
\cite{ne:st}.
\begin{theorem} \label{agree}  
The Hesselink and Kirwan-Ness stratifications agree.
\end{theorem} 
\begin{proof} By Theorem \ref{conelem} part (c).
\end{proof}  

\begin{proof}[Proof of Kirwan's theorem \ref{kirwan}]  Let 
$\psi = \psi_v$ be a Kempf-Ness function, and $X_\lambda$ a Hesselink
  stratum (or equivalently, a Kirwan-Ness stratum.)  Let $U_\lambda$
  denote the set of points in $x$ with direction of maximally negative
  slope $\lambda$.  Uniqueness of $\lambda$ implies that if $x \in
  U_\lambda$ and $g \in G$ is such that $gx \in U_\lambda$, then $g
  \in P_\lambda$.  Indeed, note $G = K P_\lambda$ and $U_\lambda$ is
  $P_\lambda$-stable. Hence it suffices to consider the case $g \in
  K$, and then $g \lambda$ is also a direction of maximal
  descent. Hence $g\lambda = \lambda$ which implies that $g \in
  K_\lambda$, hence $g \in P_\lambda$.  This implies $X_\lambda = G
  \times_{P_\lambda} U_\lambda$.  To see that $U_\lambda =
  Y_\lambda^{\ss}$ of Section \ref{kn}, note that $U_\lambda \subset
  Y_\lambda^{\ss}$ by Theorem \ref{Gvar}.  On the other hand,
  $Y_\lambda^{\ss}$ is contained in $U_\lambda$: any point in
  $Y_\lambda^{\ss}$ has a point in $\Phinv( \lambda)$ in its
  orbit-closure, and $ \Vert \lambda \Vert^2$ minimizes $ \Vert \Phi
  \Vert^2$ on $Y_\lambda^{\ss}$ which implies that $\lambda$ is the
  direction of maximally negative slope.  This completes the proof.
\end{proof} 

\begin{remark} Suppose $\omega \in \Omega^2(X)$ is a closed two form that 
is not symplectic, but satisfies $\omega (\xi_X,J\xi_X) > 0$ for any
$\xi \in \k$.  The proof above works equally well for moment maps
associated to such two-forms.  That is, only non-degeneracy of the
two-form on the directions generated by the action is used in the
proof.
\end{remark}

\section{Moment polytopes} 
\label{shiftquot}

According to work of Atiyah, Guillemin-Sternberg, and Kirwan, the
quotient of the image of the moment map is convex.  (This section
could have been placed before that on Schur-Horn convexity.)

\subsection{Convexity theorems for Hamiltonian actions} 

Let $X$ be a Hamiltonian $K$-manifold with moment map $\Phi$.  The
{\em moment image} of $X$ is $\Phi(X) \subset \k$.  The quotient
$$\Delta(X) : = \Phi(X)/K \subset \k^\dual/K$$ 
can be identified with a subset of the convex cone $\t_+^\dual \cong
\k^\dual/K$.

\begin{example}  If $X = \P^{n-1}$ and $G = U(1)^n$ acts by the standard
representation, then the moment image is the {\em standard
  $n$-simplex}
$$\Phi(X) = \{
(\nu_1,\ldots,\nu_n) \in \R_{\ge 0}^n  \ | \ \nu_1 + \ldots + \nu_n = 1
\} ,$$ 
see \eqref{projsp}.  The coordinate hyperplane $\{ z_j = 0\} \subset
X$ maps to the $j$-th facet $\{ \nu_j = 0 \} \subset \Phi(X)$.
\end{example} 

Another description of the moment polytope $\Delta(X)$ involves the
{\em shifted} symplectic quotients: for $\lambda \in \k^\dual$, the
quotient
$$ X \qu_\lambda K := \Phinv(K \lambda)/K = (\mO_\lambda^- \times X) \qu K$$ 
is the {\em symplectic quotient of $X$ at $\lambda$}.  The shifted
symplectic quotient is the classical analog of the multiplicity space
$\Hom_K(V_\lambda,V)$ of a representation $V$ in the following sense:

\begin{proposition} \label{shifted} 
Let $X$ be a polarized projective $G$-variety and $\lambda$ a dominant
weight.  Then $R(X \qu_\lambda G)_d = \Hom_G(V_{d\lambda}, R(X)_d) $
for any $ d \ge 0$.
\end{proposition}  

\begin{proof}  Combining the Borel-Weil and Kempf-Ness theorems gives
\begin{eqnarray*}
R(X \qu_\lambda K)_d &=& R(K \lambda^- \times X)_d^K \\
&=&  (V_{d\lambda}^\dual \otimes R(X)_d)^K \\
&=& \Hom_K(V_{d\lambda},   R(X)_d) \end{eqnarray*} 
\end{proof}  

The following is immediate from the definitions:

\begin{lemma} \label{nonem} $ \Delta(X) = \{ \lambda
\ | \ X \qu_\lambda K \neq \emptyset \}$ is the set of $\lambda$ for
which the shifted symplectic quotient $X \qu_\lambda K$ is non-empty.
\end{lemma} 

The set $\Delta(X)$ is the ``classical analog'' of the set of simple
modules appearing in a $G$-module.  Let $\Delta_\Q(X) :=
\Lambda_\Q^\dual \cap \Delta(X)$ denote the set of rational points in
$\Delta(X)$; furthermore $\Delta_\Q(X)$ is dense in $\Delta(X)$, see
for example \cite{le:co}.

\begin{theorem} $\Delta_\Q(X) = \Delta(X) \cap \Lambda^\dual_\Q$
  is equal to the set of points $\lambda/d$ such that $V_{\lambda}
  \subset R(X)_d$.
\end{theorem}    

\begin{proof} By Lemma \ref{nonem} and Proposition \ref{shifted}.
\end{proof}

Recall that a {\em convex polyhedron} is the intersection of
a finite number of half spaces, while a {\em convex polytope} is the
convex hull of a finite number of points.  The fundamental theorem of
convex geometry says that any {\em compact} convex polyhedron is a
convex polytope and vice-versa.

\begin{theorem} [Atiyah \cite{at:co}, Guillemin-Sternberg \cite{gu:co1} 
for the abelian case, Kirwan \cite{ki:con} for the non-abelian case]
  Let $K$ be a compact, connected Lie group and $X$ a compact
  connected Hamiltonian $K$-manifold.  Then $\Delta(X)$ is a convex
  polytope.  If $K$ is abelian, then $\Delta(X)$ is the convex hull of
  the image $\Phi(X^K)$ of the fixed point set $X^K$ of $K$.
\end{theorem}  

\noindent $\Delta(X)$ is the {\em moment polytope} of $X$.  The
arguments of Atiyah and Guillemin-Sternberg in \cite{at:co},
\cite{gu:co1} are Morse-theoretic. The equivariant version of
Darboux's theorem implies that the functions $\lan \Phi, \xi \ran$
have only critical sets of even index, and this implies that the level
sets $\lan \Phi, \xi \ran^{-1}(c)$ are connected.  Using an inductive
procedure one shows that for any subtorus $K_1 \subset K$, the level
sets of the moment map for $\Phi_1$ are connected as well.  Taking
$K_1$ of codimension one, this shows that the intersection of
$\Phi(X)$ with any rational line is connected and it follows that
$\Phi(X)$ is convex.  The reader is referred to the original papers
for details.  Kirwan's non-abelian version uses the Morse theory of
the norm-square of the moment map.  See
Lerman-Meinrenken-Tolman-Woodward \cite{le:co} for a derivation of
non-abelian convexity from the abelian case.

Brion \cite{br:im}, following earlier work of Mumford
\cite[Appendix]{ne:st}, pointed out the following proof of convexity,
which in language of geometric quantization would be called a
``quantum'' proof: Suppose $\lambda_j/d_j \in \Delta_\Q(X), j= 0,1$.
Let $v_j \in R(X)_{d_j}$ be corresponding highest weight vectors.
Then for any $n_0,n_1 \in \N$, $v_0^{n_0}v_1^{n_1} \in R(X)_{n_0 d_0 +
  n_1 d_1}$ is a highest weight vector, so
$$ \frac{n_0 \lambda_0 + n_1 \lambda_1}{n_0 d_0 + n_1 d_1} = \frac{d_0
  n_0}{d_0 n_0 + d_1 n_1} (\lambda_0/d_0) + \frac{d_1 n_1}{n_0 d_0 +
  n_1 d_1}(\lambda_1/d_1) \in \Delta_\Q(X) .$$
This implies that $\Delta_\Q(X)$ is convex.  

The inequalities of the previous section (for example, the
Horn-Klyachko problem) can now be seen as the inequalities describing
the moment polytopes of products of coadjoint orbits.

\subsection{Convexity theorems for orbit-closures} 

In the case that $X$ is K\"ahler, Atiyah \cite{at:co} also described
the images of orbit-closures under the moment map, in the case that
$K$ is abelian.  Of course if the orbit-closure is smooth, then this
falls under the assumptions of the previous convexity theorem, but
Atiyah's theorem also includes the case of singular orbit-closures:

\begin{theorem} \cite[Theorem 2]{at:mom}  \label{atiyah} Let $K$ be a torus, 
$G$ its complexification, and $X$ a K\"ahler Hamiltonian $K$-manifold.
  Let $Y \subset X$ be a $G$-orbit.  Then 
\begin{enumerate} 
\item $\Delta := \Phi(\ol{Y})$ is a convex polytope with vertices
  $\Phi(\ol{Y} \cap X^G)$;
\item For each open face $F \subset \Delta$, $\Phinv(F) \cap \ol{Y}$ is a
  single $G$-orbit.
\item $\Phi$ induces a homeomorphism of $\ol{Y}/G$ onto $\Delta$.
\end{enumerate} 
\end{theorem} 

We will describe Atiyah's arguments since they are brief and are
closely related to the one-parameter subgroups of Hesselink as well as
the Jordan-H\"older subgroups of Section \ref{jh}.  The proof depends
on the following

\begin{lemma} Let $Y \subset X$ be a $G$-orbit and $y \in Y$. Then  
\begin{enumerate} 
\item $y_\lambda = \lim_{t \to \infty}( \exp( it \lambda)y)$ exists
  and lies in the fixed point set $X^\lambda$;
\item $\lim_{t \to \infty} \lan \Phi( \exp(  it \lambda)y), \lambda
  \ran $ exists and is a constant $d_\lambda$ independent of $y$.
\item $d_\lambda = \sup_{y \in Y} \lan \Phi (y), \lambda \ran $ is the
  asymptotic slope in Theorem \ref{conelem}.
\end{enumerate}
\end{lemma} 
\noindent Suppose that $\lambda$ is generic so that $X^G = X^\lambda$.
The Lemma implies
$$  \sup_{y \in Y}  \lan \Phi(y), \lambda \ran  = \sup_{y \in X^G \cap \ol{Y}} 
\lan \Phi(y), \lambda \ran .$$
Hence $\Phi(Y)$ is contained in the convex hull of $\Phi(X^G \cap
\ol{Y})$.  To see that $\Phi(\ol{Y}) = \Delta$, Atiyah notes that for any
$y \in Y$ and direction $\xi \in \k$ of unit length, there exists a
time $t(\xi)$ such that 
$$\lan \Phi( \exp(i t(\xi) \xi y) , \xi \ran =
\hh ( \Phi(y) + d(\xi)) .$$ 
The set of points $\exp( i \xi)y$ with $\Vert \xi \Vert \leq t ( \xi/
\Vert \xi \Vert)$ defines a neighborhood $U$ of $y$ in $Y$ with $
\Phi(\ol{U}) = \Phi(y) + \hh (\Delta - \Phi(y))$; this immediately implies
that $\Phi(\ol{Y})$ is both open and closed in $\Delta$ and hence equal to
$\Delta$.

To prove the third part of the Theorem, Atiyah considers for any
$\lambda \in \k$ and fixed point component $Z \subset X^\lambda$, the
unstable manifold $Z^u$ consisting of all points that flow to $Z$
under $\exp( it \lambda)$. By the stable manifold theorem $Z^u$ is a
smooth manifold and the limit of the flow defines a smooth
$G$-equivariant projection $Z^u \to Z$.  In particular, if $Z$ is any
component of $X^\lambda$ containing a limit point of $Y$ then $Y
\subset Z^u$ and $\ol{Y} \cap Z$ is a single $G$-orbit. From this it
follows that $\Phi(Z \cap \ol{Y})$ is a face of $\Delta$ with fibers the
orbits of the compact torus $K$, see \cite[p. 10]{at:mom}, and this
completes the proof.

\begin{remark} 
Atiyah's theorem makes the theory of polystable points and
Jordan-H\"older vector described in Section \ref{polystable}
substantially easier in the abelian case.  One sees that the
``Jordan-H\"older'' cone of Theorem \ref{jh} is the dual cone to the
face of the polytope containing $0$, in the case that $Y$ is a
semistable orbit.
\end{remark}

Atiyah's convexity theorem for orbit-closures has been generalized to
Borel subgroups by Guillemin and Sjamaar \cite{gsj:co}.

\section{Multiplicity-free actions} 

In certain cases Hamiltonian or algebraic actions may be classified by
combinatorial data related to the moment map. In this section we
discuss an example of this, the {\em multiplicity-free case}, in which
the symplectic and git quotients are points.

\subsection{Toric varieties and Delzant's theorem} 

A {\em toric variety} is an irreducible normal $G$-variety $X$ such
that $G$ is an algebraic torus and $X$ contains an open $G$-orbit.
Affine toric varieties are naturally classified by monoids $M$ in the
group $\Lambda^\dual$ of weights of $G$, with the corresponding toric
variety given by $\Spec(\C[M])$.  Each such monoid spans a rational
cone in $\Lambda^\dual_\Q$, and defines a dual cone in $\Lambda_\Q$.
Toric varieties with trivial generic stabilizer are classified by {\em
  fans} in $\Lambda_\Q$, that is, collections of cones such that any
intersection of a cone is again a cone in the fan, see Oda
\cite{oda:con} or Fulton \cite{fu:in}.

\begin{example}  Suppose that $X = \P^2$ with action given by 
$(w_1,w_2)[z_0,z_1,z_2] = [z_0,w_1^{-1} z_1,w_2^{-1}z_2]$. There are
  seven orbits, given by non-vanishing of various coordinates, and in
  particular, three closed orbits $[1,0,0],[0,1,0],[0,0,1]$, whose
  cones are generated by pairs of vectors $(1,1),(-1,0)$,
  $(1,1),(0,-1)$, and $(0,-1),(0,-1)$.  The fan contains these three
  cones, and their intersections; this is the {\em dual fan} to the
  moment polytope $\hull((0,0),(1,0),(0,1))$.
\end{example} 

A Hamiltonian torus action is {\em multiplicity free} or {\em
  completely integrable} if all the symplectic quotients are points,
or equivalently, each fiber of the moment map is an orbit of the
torus.

\begin{example} The $U(1)^n$ action on $\P^n$ is multiplicity-free,
since the fibers of the moment map are given by $[z_0,\ldots, z_n]$
with $|z_1|,\ldots |z_n|$ fixed, which are orbits of $U(1)^n$.
\end{example} 
\noindent Multiplicity-free Hamiltonian torus actions are classified
by a theorem of Delzant.

\begin{definition} 
A polytope $\Delta \subset \k^\dual$ is called {\em Delzant} if the normal
cone at any vertex is generated by a basis of the weight lattice
$\Lambda^\dual \subset \k^\dual$.
\end{definition} 

\begin{theorem} [Delzant \cite{de:ha}]  There exists a one-to-one correspondence
between Delzant polytopes and multiplicity-free torus actions on
compact connected manifolds with trivial stabilizer, given by $X
\mapsto \Phi(X)$.  Any compact connected multiplicity-free Hamiltonian
torus action has the structure of a smooth projective toric variety.
\end{theorem} 

\noindent Note that any compatible complex structure is unique up to
isomorphism, but not up to K\"ahler isomorphism. That is, any toric
variety has many non-equivalent K\"ahler structures, see Guillemin
\cite{gu:kae}.  

There are ``local'' and ``local-to-global'' parts of the
proof; the local part follows from the equivariant Darboux theorem,
while the ``local-to-global'' part uses the vanishing of a certain
sheaf cohomology group over the polytope.

Existence of a smooth projective toric variety with a given polytope
follows from, for example, Lerman's method of {\em symplectic cutting}
\cite{le:sy2} which we now describe.  We begin with the simplest case,
when $X$ is a Hamiltonian $S^1$-manifold with moment map $\Phi: X \to
\R$.  The diagonal $S^1$-action on $X \times \C$ is Hamiltonian with
moment map
$$ \Phi_{X \times \C}: (x,z) \mapsto \Phi(x) - |z|^2/2 .$$
Its symplectic quotient at any value $\lambda$ 
$$ X_{\geq \lambda} := (X \times \C) \qu_\lambda S^1  $$
is called the {\em symplectic cut} of $X$ at $\lambda$ admits a
decomposition
$$ (X \times \C) \qu_\lambda S^1 
\cong X \qu_\lambda S^1 \cup (X \times \C^*) \qu_\lambda S^1 
\cong X \qu_\lambda S^1 \cup \Phinv((\lambda,\infty)) .$$
It follows from the definitions that the inclusion of
$\Phinv((\lambda,\infty))$ in $X_{\geq \lambda}$ is symplectic and so
$X_{\geq \lambda}$ is obtained by removing $\Phinv((-\infty,\lambda))$
and ``closing off'' the boundary by quotienting it by $S^1$.

More generally, suppose that $K$ is a torus, $\xi \in \k$ any rational
vector, and $\lambda \in \R$.  Let $U(1)_\lambda$ denote the
one-parameter subgroup generated by $\lambda$, with moment map $\lan
\Phi, \lambda \ran$.  Then the symplectic cut $X_{\geq \lambda} = (X
\times \C) \qu_\lambda U(1)_\lambda \cong X \qu_\lambda U(1)_\lambda
\cup \{ \lan \Phi,v \ran > \lambda \}$ admits the structure of a
Hamiltonian $K$-manifold with moment polytope $\Phi(X_{\ge \lambda}) =
\Phi(X) \cap \{ \lan \nu,v \ran \geq \lambda \}$.

\begin{example} Let $X = \P^2$ equipped with $U(1)^2$-action
given by with weights $(0,0), (-2,0),(0,-2)$.  The moment polytope is
then the convex hull of $(0,0),(2,0),(0,2)$.  Let $\lambda = (0,-1)$
so that the one-parameter subgroup generated by $\lambda$ acts with
moment map 
$$[z_0,z_1,z_2] \mapsto -2|z_1|^2/(|z_0|^2 + |z_1|^2 +
|z_2|^2) .$$ 
The symplectic cut at $-1$ is then a toric variety with polytope the
convex hull of $(0,0), (0,2), (1,0), (1,1)$, see Figure \ref{cut}.

\begin{figure}[h]
\includegraphics[height=1in]{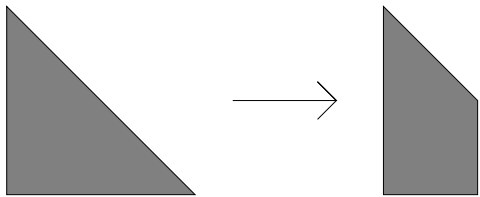}
\caption{Effect of cutting on a moment polytope}
\label{cut}\end{figure}
\end{example}

Suppose that $\Delta$ is a Delzant polytope defined by  inequalities
$$ \Delta = \{ \nu \in \k^\dual \ | \ \lan \nu, v_j \ran \geq \lambda_j, j =
1,\ldots, m \}$$
for some vectors $v_j \in \k$ and some constants $\lambda_j \in \R, j
= 1,\ldots, m$.  Let $X = T^\dual K$, with moment image $\k^\dual$ and
the standard K\"ahler structure.  Performing a symplectic cut for each
inequality gives a K\"ahler manifold with Hamiltonian $K$ action and
moment polytope $\Delta$.

Alternatively any smooth projective toric variety is a symplectic or
geometric invariant theory quotient of affine space $X = \C^m$.  There
is an explicit description of the semistable locus given by Audin
\cite{au:to} and Cox \cite{cox:hom}.

\subsection{Multiplicity-free actions and spherical varieties}

Let $K$ be a compact connected Lie group.  Recall that a $K$-module
$V$ is {\em multiplicity-free} iff $\Hom_K(V_\lambda,V)$ is dimension
at most one, for any simple $K$-module $V_\lambda$ iff $\End_K(V)$ is
abelian, using Schur's lemma.  The definition in part (a) of the
following was introduced in Guillemin-Sternberg \cite{gu:mf}:

\begin{theorem} (see 
\cite[Appendix]{wo:cl}) 
The following conditions are equivalent, and if they hold the action
is {\em multiplicity-free}:
\begin{enumerate} 
\item $C^\infty(X)^K$ is an abelian Poisson algebra.  
\item The symplectic quotient $X \qu_\lambda K: =  
\Phinv(K \lambda)/K$ is a point for all $\lambda$.  
\end{enumerate}
\end{theorem} 

\begin{proof} 
We denote by $r_\lambda: C^\infty(X)^K \to C^\infty(X \qu_\lambda K)$
the map of Poisson algebras induced by the symplectic quotient
construction, if $\lambda$ is free. In general, we define $C^\infty(X
\qu_\lambda K) := C^\infty(X)^K/\{ f, f| \Phinv(\lambda) = 0 \}$.  A
lemma of Arms, Cushman, and Gotay \cite{ar:un}, see Sjamaar-Lerman
\cite{sj:st}, says that this quotient is a non-degenerate Poisson
algebra, that is, the bracket vanishes only on constant functions.
Suppose (a).  Since $r_\lambda$ is surjective, $C^\infty(X \qu_\lambda
K)$ is abelian as well, and so $X \qu_\lambda K$ must be discrete,
hence a point by Kirwan's results.  Conversely, if all the reduced
spaces are points and $f,g \in C^\infty(X)^K$ then $r_\lambda(\{ f , g
\}) = 0$ for all $\lambda$ implies that $\{ f, g \} = 0$.
\end{proof}  

The complex analogs of multiplicity-free Hamiltonian actions are
called {\em spherical varieties}.  Let $G$ be a connected complex
reductive group.  For the following, see Brion-Luna-Vust \cite{br:es},
the review \cite{kn:lv}, or the second part of Brion's review in this
volume.

\begin{theorem}  The following conditions for a normal $G$-variety $X$ 
are equivalent; if they hold $X$ is called {\em spherical}:
\begin{enumerate}
\item   some (hence any) Borel subgroup $B$ has an open orbit; 
\item the space of rational functions $\C(X)$ is a multiplicity-free
$G$-module;
\item   some (hence any) Borel subgroup $B$ has finitely many orbits.
\end{enumerate} 
\end{theorem}  
\begin{remark}  For an arbitrary group action, 
existence of a dense orbit does not imply finitely many orbits.  For
example, consider the action of $SL(n,\C)$ on the space of $n \times
n$ matrices on the left: any two invertible matrices are related by an
element of $SL(n,\C)$, but there are infinitely many orbits of
degenerate matrices distinguished by their kernels.
\end{remark} 

The classification of toric varieties is generalized to spherical
varieties as a special case of a theorem of Luna-Vust \cite{lu:pl}
which gives a classification of spherical varieties by their generic
isotropy group and a {\em colored fan}, see the contribution of
Pezzini in this volume or Knop \cite{kn:lv}.  Each colored fan is a
collection of {\em colored cones}, convex cones in the space
$\Lambda_X$ dual to the space $\Lambda_X^\dual$ of characters
corresponding to $B$-semiinvariant functions $\C(X)^{(B)}$, together
with a finite set of $B$-stable divisors, satisfying certain
conditions.  The classification of generic isotropy groups that
appear, which are called {\em spherical subgroups}, is the subject of
an open conjecture of Luna, see the contribution of Bravi in this
volume.  The relation between multiplicity-free Hamiltonian actions
and spherical varieties is given by the following, which is a
consequence of the Kempf-Ness theorem:

\begin{proposition}  A smooth $G$-variety $X \subset \P(V)$
is spherical if and only if it is a multiplicity-free Hamiltonian
$K$-manifold.
\end{proposition} 

\begin{proof}  By Proposition \ref{shifted} $X \qu_\lambda K  = \pt$ iff   
$\Hom_G( V_\lambda, H^0(X,\mO_X(d)))$ is dimension one or zero for all
  $d \ge 0$.  This holds for all $\lambda$ and $d \ge 0$ iff $\C(X)_d$
  is a multiplicity-free $G$-module for all $d \ge 0$ iff $X$ is
  spherical.
\end{proof} 
\noindent In contrast to the toric case, not every multiplicity-free
Hamiltonian action admits the structure of a spherical variety
\cite{wo:mu}.

\subsection{Moment polytopes of spherical varieties} 

We have already seen several examples of the following: 

\begin{theorem}   
Let $X$ be a smooth polarized spherical $K$-variety with
moment polytope $\Delta$ and trivial generic stabilizer.  Then
$H^0(X,\mO_X(1))$ is the multiplicity-free $K$-module whose weights
are the integral points $\Delta \cap \k^\dual$ of $\Delta$.
\end{theorem} 

\begin{proof}  By Proposition \ref{shifted} and the fact that the symplectic
quotients are points.
\end{proof}  

The moment polytope of a spherical variety $X$ is described by a
finite set of linear inequalities corresponding to the $B$-stable
divisors of $X$ which was described by Brion \cite{br:gr} in the
non-abelian case.  Let $X$ be a spherical $G$-variety and $L \to X$ a
$G$-equivariant line bundle.  First some notation: Let $\C(X)$ denote
the space of rational functions on $X$, and $\C(X)^{(B)}$ the space of
$B$-semiinvariant vectors.  Let $\Lambda_X^\dual \subset
\Lambda^\dual$ denote the group of weights appearing in
$\C(X)^{(B)}$. Let $\D(X)$ denote the set of prime $B$-stable divisors
of $X$.  Each $D \in \D(X)$ defines a valuation $ \C(X)^{(B)} \to \Z$
and so a vector $v_D$ in the dual $\Lambda_X$ of $\Lambda_X^\dual$.
Let $\C(X,L)$ denote the space of rational sections of $L$, and $s \in
\C(X,L)^{(B)}$ with weight $\mu(s)$.  Let $n_D(s)$ denote the order of
vanishing of $s$ at $D$.  Consider the identification $\C(X)^{(B)} \to
\C(X,L)^{(B)}, f \mapsto fs$.  The section $fs$ is global iff $fs$
vanishes to at least zeroth order on each $D \in \D(X)$, iff $f$
vanishes at least to order $-n_D$.  Thus

\begin{proposition}  Let $X$ be a spherical $G$-variety, 
and $L \to X$ a $G$-line-bundle.  The space of weights for elements of
$\C(X,L)^{(B)}$ is
$$ \Delta(X,L) = \{ \mu \in \Lambda_X^\dual | v_D(\mu) \ge - n_D(s) \} +
\mu(s) .$$
\end{proposition} 

\begin{example} \label{macex}
Here is a typical application which appears in Brion \cite{br:gr} and
seems to be due to Macdonald \cite{ma:sy}:

\begin{theorem}  Let $V_\lambda$ be a simple $GL(r)$ module with highest
weight $\lambda = (\lambda_1 \ge \ldots \ge \lambda_r)$.  Then
$V_\lambda \otimes \on{Sym}(\C^r)$ admits a multiplicity-free
decomposition into simple modules $V_\mu$ with highest weights $\mu =
(\mu_1,\ldots, \mu_{r})$ satisfying
$$ \mu_1 \ge \lambda_1 \ge \mu_2 \ge \ldots \mu_r \ge \lambda_r . $$
\end{theorem} 

\begin{proof}   We prove only the case $r = 2$; the general case is similar. 
 $V_\lambda \otimes \on{Sym}(\C^2)$ is isomorphic to the space of
  holomorphic sections of the line bundle $\pi_1^* L_\lambda$ over $X
  = \P^1 \times \C^2 = \{ ([w_0,w_1],(z_0,z_1) \}$, where $\pi_1: \P^1
  \times \C^2 \to \P^1$ is projection on the first factor.  We take
  $B$ to be the subgroup of upper-triangular invertible matrices.  The
  $B$-invariant divisors are given by a single $G$-invariant divisor
  $D_1 = \{ (w,z) | z \in w \}$ and two $B$-stable divisors $D_2 = \{
  w = [1,0] \}$ and $D_3 = \{ z \in \C \oplus 0 \}$.  The space of
  singular vectors $ \C(X)^{(B)}$ is generated by $z_1 - w_1 z_0/w_0$
  and $z_0$ with highest weights $(0,1)$ resp. $(1,0)$.  The
  $B$-stable divisors are defined by $D_1 = \{ z_1/z_0 = w_1/w_0 \},
  D_2 = \{ w_1 = 0 \}, D_3 = \{ z_1 = 0 \}$ respectively.  Hence $z_1
  - w_1 z_0/w_0$ vanishes to order $1$ resp. $-1,0$ on $D_1$
  resp. $D_2,D_3$; $z_1$ vanishes to order $0$ resp. $0,1$ on $D_1$
  resp. $D_2,D_3$.  So $v_{D_1} = (0,1)$, $v_{D_2} = (0,-1)$, $v_{D_3}
  = (1,0)$. Taking $s$ to be the section of $\P^1$ with weight
  $(\lambda_1,\lambda_2)$, which vanishes to order $0$ on $D_1$,
  $\lambda_1 - \lambda_2$ in $D_2$, and $0$ on $D_3$ one obtains
  $n_{D_1} = 0, n_{D_2} = \lambda_1 - \lambda_2, n_{D_3} = 0$.  This
  yields the inequalities
$ \mu_2 \ge \lambda_2,\ -\mu_2 \ge -\lambda_2 - (\lambda_1 -
  \lambda_2) = -\lambda_1, \ \mu_1 \ge \lambda_1 $
as claimed. See Figure \ref{mac}.
\end{proof} 

\begin{figure} 
\begin{picture}(0,0)%
\includegraphics{mac.pstex}%
\end{picture}%
\setlength{\unitlength}{4144sp}%
\begingroup\makeatletter\ifx\SetFigFontNFSS\undefined%
\gdef\SetFigFontNFSS#1#2#3#4#5{%
  \reset@font\fontsize{#1}{#2pt}%
  \fontfamily{#3}\fontseries{#4}\fontshape{#5}%
  \selectfont}%
\fi\endgroup%
\begin{picture}(1881,1881)(2239,-2380)
\put(2966,-2080){\makebox(0,0)[lb]{$\lambda$}%
}
\end{picture}%

\caption{Decomposition of $V_\lambda \otimes \on{Sym}(\C^2)$
via Brion's method}
\label{mac}
\end{figure}

\begin{remark}   Not every $B$-stable divisor defines a facet
of the moment polytope.  This is already apparent in the case of the
Borel-Weil theorem, where for a group of rank $r$ there are $r$
$B$-stable divisors (the Schubert varieties of codimension one) but
the moment polytope is simply a point.
\end{remark}
\end{example}

Based on his work on the toric case, Delzant asked the question of
whether compact multiplicity-free actions are classified by their
moment polytopes and generic stabilizers, and answered the question
affirmatively in the rank two case \cite{de:cl}.  A result of Knop
\cite{knop:aut} reduces this to the question of whether affine
spherical varieties are classified by their moment polytopes and
generic stabilizers of the compact group actions; this conjecture has
recently been proved by Losev \cite{losev:knop}, see also his review
in this volume.

In the torus case we have 
\begin{corollary}  With $X,K,\mO_X(1)$ as above, if $K$ is a torus then
the dimension of $H^0(X,\mO_X(1))$ is the number of integral points
$\Delta \cap \k^\dual$ of $\Delta$.
\end{corollary} 
\noindent The dimension of $H^0(X,\mO_X(k))$ can be computed by Riemann-Roch for
sufficiently large $k$, since $\mO_X(1)$ is by assumption positive.
This led to an interesting series of papers on formulas for the number
of lattice points in a convex polytope which generalize the
Euler-Maclaurin formula and were later proved combinatorially, see
\cite{br:la} for references.

\section{Localization via sheaf cohomology} 

In this section we review various ``fixed point methods'' for
computing moment polytopes, in the context of sheaf cohomology.  These
include not only the ``localization'' methods which take as input
fixed point data for a one-parameter subgroup, but also the
``non-abelian localization'' principle which uses the Kirwan-Ness
stratification.

\subsection{Grothendieck's local cohomology}  

A powerful technique for computing cohomology groups, and therefore
for computing moment polytopes, is Grothendieck's local cohomology
theory, exposed in \cite{sga2} and Hartshorne \cite{ha:rd}.  Let $X$
be a $G$-variety and $Y \subset X$ a $G$-subvariety.  Let $E \to X$ be
a $G$-equivariant coherent sheaf.  Denote by $\Gamma_Y(X,E)$ the group
of sections whose support is contained in $Y$.  We denote by $H^i_Y$
the $i$-th derived functor of $\Gamma_Y$, so that the {\em local
  cohomology group} $H^i_Y(X,E)$ is a $G$-module.  These modules have
the following properties:

\begin{theorem} 
\begin{enumerate}  
\item (Long Exact Sequence) There is an exact triangle
$$ \ldots H_Y(X,E) \to H(X,E) \to H(X-Y, E|X-Y) \to \ldots $$
\item (Gysin isomorphism) Suppose $Y \subset X$ is smooth.  Then
$$H^j_Y(X,E) \cong H^{j - \codim(Y)}(Y, E|Y \otimes \Eul(N)^{-1}) $$
where $N$ is the normal bundle of $Y$ in $X$ and $\Eul(N)^{-1} :=
\det(N) \otimes \S(N)$ (this is an inverse of the $K$-theory Euler
class $\Eul(N) = \Lambda(N^\dual)$ although we do not discuss
$K$-theory here)
\item (Spectral sequence associated to a stratification) Let $X_1
  \subset X_2 \subset \ldots \subset X_m = X$ be a filtration of
  $X$. There is a spectral sequence
$$ \bigoplus_{i=1}^m H_{X_{i} - X_{i-1}}(X_i,E | X_i) \implies H(X,E)
  .$$
\end{enumerate}
\end{theorem} 

Let $\chi(X,E) = \bigoplus (-1)^i H^i(X,E)$ be the Euler
characteristic, considered as a virtual $G$-representation, and
$\chi_Y(X,E)$ the Euler characteristic of the local cohomology along
$Y$.  These will generally not be finite-dimensional, but rather in
our cases of interest the multiplicity of each simple module is
finite.  Thus the formula below holds in the completion of the
representation ring, as an immediate consequence of the spectral
sequence:

\begin{corollary}  Suppose that $X_1 \subset \ldots \subset X_m = X$
is a filtration of $X$ such that the differences $X_i - X_{i-1}$ are
smooth with normal bundles $N_i \to X_i - X_{i-1}$.  Then
\begin{equation} \label{chi} 
\chi(X,E) = \sum_i (-1)^{\codim(X_i- X_{i-1})} \chi(X_i - X_{i-1}, E
|_{X_i - X_{i-1}} \otimes \Eul(N_i)^{-1}) \end{equation}
if both sides are well-defined in the sense that the multiplicity of
any simple module is finite.
\end{corollary}  
\noindent This formula applies to various filtrations associated to group
actions to give ``localization'' formulae.  

\begin{example} (Weyl character formula and Borel-Weil-Bott, 
c.f. Atiyah-Bott \cite{at:le1}) Let $X = G/B^-$ and $E =
\mO_X(\lambda)$ so that if $\lambda$ is dominant then $H^0(X,E) =
V_\lambda$ by Borel-Weil \ref{BW}.  The Bruhat decomposition $X =
\cup_{w \in W} X_w$ gives a filtration $X_i = \cup_{w \in W, l(w) \ge
  i} X_w$.  Each cell $X_w$ fibers over $x_w = wB/B$ with fiber $X_w
\cong M_w := \b \cap \Ad(w) \b$.  The normal bundle $X_w$ has
restriction to $x_w$ given by $N_w = (\b/\b \cap \Ad(w)\b)^\dual$.
The formula \eqref{chi} gives
\begin{eqnarray*}
 \chi(X,\mO_X(\lambda)) &=& 
\bigoplus_{w \in W} (-1)^{l(w)} 
\chi( X_w, E | X_w
\otimes \S(N_w ) \otimes \det(N_w))  \\ 
&=& \bigoplus_{w \in W } 
 (-1)^{l(w)} 
\chi( x_w, E 
\otimes \S(N_w ) \otimes \det(N_w)\otimes \S(M_w^\dual)
| x_w)  \\ 
&=& 
\bigoplus_{w \in W} (-1)^{l(w)}   \C_{w \lambda} 
\otimes \S(\b^-) \otimes \C_{w \rho - \rho} \end{eqnarray*}
where $\rho$ is the half-sum of positive roots.  Thus its character is 
\begin{equation} \label{weyl}
  \sum_{w \in W} (-1)^{l(w)} \frac{ t^{w(\lambda +
    \rho) - \rho}} {\prod_{\alpha > 0} (1 - t^{- \alpha})} .\end{equation}
Thus if $\lambda$ is dominant then 
\begin{proposition} (Weyl character formula) The character of the action of $T$ on $V_\lambda$
is given by \eqref{weyl}.
\end{proposition}
In general, suppose that $w$ is such that $w(\lambda + \rho) - \rho$
is dominant.  From the spectral sequence we see that the only
contribution to $\chi(X,\mO_X(\lambda))$ comes from
$H^{l(w)}(X,\mO_X(\lambda))$, since $l(w) = \codim(X_w)$.  This is a
simple $G$-module of highest weight $w(\lambda + \rho) - \rho$, since
it has the same character as that of $V_{w(\lambda + \rho) - \rho}$ by
the Weyl character formula.  If no such $w$ exists, then the Fourier
expansion of the character vanishes on dominant weights and is
$W$-invariant and so $H(X,\mO_X(\lambda))$ is trivial.  Thus:

\begin{proposition}  (Borel-Weil-Bott \cite{bott:hom}) Let $X = G/B^-$.   
$H^j(X,\mO_X(\lambda)) \cong V_{w(\lambda + \rho) - \rho}$ if
  $w(\lambda + \rho) - \rho $ is dominant for some (unique) $w \in W$
  and $j = l(w)$, and is zero otherwise.
\end{proposition} 
\end{example}  

\subsection{One-parameter localization} 

The derivation of the Weyl character formula given in the previous
section generalizes to varieties with circle actions as follows.  Let
$X$ be a compact $G \times \C^*$-variety, and $X^{\C^*}$ its
$\C^*$-fixed point set.  Let $\cF$ be the set of components of
$X^{\C^*} = \{ x \in X | zx = x \ \forall z \in \C^* \}$.  For each $F
\in \cF$, define
$$ X_F := \{ x \in X | \lim_{z \to 0} zx \in F \} .$$
Let $N_F$ denote the normal bundle of $F$ in $X$. It admits a
decomposition $N_F = N_F^+ \oplus N_F^-$ into positive and negative
weight spaces for the $\C^*$-action.

\begin{proposition} (Bialynicki-Birula decomposition \cite{bi:so}) 
Suppose that $X$ is smooth. Then each $X_F$ is a smooth $G \times
\C^*$-stable subvariety, equipped with a morphism 
$\pi_F : X_F \to F, \quad x \mapsto \lim_{z \to 0} zx$
which induces on $X_F$ the structure of a vector bundle whose fibers
are isomorphic to the fibers of the normal bundle $N_F^+ \to F$ of $F$
in $X$.
\end{proposition} 

\noindent By filtering by the dimension of $N_F^+$, applying the
localization formula \eqref{chi}, and pushing forward with $\pi_F$ one
obtains

\begin{theorem} [Localization for one-parameter subgroups]  Let
$E \to X$ be any $G \times \C^*$-equivariant coherent sheaf.  Then 
$$ \chi(X,E) = \sum_{F \subset X^{\C^*}} \chi(F, E|F \otimes \S(N_F^{+,\dual})
  \otimes \S(N_F^{-}) \otimes \det(N_F^{-})) $$
in the completion of the representation ring of $G$.
\end{theorem} 

\noindent One could equally well choose the stratification for the inverted
$\C^*$-action, which would lead to the same formula with $N_F^+,N_F^-$
inverted.  In the equivariant cohomology literature such a choice of
direction is called a choice of {\em action chamber}, see Duistermaat
\cite{du:eq}.

The spectral sequence contains more information than the localization
formula, namely, information about the individual cohomology groups.
For example,

\begin{example}   Let $X = \P^2$ equipped with the $G = (\C^*)^2$
action by 
$(g_1,g_2)[z_0,z_1,z_2]  = [z_0,g_1^{-1}z_1,g_2^{-1}z_2] .$  
Then $H^0(X,\mO_X(d))$ is spanned by homogeneous polynomials of degree
$d$.  Its Euler characteristic has character
$$ (\chi(X,\mO_X(d)))(g) 
= \sum_{d_1 + d_2 \leq d, d_1,d_2 \ge 0 }
  g_1^{d_1} g_2^{d_2} .$$
One can also see this easily from the localization formula, which
gives (for the $\C^*$-action induced by the map $z \mapsto (z,z^2)$)
three fixed points with normal weights $(1,0),(0,1)$, resp.
$(-1,0),(-1,1)$ resp. $(1,-1), (0,-1)$ and so
\begin{multline}
 (\chi(X,\mO_X(d)))(g) = (1-g_1)^{-1}(1 - g_2)^{-1} - g_1^{d+1} (1-
g_1)^{-1}(1- g_1^{-1} g_2)^{-1} \\
+ g_2^{d+1} g_1^{-1} (1- g_1^{-1} g_2)^{-1} (1 -
g_2)^{-1} .\end{multline}
Now suppose that $X'$ is the blow-up of $X$ at $[1,0,0]$.  Let $\pi:X'
\to X$ denote the projection, $\mO_{X'}(d,e) = \pi^* \mO_X(d) \otimes
E^e$.  The action of $\C^*$ on $X'$ has four fixed points (the point
at $[1,0,0]$ is replaced by two fixed points in the exceptional
divisor with fiber weights $(e,0), (0,e)$).  Hence
\begin{multline} 
 (\chi(X',\mO_{X'}(d,e)))(g) = g_1^{e}(1-g_1)^{-1}(1 - g_1^{-1} g_2)^{-1}
  \\ - g_2^{e + 1}g_1^{-1} (1-g_1 g_2^{-1})^{-1}(1 - g_2)^{-1} - 
g_1^d
  (1- g_1)^{-1}(1- g_1^{-1} g_2)^{-1} \\ + g_2^d (1- g_1^{-1} g_2)^{-1}
  (1 - g_2)^{-1} .\end{multline}
Its Fourier transform is shown below in Figure \ref{fourier}. 
\begin{figure}[ht]
\includegraphics[height=2in]{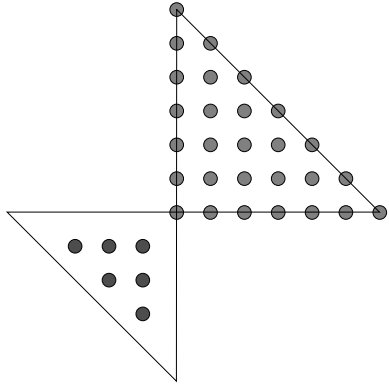}
\caption{Euler characteristic of a line bundle on blow-up of $\P^2$}
\label{fourier} 
\end{figure}
The contributions with weights $g_1^e$ contributes only to $H^0$,
while the contribution with weight $g_2^{e+1} g_1^{-1} $ contributes
only to $H^1$.  The former is the only term whose Fourier transform
has support in the larger triangle, while the latter is the only term
whose Fourier transform has support in the smaller.  Hence the dots in
the smaller triangle correspond to vectors in $H^1$ while those in the
larger correspond to $H^0$.  Very similar results are obtained by a
deformation method introduced by Witten \cite{witten:hm}, and studied
by a number of other authors since then, see for example
\cite{wu:ehm}.
\end{example} 

\subsection{Localization via orbit stratification} 

Other stratifications lead to interesting but less well-known
localization formulae.  For example, suppose that $G$ acts on $X$ with
only finitely many orbits $Y$.  We then obtain a formula
$$ \chi(X,E) = \sum_{Y \subset X} (-1)^{\codim(Y)} \chi(Y,E |Y \otimes 
\Eul(Y)^{-1})$$ 
assuming that each simple module appears with finite multiplicity as
before. In particular, suppose that $X$ is a toric variety and $E =
\mO_X(1)$ a polarization.  Indexing the orbits $Y_F$ by faces $F$ of
the moment polytope $\Delta$ we see that
$$\chi(Y,E | Y \otimes \Eul(Y)^{-1}) = \sum_{\mu \in \Lambda^\dual \cap
  C_F} t^\mu \det(N_F)$$
where the sum is over $\mu$ is the outward normal cone $C_F$ to
$\Delta$ at $F$, and $\det(N_F)$ is the determinant $N_F$ of the
normal bundle to $Y_F$.  This is closely related to the {\em
  Brianchon-Gram formula}: for any convex polytope $\Delta$,
$$ \chi_\Delta = \sum (-1)^{\codim(F)} \chi_{C_F} $$
where $\chi_{C_F}$ is the characteristic function of $C_F$
\cite{shep:gram}.

\subsection{Non-abelian localization} 

Let $X$ be a polarized smooth $G$-variety and $E \to X$ a
$G$-equivariant coherent sheaf.  Combining the Kirwan-Hesselink-Ness
stratification with the Euler characteristic formula \eqref{chi} gives
$$ \chi(X,E) = \sum_\lambda 
\chi(X_\lambda, E|_{X_\lambda} \otimes
\Eul(N_{X_\lambda})^{-1} ) $$
where the sum is over types $\lambda$ or equivalently critical sets
for the norm-square of the moment map.  This is a sheaf cohomology
version of a ``non-abelian localization principle'' suggested by
Witten in the setting of equivariant de Rham cohomology \cite{wi:tw}.
In fact, this terminology in the sheaf cohomology setting is somewhat
confusing: the formula is already quite interesting in the abelian
case (non-abelian should read ``not necessarily abelian'') and the
formula is not really a localization formula, since there is a
contribution from the (dense) open stratum.  Since $X_\lambda = G
\times_{P_\lambda} Y_\lambda^{\ss}$, we have
$$ \chi(X_\lambda, E|_{X_\lambda} \otimes \Eul(N_{X_\lambda})^{-1} ) = 
\Ind_{G_\lambda}^G 
\chi(Y_\lambda^{\ss}, E|_{Y_\lambda^{\ss}} 
\otimes \Eul(N_{X_\lambda} | Y_\lambda^{\ss})^{-1} ) .$$
(Here $\Ind$ denotes holomorphic induction, that is, if $V$ is a
$G_\lambda$-module then $\Ind_{G_\lambda}^G(V) = \chi( G
\times_{P_\lambda^-} V)$.  )  Since $Y_\lambda^{\ss}$ fibers over
$Z_\lambda^{\ss}$ with affine fibers,
\begin{multline}
 \chi(Y_\lambda^{\ss}, E|_{Y_\lambda^{\ss}} \otimes
\Eul(N_{X_\lambda} | Y_\lambda^{\ss})^{-1} ) \\ = \chi(Z_\lambda^{\ss},
E|_{Y_\lambda^{\ss}} \otimes \on{Sym}(N_{X_\lambda} X | Y_\lambda^{\ss})
\otimes \det(N_{X_\lambda} X  |_{Z_\lambda^{\ss}} )
\otimes
\on{Sym}(N_{Z_\lambda}^{\ss} Y_\lambda^{\ss})^\dual ) .\end{multline}
This can be put into a more understandable form if we recognize that
$N_{X_\lambda} X |_{Z_\lambda^{\ss}}$ resp.  $N_{Z_\lambda^{\ss}}
Y_\lambda^{\ss}$ is the positive resp. negative part of the normal
bundle of $Z_\lambda^{\ss}$ in $Y_\lambda^{\ss}$.  One obtains a
formula due to Teleman \cite{te:qu} in the algebraic case and Paradan
\cite{pa:lo} in the general symplectic setting; the latter proof uses
techniques of transversally elliptic operators:

\begin{theorem} \label{nal}
$$ \chi(X,E) = \sum_\lambda \Ind_{G_\lambda}^G( \chi(Z_\lambda^{\ss},
E|_{Z_\lambda^{\ss}} \otimes \Eul(N_{Z_\lambda^{\ss}} Y_\lambda^{\ss}
)^{-1}_+ )) $$
where the $+$ indicates the particular choice of (formal) inverse to
the $K$-theory Euler class given in the previous formula.  
\end{theorem}

\begin{example}  Let $X = \P^1$ and $E = \mO(d)$ so 
$\chi(X,E)$ has character $z^{-d} + z^{-d + 2} + \ldots + z^d$.  
The stratification $\P^1 = \{ 0 \} \cup \C^* \cup \{ \infty \}$ leads
to the formula 
$$ z^{-d} + \ldots + z^d = (\sum_{n \in \Z} z^{d + 2n}) - z^{d + 2}/ (1 - z^2) 
- z^{-d - 2}/ (1 - z^{-2}) .$$
\end{example} 

\begin{example}  We describe the non-abelian localization formula
for the action of $SL(3,\C)$ on a partial flag variety for the
exceptional group of type $G_2$, corresponding to the decomposition of
a simple $G_2$-module into $SL(3,\C)$-modules.  Let
$\omega_1,\omega_2$ denote the fundamental weights for $SL(3,\C)$.
The dual positive Weyl chamber for $G_2$ is the span of $\omega_1$ and
$\omega_1 + \omega_2$.  Let $P_{\omega_1 + \omega_2}$ denote the
maximal parabolic of $G_2$ corresponding to $\omega_1 + \omega_2$, and
$X = G_2/P^-_{\omega_1 + \omega_2}$, that is, the coadjoint orbit
through $\omega_1 + \omega_2$.  The action is spherical and moment
polytope the convex hull of $\omega_1,\omega_2,\omega_1 + \omega_2$.
We leave the computation of the moment polytope to the reader; it can
be computed using one-parameter localization.  By Borel-Weil and the
computation of the moment polytope,
$$ \chi(\mO_X(k)) = \sum_{\lambda \in k \Delta \cap Q} \chi_\lambda =
\Res^{G_2}_{SL(3,\C)}( \chi_{k(\omega_1 + \omega_2)} )
$$
the character of the irreducible $G_2$-representation with highest
weight $k(\omega_1 + \omega_2)$, restricted to $SL(3,\C)$; here $Q$ is
the lattice generated by the long roots shifted by $k(\omega_1 +
\omega_2)$.

We compute the Kirwan-Ness stratification as follows.  Let $F_1$ be
the open face connecting $\omega_2,\omega_1 + \omega_2$, $F_2$ the
open face connecting $\omega_1,\omega_1 + \omega_2$, and $F_3$ the
open face connecting $\omega_1,\omega_2$.  Let $F_{ij} = F_i \cap
F_j$.  The inverse image $\Phinv(F_{12})$ contains a unique point,
$x_1$, which is $T$-fixed.  None of the other $T$-fixed points map to
$\t_+^\dual$.  Therefore, the remaining points in
$\Phinv(\on{int}(\t_+^\dual))$ (the interior of the positive Weyl
chamber) have one-dimensional stabilizers.  Since
$\Phinv(\on{int}(\t_+^\dual))$ has dimension $ 2 \dim(T)$, it is a 
multiplicity free action, so the inverse image of any face $F \subset
\on{int}\t_+^\dual$ has infinitesimal stabilizer the annihilator of
the tangent space of $F$.  The stabilizers of the faces $F_1,F_2,F_3$
are
$$ \t_1 = \on{span}(h_1), \t_2 = \on{span}(h_2), \t_3 = \on{span}(h_3) $$
where $h_1,h_2,h_3$ are the coroots of $SL(3,\C)$.  The level set
$\Phinv((\omega_1 + \omega_2)/2)$ is a critical set of $\phi$ with
type $\lambda = ((\omega_1 + \omega_2)/2$.
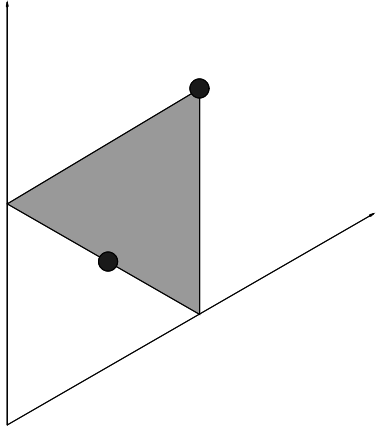
\begin{figure}[h]
\setlength{\unitlength}{0.00033333in}
\begingroup\makeatletter\ifx\SetFigFont\undefined%
\gdef\SetFigFont#1#2#3#4#5{%
  \reset@font\fontsize{#1}{#2pt}%
  \fontfamily{#3}\fontseries{#4}\fontshape{#5}%
  \selectfont}%
\fi\endgroup%
{\renewcommand{\dashlinestretch}{30}
\begin{picture}(5755,6639)(0,-10)
\blacken\path(42.000,6552.000)(27.000,6612.000)(12.000,6552.000)(27.000,6534.000)(42.000,6552.000)
\texture{ffffffff ffeeeeee eeffffff fffbfbfb fbffffff ffeeeeee eeffffff ffbfbbbf 
	bbffffff ffeeeeee eeffffff fffbfbfb fbffffff ffeeeeee eeffffff ffbfbfbf 
	bfffffff ffeeeeee eeffffff fffbfbfb fbffffff ffeeeeee eeffffff ffbfbbbf 
	bbffffff ffeeeeee eeffffff fffbfbfb fbffffff ffeeeeee eeffffff ffbfbfbf }
\path(27,6612)(27,12)
\path(27,6612)(27,12)
\blacken\path(5698.538,3269.010)(5743.000,3312.000)(5683.538,3294.991)(5675.449,3273.001)(5698.538,3269.010)
\path(5743,3312)(27,12)
\path(5743,3312)(27,12)
\texture{55888888 88555555 5522a222 a2555555 55888888 88555555 552a2a2a 2a555555 
	55888888 88555555 55a222a2 22555555 55888888 88555555 552a2a2a 2a555555 
	55888888 88555555 5522a222 a2555555 55888888 88555555 552a2a2a 2a555555 
	55888888 88555555 55a222a2 22555555 55888888 88555555 552a2a2a 2a555555 }
\shade\path(3027,1737)(3027,5262)(27,3462)(3027,1737)
\path(3027,1737)(3027,5262)(27,3462)(3027,1737)
\texture{ffffffff ffeeeeee eeffffff fffbfbfb fbffffff ffeeeeee eeffffff ffbfbbbf 
	bbffffff ffeeeeee eeffffff fffbfbfb fbffffff ffeeeeee eeffffff ffbfbfbf 
	bfffffff ffeeeeee eeffffff fffbfbfb fbffffff ffeeeeee eeffffff ffbfbbbf 
	bbffffff ffeeeeee eeffffff fffbfbfb fbffffff ffeeeeee eeffffff ffbfbfbf }
\put(1602,2562){\shade\ellipse{300}{300}}
\put(1602,2562){\ellipse{300}{300}}
\put(3027,5262){\shade\ellipse{300}{300}}
\put(3027,5262){\ellipse{300}{300}}
\end{picture}
}
\caption{Critical values of the norm-square of the moment map for $X =
  G_2/P_{\omega_1 + \omega_2}$ \label{g2exfig}}
\end{figure}
\noindent The fixed point component $Z_\xi$ has moment image $
\Phi(Z_\xi) = \hull(2\omega_2 - \omega_1, 2\omega_1 - \omega_2 ) .$
The unstable manifold $Y_\xi$ has image under the moment map for $T$
(that is, for the maximal torus of the compact group $SU(3)$)
$$ \pi^G_T \Phi(\ol{Y_\xi}) = \hull(2\omega_2 - \omega_1, 2\omega_1 -
\omega_2 , \omega_1 + \omega_2) .$$
None of the other facets $F_j$ contain points $\xi$ with $\xi \in
\t_j$.  Therefore, there are no other critical points of $\phi$ in $
\Phinv(\on{int}(\t_+^\dual))$.  Finally consider the inverse image of
the vertices $F_{13},F_{23}$.  Any $x \in \Phinv(F_{jk})$ has $G_x
\neq T$, hence $G_x$ cannot intersect the semisimple part
$[G_{\Phi(x)}, G_{\Phi(x)}]$.  Therefore, $G_x$ is one-dimensional.
let $Z_x$ denote the fixed point component of $G_x$ containing $x$.
Since $G_x$ is one-dimensional, the image $\Phi(Z)$ is codimension
one, and so meets $\Phinv(\on{int}(t_+^\dual))$.  But this implies
that the $\g_x$ is conjugate to either $\t_j$ or $\t_k$, and so $\g_x$
cannot equal the span of $F_{jk}$.  Therefore, set of types for the
action is $ \{ \omega_1 + \omega_2, \hh (\omega_1 + \omega_2) \} .$
(In fact the Kirwan-Ness stratification coincides with the orbit
stratification for $G_\C$. That is, $X$ is a two-orbit variety, with
one open orbit and one of complex codimension two
\cite{feld:twoorbit}.)

We now compute the contributions from the Kirwan-Ness strata.  For
$\xi = \omega_1 + \omega_2$, $Z_\xi^\ss$ is equal to a point, and the
bundle $N_\xi$ is the representation with weights $\beta_5,\beta_6$.
Hence
$$ \chi_{G_\xi}(Z_\xi^\ss ,E \otimes \Eul(N_{\xi})_+^{-1}) =
\sum_{(\lambda,\alpha_1) > k,(\lambda,\alpha_2) > k} z^\lambda .$$
Its induction to $G$ is
$$ \Ind_{G_\xi}^G  \chi_{G_\xi} (Z_\xi^\ss,E \otimes
\Eul(N_{\xi})_+^{-1}) = \sum_{(\lambda,\alpha_1) >
k,(\lambda,\alpha_2) > k} \chi_\lambda .$$
For $\xi = (\omega_1 + \omega_2)/2$, we have $Z_\xi^\ss \cong \C^*$
and $N_\xi$ trivial.  Therefore,
$$ \chi_{G_\xi} (Z_\xi^\ss, E \otimes
\Eul(N_{\xi})_+^{-1}) = \sum_{(\lambda,\xi) \ge k(\xi,\xi)}
z^{\lambda} $$
where the sum is over vectors $\lambda$ such that $\lambda -
k(\omega_1 + \omega_2)$ is in some lattice $\Lambda^\dual_1$, and
satisfying the inequality above.  Hence 
$$ \Ind_{G_\xi}^G( \chi_{G_\xi}( Z_\xi^\ss,E \otimes
\Eul(N_{\xi})_\xi^{-1}))  = \sum_{\lambda \in k\Delta} \chi_\lambda -
\sum_{(\lambda,\alpha_1) > k,(\lambda,\alpha_2) > k} \chi_\lambda .$$
Since the contributions from $\xi = (\omega_1 + \omega_2), \hh
(\omega_1 + \omega_2)$ must have finite sum, the lattice $\Lambda_1^\dual$
must be the long root lattice.  The contribution (for $k = 6$) is
shown in Figure \ref{g2ex2fig}.

\begin{figure}[h]
\setlength{\unitlength}{0.00033333in}
\begingroup\makeatletter\ifx\SetFigFont\undefined%
\gdef\SetFigFont#1#2#3#4#5{%
  \reset@font\fontsize{#1}{#2pt}%
  \fontfamily{#3}\fontseries{#4}\fontshape{#5}%
  \selectfont}%
\fi\endgroup%
{\renewcommand{\dashlinestretch}{30}
\begin{picture}(9095,11664)(0,-10)
\blacken\path(173.000,6552.000)(158.000,6612.000)(143.000,6552.000)(158.000,6534.000)(173.000,6552.000)
\texture{ffffffff ffeeeeee eeffffff fffbfbfb fbffffff ffeeeeee eeffffff ffbfbbbf 
	bbffffff ffeeeeee eeffffff fffbfbfb fbffffff ffeeeeee eeffffff ffbfbfbf 
	bfffffff ffeeeeee eeffffff fffbfbfb fbffffff ffeeeeee eeffffff ffbfbbbf 
	bbffffff ffeeeeee eeffffff fffbfbfb fbffffff ffeeeeee eeffffff ffbfbfbf }
\path(158,6612)(158,12)
\path(158,6612)(158,12)
\blacken\path(5829.538,3269.010)(5874.000,3312.000)(5814.538,3294.991)(5806.449,3273.001)(5829.538,3269.010)
\path(5874,3312)(158,12)
\path(5874,3312)(158,12)
\dashline{60.000}(3158,1662)(3158,11637)
\path(3188.000,11517.000)(3158.000,11637.000)(3128.000,11517.000)
\dashline{60.000}(158,3462)(9083,8862)
\path(8995.860,8774.213)(9083.000,8862.000)(8964.800,8825.548)
\dashline{60.000}(158,3462)(3158,1737)
\put(6008,10212){\shade\ellipse{300}{300}}
\put(6008,10212){\ellipse{300}{300}}
\put(4583,7512){\shade\ellipse{300}{300}}
\put(4583,7512){\ellipse{300}{300}}
\put(4583,11037){\shade\ellipse{300}{300}}
\put(4583,11037){\ellipse{300}{300}}
\put(7583,9312){\shade\ellipse{300}{300}}
\put(7583,9312){\ellipse{300}{300}}
\put(6083,8412){\shade\ellipse{300}{300}}
\put(6083,8412){\ellipse{300}{300}}
\put(4642,9304){\shade\ellipse{300}{300}}
\put(4642,9304){\ellipse{300}{300}}
\texture{44555555 55aaaaaa aa555555 55aaaaaa aa555555 55aaaaaa aa555555 55aaaaaa 
	aa555555 55aaaaaa aa555555 55aaaaaa aa555555 55aaaaaa aa555555 55aaaaaa 
	aa555555 55aaaaaa aa555555 55aaaaaa aa555555 55aaaaaa aa555555 55aaaaaa 
	aa555555 55aaaaaa aa555555 55aaaaaa aa555555 55aaaaaa aa555555 55aaaaaa }
\put(158,3462){\shade\ellipse{300}{300}}
\put(158,3462){\ellipse{300}{300}}
\put(1658,4362){\shade\ellipse{300}{300}}
\put(1658,4362){\ellipse{300}{300}}
\put(3158,5262){\shade\ellipse{300}{300}}
\put(3158,5262){\ellipse{300}{300}}
\put(3099,3470){\shade\ellipse{300}{300}}
\put(3099,3470){\ellipse{300}{300}}
\put(3158,1737){\shade\ellipse{300}{300}}
\put(3158,1737){\ellipse{300}{300}}
\put(1733,2562){\shade\ellipse{300}{300}}
\put(1733,2562){\ellipse{300}{300}}
\end{picture}
}
\caption{ $\Ind_T^G \chi_{Z_{(\omega_1 + \omega_2)/2}^\ss,T}(E)$
\label{g2ex2fig}}
\end{figure}
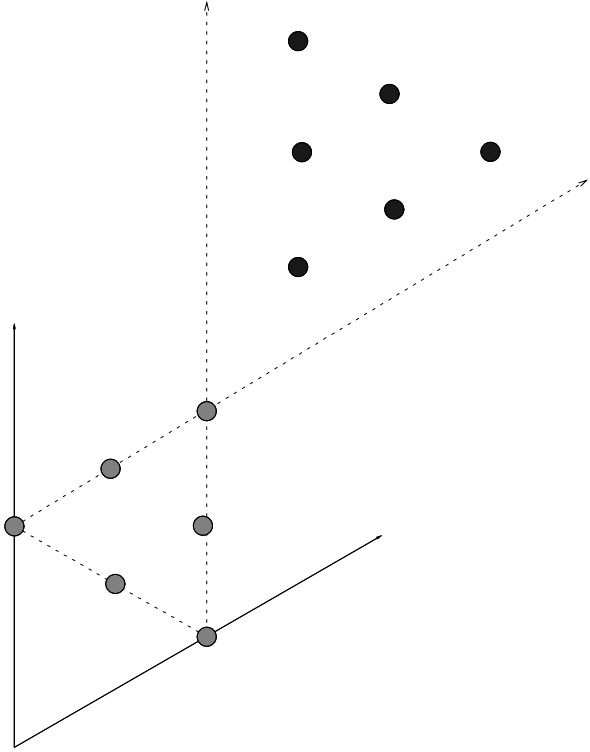

The positive contribution of the open stratum is finite ($6$
representations, for $k = 6$) and the negative contribution infinite,
that is $\dim(H^{odd}(M_\xi,L^k)) = \infty$, for any $k$.  One can
show that the higher cohomology lies in $H^1$, using the spectral
sequence.  The sum of the contributions is
$ \chi( \mO_X(k)) = \sum_{\lambda \in k \Delta} \chi_\lambda $
%
%
as claimed.  This completes the example.
\end{example}

Taking invariants in Theorem \ref{nal} gives a formula expressing the
difference between $\chi(X,E)^G$ and $\chi(X \qu G, E \qu G)$:

\begin{theorem} 
$$ \chi(X,E)^G - \chi(X \qu G, E \qu G) = \sum_{\lambda \neq 0}
  \chi(Z_\lambda^{\ss}, E|_{Z_\lambda^{\ss}} \otimes
  \Eul(N_{Z_\lambda^{\ss}} Y_\lambda^{\ss} )^{-1}_+ \otimes
  \Eul(\g/\p_\lambda^-))^{G_\lambda} $$
\end{theorem} 

The spectral sequence also contains information about the individual
cohomology groups.  For example, let $\C^*_\lambda \subset G_\lambda$
denote the one-parameter subgroup generated by $\lambda$.  The weight
of $\C^*_\lambda$ on $\det(N_{X_\lambda} X |_{Z_\lambda^{\ss}})$ is
positive, if $\lambda$ is non-trivial.  Indeed, $N_{X_\lambda} X
|_{Z_\lambda^{\ss}}$ is the negative part of the tangent
bundle. Furthermore, $\g/\p_\lambda^-$ has positive weights under
$\C^*_\lambda$.  Thus

\begin{corollary}  [Teleman \cite{te:qu}] 
Suppose that the weights of $\C^*_\lambda$ on $E | Z_\lambda$ are
positive for all types $\lambda$.  (This is automatically the case if
$E = \mO_X(d)$ is the $d$-th tensor product of a polarization
$\mO_X(1)$ of $X$).  Then $H^j(X,E)^G = H^j(X \qu G, E \qu G)$ for all
$j$.
\end{corollary} 

\noindent In particular, if the higher cohomology of $E$ vanishes then
so does that of $E \qu G$.  

The index maps naturally induce a diagram in $K$-theory
$$ \begin{diagram} \node{K_G(X)} \arrow[2]{e} \arrow{se} \node{} \node{K( X \qu G)} \arrow{sw} \\
\node{} \node{\Z} \end{diagram} $$
which fails to commute by the above explicit sum of fixed point
contributions for one-parameter subgroups.  There are similar results
in the equivariant cohomology of $X$ due to Paradan \cite{pa:mo} and
the author \cite{wo:norm}, based on earlier work of Witten
\cite{wi:tw}: a natural diagram of equivariant cohomology groups
$$ \begin{diagram} \node{H_G(X)} \arrow[2]{e} \arrow{se} \node{} \node{H( X \qu G)} \arrow{sw} \\
\node{} \node{\R} \end{diagram} $$
fails to commute by an explicit sum of fixed point contributions from
one-parameter subgroups.  The first explicit version of non-abelian
localization is due to Jeffrey-Kirwan \cite{je:lo1}, and expresses the
difference as a sum over certain fixed point sets of the maximal
torus.  The versions of Paradan, myself \cite{wo:norm}, and
Beasley-Witten \cite{bw:nal} express the difference as a sum over
critical points of the norm-square of the moment map.  The left hand
arrow in the diagram above takes some work to define: morally speaking
it is defined by $\alpha \mapsto \int_{X \times \g} \alpha$, but this
is not well-defined for polynomial equivariant classes.  Rather, the
left-hand side must be defined by a suitable limit procedure, either
by taking the leading term in Riemann-Roch, or (in the context of
equivariant de Rham cohomology with smooth coefficients) shifting by
equivariant Liouville form and taking the zero limit of the shift, see
\cite{wo:norm}.  From this point of view, the $K$-theory approach is
more natural.

\def\cprime{$'$} \def\cprime{$'$} \def\cprime{$'$} \def\cprime{$'$}
  \def\cprime{$'$} \def\cprime{$'$}
  \def\polhk#1{\setbox0=\hbox{#1}{\ooalign{\hidewidth
  \lower1.5ex\hbox{`}\hidewidth\crcr\unhbox0}}} \def\cprime{$'$}
  \def\cprime{$'$}

\end{document}